\documentclass[preprint]{imsart}
\RequirePackage[T1]{fontenc} 
\RequirePackage[utf8]{inputenc} 
\RequirePackage[english]{babel}
\RequirePackage{amsthm,amsmath,amsfonts,amssymb,mathrsfs,dsfont,mathtools}
\RequirePackage[numbers]{natbib}
\RequirePackage[colorlinks,citecolor=blue,urlcolor=blue]{hyperref}
\hypersetup{colorlinks,
            linkcolor=blue,
            anchorcolor=blue,
            citecolor=blue}
\usepackage[capitalize,nameinlink]{cleveref}

\usepackage{paralist}
\usepackage[shortlabels]{enumitem}
\usepackage{textcase} 
\usepackage{filemod} 

\usepackage[final]{showlabels}

\usepackage[norefs,nocites]{refcheck}
\startlocaldefs

\makeatletter
\newcommand{\refcheckize}[1]{%
  \expandafter\let\csname @@\string#1\endcsname#1%
  \expandafter\DeclareRobustCommand\csname relax\string#1\endcsname[1]{%
    \csname @@\string#1\endcsname{##1}\@for\@temp:=##1\do{\wrtusdrf{\@temp}\wrtusdrf{{\@temp}}}}%
  \expandafter\let\expandafter#1\csname relax\string#1\endcsname
}
\newcommand{\refcheckizetwo}[1]{%
  \expandafter\let\csname @@\string#1\endcsname#1%
  \expandafter\DeclareRobustCommand\csname relax\string#1\endcsname[2]{%
    \csname @@\string#1\endcsname{##1}{##2}\wrtusdrf{##1}\wrtusdrf{{##1}}\wrtusdrf{##2}\wrtusdrf{{##2}}}%
  \expandafter\let\expandafter#1\csname relax\string#1\endcsname
}
\makeatother

\refcheckize{\cref}
\refcheckize{\Cref}
\refcheckizetwo{\crefrange}
\refcheckizetwo{\Crefrange}
\newcommand{\ignore}[1]{}{}
\newcommand{\cla}{{\lambda_1}}
\newcommand{\clb}{{\lambda_2}}

\def\cite{\citet}

\numberwithin{equation}{section}

 \theoremstyle{plain}
 
 \newtheorem{theorem}{Theorem}[section]

 \newtheorem{lemma}[theorem]{Lemma}
 
 \newtheorem{corollary}[theorem]{Corollary}
 
 \newtheorem{proposition}[theorem]{Proposition}

 \theoremstyle{definition}
 \newtheorem{remark}{Remark}[section]
 
\newcommand{\IE}{\mathop{{}\mathbb{E}}\mathopen{}}
\newcommand{\IP}{\mathop{{}\mathbb{P}}\mathopen{}}
\newcommand{\Var}{\mathop{\mathrm{Var}}}

\newcommand{\bigo}{\mathop{{}\mathrm{O}}\mathopen{}}
\newcommand{\lito}{\mathop{{}\mathrm{o}}\mathopen{}}

\newcommand{\IR}{\mathbb R}
\newcommand{\IM}{\mathbb M}

\def\^#1{\relax\ifmmode {\mathaccent"705E #1} \else {\accent94 #1} \fi}
\def\~#1{\relax\ifmmode {\mathaccent"707E #1} \else {\accent"7E #1} \fi}
\def\*#1{\relax#1^\ast}
\edef\-#1{\relax\noexpand\ifmmode {\noexpand\bar{#1}} \noexpand\else \-#1\noexpand\fi}
\def\>#1{\vec{#1}}
\def\.#1{\dot{#1}}

\def\atop{\@@atop}

\renewcommand{\leq}{\leqslant}
\renewcommand{\geq}{\geqslant}
\renewcommand{\phi}{\varphi}

\newcommand{\D}{\Delta}

\newcommand{\eq}{\eqref}
\newcommand{\I}{\mathop{{}\mathrm{I}}\mathopen{}}

\newcommand\indep{\protect\mathpalette{\protect\@indep}{\perp}}
\def\@indep#1#2{\mathrel{\rlap{$#1#2$}\mkern2mu{#1#2}}}

\def\E{\IE}

\def\parsetime#1#2#3#4#5#6{#1#2:#3#4}
\def\parsedate#1:20#2#3#4#5#6#7#8+#9\empty{20#2#3-#4#5-#6#7 \parsetime #8}
\def\moddate{\expandafter\parsedate\pdffilemoddate{\jobname.tex}\empty}

\def\Var{{\rm Var}}

\def\I{\mathds{1}}

\def\arrowp{\stackrel{p}{\rightarrow }}

\def\seta{W\in A^{4\gamma+ \Delta_1 }\setminus A^{4\gamma - \bar\Delta_2}}
\def\setb{\|\xi_i\|\leq 4\gamma}
\def\R{\mathbb{R}}
\def\A{\mathcal{A}}

\def\mbar(s){\bar{M}(s)}
\def\A{\mathcal{A}}
\def\R{\mathbb{R}}

\def\inprod#1{\langle #1 \rangle}
\def\inprodb#1{\bigl\langle #1 \bigr\rangle}
\def\bangle#1{\bigl\langle #1 \bigr\rangle}
\DeclareMathOperator*{\argmin}{arg\,min}
\def\transpose{^{\intercal}}
\def\clc#1{\{ #1 \}}
\def\bclc#1{\bigl\{ #1 \bigr\}}
\def\bclr#1{\bigl( #1 \bigr)}

\newcommand{\ccc}{c_3}
\newcommand{\ccd}{c_4}
\newcommand{\cce}{c_5}

\newcommand{\myphi}{h}
\def\thestar{\theta^{*}}
\def\thebar{\bar\theta}
\def\thetaihat{\theta^{(i)}}
\newcommand{\tpstring}{\texorpdfstring}
\def\D{D_{\Theta}}
\def\tilde{\widetilde}
\def\epsilon{\varepsilon}

\crefname{equation}{}{}
\crefformat{equation}{\textup{(#2{#1}#3)}}
\crefrangeformat{equation}{\textup{(#3\textup{#1}#4)--(#5\textup{#2}#6)}}
\crefdefaultlabelformat{#2\textup{#1}#3}
\crefname{page}{p.}{pp.}
\crefname{subsection}{Subsection}{Subsections}

\newenvironment{equ}
{\begin{equation} \begin{aligned}}
{\end{aligned} \end{equation}}
\AfterEndEnvironment{equ}{\noindent\ignorespaces}

\newlist{condition}{enumerate}{10}
\setlist[condition]{label*=({M}\arabic*)}
\crefname{conditioni}{}{}
\Crefname{conditioni}{}{}

\newlist{conditionB}{enumerate}{10}
\setlist[conditionB]{label*=({B}\arabic*)}
\crefname{conditionBi}{}{}
\Crefname{conditionBi}{}{}

\newlist{conditionC}{enumerate}{10}
\setlist[conditionC]{label*=({C}\arabic*)}
\crefname{conditionCi}{}{}
\Crefname{conditionCi}{}{}

\endlocaldefs

\begin{document}

\begin{frontmatter}
\title{Berry--Esseen Bounds for Multivariate Nonlinear Statistics with Applications to M-estimators and
Stochastic Gradient Descent Algorithms}
\runtitle{Berry--Esseen bounds for multivariate nonlinear statistics}

\begin{aug}

\author[A,B]{\fnms{Qi-Man} \snm{Shao}\ead[label=e1]{shaoqm@sustech.edu.cn}\thanksref{t1}}
\and
\author[C]{\fnms{Zhuo-Song} \snm{Zhang}\ead[label=e3]{zszhang.stat@gmail.com}\thanksref{t2}}
\thankstext{t1}{Research partially supported by NSFC12031005 and Shenzhen Outstanding Talents Training Fund and also by Hong Kong RGC GRF 14302515 and 14304917.}
\thankstext{t2}{Corresponding author. Research supported by Singapore Ministry of Education Academic Research Fund MOE 2018-T2-076.}
\runauthor{Q.-M.~Shao and Z.-S.~Zhang}
\address[A]{
Department of Statistics and Data Scinece,
				Southern University of Science and Technology,
				Shenzhen, Guangdong, P.R. China. 
\printead{e1}}

\address[B]{
	   Department of Statistics,
              The Chinese University of Hong Kong,
              Shatin, N.T. Hong Kong 
}
\address[C]{
Department of Statistics and Applied Probability, 
		 National University of Singapore, 
		 Singapore 117546. \printead{e3}
}
\end{aug}

\begin{abstract}
\noindent
\quad
We establish a Berry--Esseen bound for general multivariate nonlinear statistics by developing a new  multivariate-type
randomized concentration inequality. The bound is the best possible for many known statistics.
As applications, Berry--Esseen bounds for  M-estimators and
averaged stochastic gradient descent
algorithms are obtained.

\end{abstract}



\end{frontmatter}

\section{Introduction}

Let $X_1, \dots, X_n$ be independent random variables taking values on $\mathcal{X}$ and let  $T
\coloneqq T(X_1, \dots, X_n)$ be a
general $d$-dimensional nonlinear statistic. In many cases the nonlinear statistic  can be written as
a linear statistic plus an error term:
\begin{equ}
	\label{eq:t}
	T = W + D,
\end{equ}
where
\begin{equ}
	\label{eq:1}
	W = \sum_{i = 1}^n \xi_i, \quad D \coloneqq D(X_1, \dots, X_n) = T - W,
\end{equ}
$\xi_i \coloneqq   h_i (X_i) \in \IR^d$ and $h_i  \colon \mathcal{X} \mapsto \IR^d$ is a Borel measurable
function.
Assume that
\begin{equ}
	\label{eq:con-xi}
	\IE \xi_i = 0 \text{ for each } 1 \leq i \leq n \text{ and }\sum_{i = 1}^{n} \IE \{\xi_i \xi_i\transpose\} = I_d.
\end{equ}
Let
\begin{align}
	\label{eq:gamma}
	\gamma:= \gamma_n = \sum_{i = 1}^n \IE \lVert \xi_i\rVert^3.
\end{align}
Since \(\xi_i\) is standardized, we remark that \(h_i = h_{n, i }\) and \( \xi_i = \xi_{n, i }\).
If $\lVert D\rVert \arrowp 0$ and $\gamma \to 0$ as $n \to \infty$, then, clearly, $T$ converges in
distribution to a
$d$-dimensional
standard normal distribution $N(0, I_d)$.

The aim of this paper is to provide a Berry--Esseen bound of the multivariate normal approximation
for the nonlinear statistic $T$.
The Berry--Esseen bound for multivariate normal approximation has been well studied in the past
decades.
For the linear statistic $W$, \cite{Ben03b,Ben05} used induction and Taylor's expansion to prove a Berry--Esseen bound of order
$d^{{1/4}}\gamma$, which is the best known result for the dependence on the dimension $d$.
We refer to \cite{Nag76,Sen81,Go91R,Bha10} and \cite{Rai19} for other results for independent random vectors.

In the case where $d = 1$, \citet*{Ch07N} proved a Berry--Esseen bound for $T$ using the
Berry--Esseen bound for $W$ and a randomized-type concentration
inequality approach:
\begin{align}\label{n4-i0}
  & \sup_{z\in \R} | \IP(T \leq z ) - \Phi(z) |  \leq 6.1 \gamma + \E |W D| +
  \sum_{i=1}^n  \E | \xi_i (D - D^{(i)} ) |  ,
\end{align}
where $D^{(i)}$ is any random variable such that $\xi_i$ is independent of $D^{(i)}$ and $\Phi$ is the standard normal distribution function.
\ignore{
The concentration inequality approach was firstly introduced to the Stein's method by
Charles Stein for independent and identically distributed (i.i.d.) random variables (see \cite{Ho78}) and later extended by \citet*{Ch98S}.
Let $\xi_1,\dots,\xi_n$ be independent random variables satisfying $\E \xi_i = 0$ for all
$1 \leq i \leq n$ and $\sum_{i=1}^n \E \xi_i^2  = 1$. Define $W = \sum_{i=1}^n \xi_i$ and $\gamma =
\sum_{i=1}^n \E |\xi_i|^3 $. \cite*{Ch98S} proved that for any $a < b$,
\[ \IP (a \leq W \leq b) \leq 3 (b-a)
+ 7 \gamma.  \]
More generally, \citet*{Ch07N} proved a randomized-type concentration inequality:
let $\Delta_1 : = \Delta_1  (\xi_1,\dots, \xi_n)$ and $\Delta_2 := \Delta_2(\xi_1,\dots,\xi_n)$   be
functions of $(\xi_1,\dots,\xi_n)$ such that $\Delta_1 \leq \Delta_2$,
then
\begin{equ}
	\label{eq:1d.con.inequality}
  \IP (\Delta_1 \leq W \leq \Delta_2) \leq C \biggl( \gamma + \E |W (\Delta_2 - \Delta_1)| +
  \sum_{i=1}^n \sum_{j = 1}^2  \E | \xi_i (\Delta_j - \Delta_j^{(i)} ) |  \biggr), 
\end{equ}
where $\Delta_1^{(i)}$ and $\Delta_2^{(i)}$ are any random variables such that  $\xi_i$ is independent of $(\Delta_1^{(i)}, \Delta_2^{(i)} )$ for each $1 \leq i \leq n$.
Furthermore, they obtained the Berry-Esseen bound for non-linear random variables. Later on,
\cite{Sh16C} used Stein's method to obtain a refined version of \cref{eq:1d.con.inequality} where they replaced
$\IE  \lvert W \Delta\rvert$ by $\IE  \lvert \Delta \rvert$.
}
For the Berry--Esseen bound for multivariate normal approximation, \cite{Che15b} proved a
concentration inequality for $d$-dimensional exchangeable pairs.
We also refer to
\cite{Ba90S,Go91R,Go96M,Ch08M,Re09M,Bha10,Ch11N0,Che15b} and \cite{Rai19}
for the development of Stein's method for multivariate normal approximations.


The main purpose of this paper is to prove a Berry--Esseen bound for nonlinear multivariate
statistics by developing a new randomized multivariate concentration inequality which generalizes the results of \cite{Ch07N}
and \cite{Che15b}.
Our main result can be applied to a large class of non-linear statistics, including M-estimators and
averaged stochastic gradient descent estimators.

Throughout  this paper, we  use the following notations.
Let $d \geq 1$ and  $x=(x_1,\dots,x_d)$ be a vector in $\R^{d}$.
For $x,y\in \IR^d$, denote by $\langle x, y\rangle$ the inner product of $x$ and $y$.
Let $\lVert x\rVert = \sqrt{ \langle x, x\rangle }$ be the $l_2$-norm of $x$.
For a $d\times d$ matrix $A$, and let $\lambda_{\min}(A)$ and $\lambda_{\max}(A)$ be the minimal and
maximal eigenvalue of $A$, respectively. Denote by $A\transpose$ the transpose of $A$ and by $\lVert
A\rVert$ the spectral norm, i.e.,
$
\lVert A \rVert := (\lambda_{\max}(A{\transpose} A)){}^{1/2}.
$
\ignore{For symmetric matrices $A$ and $B$, denote by $A \preccurlyeq (\text{resp.\,}\succcurlyeq) B$ if $A - B$
is non-positive (resp.\ non-negative) definite.}
Let $I_d$ be the $d$-dimensional identity matrix.
For $X \in \IR$ (resp. $\IR^d$) and $p \geq 1$, let
$\lVert X\rVert_p = (\IE \{|X|^p\})^{1/p}$ (resp. $(\IE \{ \lVert X\rVert^p\})^{1/p}$) be the
$L_p$-norm of $X$.

The rest of this paper is organized as follows. In \cref{n4:se2}, we present the Berry--Esseen bound of the multivariate normal approximation for $T$. In
\cref{sec:app}, we apply our main result to M-estimators and averaged stochastic gradient descent
algorithms. In \cref{sec:proof1}, we present a randomized concentration inequality for multivariate linear
statistics and give the proof of the main result. The proofs of theorems in \cref{sec:app} are postponed
to \cref{sec:proof2}.

\section{Main results}
\label{n4:se2}
Let $(X_1, \dots, X_n), (\xi_1, \dots, \xi_n), W, T $ and $D$ be defined as in \cref{eq:t,eq:1}.
Let $\A$ be the collection of all convex sets in
$\R^{d}$. Let $Z \sim N(0, I_d)$.
The following theorem provides a Berry--Esseen bound for $T$.
\begin{theorem} \label{thm2.1}
	Assume that \cref{eq:con-xi} is satisfied.
	Then,
    \begin{equ}\label{n4-t3-1}
    \sup_{A\in \A} \bigl\lvert \IP(T\in A)-\IP(Z\in A)\bigr\rvert  & \leq   259
	d^{1/2} \gamma + 2 \E \bigl\{ \|W\|  \Delta \bigr\} + 2 \sum_{i=1}^n \E
	\bigl\{\|\xi_i\| \lvert \Delta - \Delta^{(i)}\rvert\bigr\},
  \end{equ}
  for any random variables \(\Delta\) and \( (\Delta^{(i)})_{1 \leq i \leq n}\) such that \(\Delta \geq \lVert D \rVert\) and \(\Delta^{(i)}\) is independent of \(X_i\), where \(\gamma\) is as defined in \cref{eq:gamma}. \\
  \label{n4-thm3}
\end{theorem}

\begin{remark}
	The choices of $\Delta$ and  $\Delta^{(i)}$ are flexible. For example, let $(X_1', \dots, X_n')$ be an independent
	copy of $(X_1, \dots, X_n)$, one may choose  $\Delta = \lVert D\rVert$ and  $\Delta^{(i)} =
		\lVert D^{(i)} \rVert$, where
	\begin{math}
		D^{(i)} = D (X_1, \dots, X_{i - 1}, X_i', X_{i + 1}, \dots, X_n).
	\end{math}
	One can also choose $	D^{(i)} = D (X_1, \dots, X_{i - 1}, 0 , X_{i + 1}, \dots, X_n)$.
Moreover, the last term in \cref{n4-t3-1} cannot be removed, and we refer to \cite[Section 4]{Ch07N} for a
	counterexample.
\end{remark}

\begin{remark}
	For  $d = 1$, the right hand side of \cref{n4-t3-1} reduces to
	\begin{align*}
		259 \gamma + 2 \E  \{|W|   \Delta \}  +  2 \sum_{i=1}^n \E   \bigl\lvert
				\xi_i (\Delta -
		\Delta^{(i)}) \bigr\rvert,
	\end{align*}
	which differs from \cref{n4-i0} up to a constant factor.

	The Berry--Esseen bound \cref{n4-t3-1} provides an optimal order in terms of $n$ for many applications.  However, the order in $d$ may not be optimal in  \cref{n4-t3-1}. For a linear
	statistic $W$, \cite{Ben05} proved that
	\begin{align*}
		\sup_{A \in \mathcal{A}} \bigl\lvert  \IP (W \in A) - \IP (Z \in A )\bigr\rvert \leq C
		d^{{1/4}} \gamma,
	\end{align*}
	where $C > 0$ is an absolute constant and
	$d^{1/4}$ is believed to be the best possible. Here, $C > 0$ is an absolute constant, and \cite{Rai19} recently
	obtained a bound  with an explicit constant $42 d^{{1/4}} + 16$ by
	using Stein's method.  However, it is not clear how to obtain the order $d^{{1/4}}$ in our result.

\end{remark}

Using the technique of truncation, we obtain the following corollary, which may be useful for applications.
\begin{corollary}
	\label{cor:01}
	Let $O$ be a measurable set and $\Delta$ be a random variable such that $\Delta \geq \lVert D
	\rVert \I(O)$. Under the conditions of \cref{n4-thm3}, we have
	\begin{align*}
		\sup_{A \in \mathcal{A}} \bigl\lvert \IP (T \in A) - \IP (Z \in A) \bigr\rvert
		& \leq 259 d^{ {1/2}} \gamma + 2 \E \bigl\{\|W\| \Delta \bigr\} + 2 \sum_{i=1}^n
		\E \bigl\{\|\xi_i\| \lvert \Delta - \Delta^{(i)}\rvert\bigr\} + \IP (O^c),
	\end{align*}
	where $\Delta^{(i)}$ is any measurable random variable that is independent
	of $X_i$.
\end{corollary}

Condition \eq{eq:con-xi} can be extended to a general case. We have the following corollary.
\begin{corollary}
	\label{cor:2.3}
	Let \( T, W, D\) and \( (\xi_1, \dots, \xi_n) \) be defined as in \cref{eq:t,eq:1}.
	Assume that \( (\xi_1, \dots, \xi_n )\) satisfies:
	\begin{align*}
		\IE \{ \xi_i \} = 0 \text{ for \(1 \leq i \leq n\) and }\sum_{i=1}^n \IE \{ \xi_i \xi_i\transpose \} = \Sigma,
	\end{align*}
	where \(\Sigma\) is a positive definite matrix with \(\lambda_{\min } (\Sigma ) \geq \sigma >0\).
Then
    \begin{multline*}
    \sup_{A\in \A} \bigl\lvert \IP(T\in A)-\IP(\Sigma^{1/2} Z\in A)\bigr\rvert   \leq   259
	\sigma^{-3/2 }d^{1/2} \gamma + 2 \sigma^{-1 }\E \bigl\{ \|W\|  \Delta \bigr\}  \\
																				 \quad + 2 \sigma^{-1 }\sum_{i=1}^n \E
	\bigl\{\|\xi_i\| \lvert \Delta - \Delta^{(i)}\rvert\bigr\},
  \end{multline*}
  for any random variables \(\Delta\) and \( (\Delta^{(i)})_{1 \leq i \leq n}\) such that \(\Delta \geq \lVert D \rVert\) and \(\Delta^{(i)}\) is independent of \(X_i\), where \(\gamma\) is as defined in \cref{eq:gamma}. 
\end{corollary}

\section{Applications}
\label{sec:app}
In this section, we apply \cref{n4-thm3} to M-estimators and stochastic gradient
descent algorithms.

\subsection{M-estimators}

Let $X, X_1, \dots, X_n$ be i.i.d.\ random variables with common probability distribution
$P$ that take values in a measurable space  $(\mathcal{X}, \mathcal{B} (\mathcal{X}))$.
For any function $f: \mathcal{X} \mapsto \mathbb{R}$,
let
\begin{equ}
	\label{eq:m-p}
	\mathbb{P}_n f = \frac{1}{n} \sum_{i = 1}^{n} f(X_i), \quad
	P f = \int_{\mathcal{X}} f(x) P(d x), \quad \mathbb{G}_n f = \sqrt{n} (\mathbb{P}_n - P) f.
\end{equ}
Let  $\Theta \subset \mathbb{R}^d$ be a parameter space. For each  $\theta \in \Theta$, let
$m_{\theta}(\cdot): \mathcal{X} \mapsto \mathbb{R}$ be twice differentiable with respect to $\theta$,
and write
\begin{equ}
	\label{eq:map}
	\IM_n(\theta) = {\IP}_n m_{\theta}, \quad  M (\theta) = P m_{\theta}.
\end{equ}
Following the
notations in \cite{Vaa98},
we briefly write
\begin{equ}
	\label{eq:m-mdot}
	\dot{m}_{\theta} (x)= {\nabla_{\theta}} m_{\theta} (x), \quad  \ddot{m}_{\theta}
	(x)= \nabla_{\theta}^2 m_{\theta}(x ),
\end{equ}
where $\nabla_{\theta} m_{\theta}(x)$ is the gradient with respect to $\theta$.
Let
\begin{equ}
	\label{eq:mstar-true}
	\thestar = \argmin_{\theta \in \Theta } M (\theta)
\end{equ}
and
we say $\hat\theta_n$ is an \textit{M-estimator} of $\thestar$ if 
\begin{equ}
	\label{eq:m-estimator}
	\hat\theta_n = \argmin_{\theta \in \Theta } \IM_n (\theta).
\end{equ}
For any $p\geq 1$ and $Y \in \IR^d$, let
$ \lVert Y\rVert_p = (\IE \{ \lVert Y\rVert^p\})^{1/p}$   be the $L_p$-norm of  $Y$.

The asymptotic properties for M-estimators have been well studied in the literature, and we refer to \cite{van96,Vaa98}
and the references therein for a thorough reference. Under some regularity conditions, one has $\hat\theta_n
\arrowp
\thestar$, and \cite{Pol85}  showed that $\sqrt{n}(\hat\theta_n -
\thestar)$ converges weakly to a $d$-dimensional normal distribution. The convergence rate was also
studied by many authors, for instance, \cite{Pf71T,Pfa72,Pfa73} proved a Berry--Esseen bound of order
$O(n^{-1/2})$ for the  minimum contrast estimates under some regularity conditions.

In this subsection, we provide a Berry--Esseen bound for $\sqrt{n} (\hat \theta_n - \thestar)$ under
some convexity conditions, which are different from those in \cite{Pfa72}.
For symmetric matrices $A$ and $B$, denote by $A \preccurlyeq (\text{resp.\,}\succcurlyeq) B$ if $A - B$
is non-positive (resp.\ non-negative) definite.
We first propose the following two assumptions.
\begin{condition}
\item \label{con:a1}
	The function $m_{\theta}(\cdot)$ is twice differentiable with respect to $\theta$ and there exist constants
	$\mu > 0,  c_1 > 0, c_2 > 0$ and two nonnegative functions $m_1,m_2 : \mathcal{X} \mapsto \IR$ with
$\lVert m_1(X)\rVert_9 \leq c_1$ and $\lVert m_2(X)\rVert_4 \leq c_2$,  such that for any
	$\theta \in \Theta$,
	\begin{gather}
		M(\theta) - M(\thestar) \geq \mu \lVert \theta - \thestar\rVert^2, \label{eq:m-A11}\\
		\lvert m_{\theta}(x) - m_{\thestar} (x)\rvert \leq m_1(x) \lVert \theta - \thestar\rVert,
		\quad \forall\, x \in \mathcal{X}, \label{eq:m-A12}
		\intertext{and}
		\lVert \ddot{m}_{\theta}(x) - \ddot{m}_{\thestar}(x) \rVert \leq m_2(x) \lVert \theta -
		\thestar\rVert, \quad  \forall\, x \in \mathcal{X}. \label{eq:m-A13}
	\end{gather}
	{Moreover, there exists a constant $c_3 \geq 0$ and a nonnegative function $m_3 \colon \mathcal{X} \mapsto \IR$
		such that for any $x\in \mathcal{X}$,
		\begin{equ}
			\label{eq:m-A14}
			\ddot{m}_{\thestar} (x) \preccurlyeq m_3(x) I_d \mbox{ and $\lVert m_3(X)\rVert_4 \leq c_3$.}
		\end{equ}
		
	}

\item \label{con:a2}
	Let $\xi_i = \dot{m}_{\thestar} (X_i) \coloneqq (\xi_{i, 1}, \dots, \xi_{i, d})\transpose$, $\Sigma = \IE \left\{ \xi_i \xi_i\transpose
	\right\}$ and $V = \IE \{ \ddot{m}_{\thestar} (X) \}$.
	Assume that there exist  constants $\lambda_1 > 0$ and  $\lambda_2 > 0$ such that
	\begin{math}
		\lambda_{\min}(\Sigma) \geq \lambda_1 \text{ and } \lambda_{\min}(V) \geq \lambda_2.
	\end{math}
	Moreover, assume that there exists a constant $c_4 > 0$ such that
	\begin{equ}
		\label{eq:m-A22}
		\lVert \xi_1\rVert_4 \leq c_4 d^{1/2}.
	\end{equ}

\end{condition}

The following theorem provides a Berry--Esseen bound for the M-estimators.

\begin{theorem}
	\label{thm:mt1}
	Let \(\thestar\) and \(\hat\theta_n\) be defined as in \cref{eq:mstar-true,eq:m-estimator}.
	Under the conditions \cref{con:a1,con:a2}, we have
	\begin{align*}
		\sup_{A \in \mathcal{A}} \Bigl\lvert \IP \bigl( n^{1/2} \Sigma^{-1/2}V (\hat\theta_n -
		\thestar) \in A \bigr) - \IP (Z \in A)\Bigr\rvert \leq C d^{9/4} n^{ - {1/2}},
	\end{align*}
	where $C > 0$ is a constant depending only on  $c_1, c_2, c_3, c_4, \mu, \lambda_1$ and $\lambda_2$.

\end{theorem}
\begin{remark}
	The assumptions \cref{con:a1,con:a2} are neater than those in \cite{Pfa72}.
	Moreover, \cref{thm:mt1} provides a Berry--Esseen bound
	with the dependence on the dimension.
\end{remark}

\begin{remark}
	Based on the proof of \cref{thm:mt1}, if we further assume that $|m_2(X_i)|\leq c_2$ for each $1
	\leq i \leq n$ almost surely, then
	the assumption for $m_1(x)$ can be replaced by
	\begin{math}
		\lVert m_1(X) \rVert_5 \leq c_1.
	\end{math}
	The condition \cref{eq:m-A22} is satisfied if
	\begin{math}
		\lVert \xi_{ij} \rVert_4 \leq c_4
	\end{math}
	for all $1 \leq i \leq n$ and $1 \leq j \leq d$.
\end{remark}

\begin{remark}
	The  twice differentiability of $m_{\theta}(x)$ holds for many applications. However,  in
	general, $\ddot{m}_{\theta}(x)$ does not necessarily exist. We will discuss this case in the
	next subsection.
\end{remark}


When \(m_{\theta} (\cdot)\) is smooth in \(\theta\), one can compute \(\hat\theta_n\) by solving the score equation
\begin{align*}
	\IP_n \.m_{\theta} = \frac{1}{n} \sum_{i = 1 }^{n} \.m_{\theta} (X_i)  = 0.
\end{align*}
More generally, we can consider the estimating equations of the following type.
Let  $\Theta \subset \IR^d$ be the parameter space and for each $\theta \in \Theta$,  let $\myphi_{\theta}:
\mathcal{X} \mapsto \IR^d$, and let
\begin{align*}
	\Psi_n(\theta) = \frac{1}{n} \sum_{i = 1}^{n} \myphi_{\theta}(X_i), \quad  \Psi(\theta) = \IE \{
	\myphi_{\theta} (X)\}.
\end{align*}
Let $\hat\theta_n$ and  $\thestar$ satisfy
\begin{equ}
	\label{eq:z-theta}
	\Psi_n(\hat\theta_n) = 0, \quad  \Psi(\thestar) = 0.
\end{equ}
The estimator \(\hat\theta_n\) in \cref{eq:z-theta} is often called a \emph{Z-estimator} of \(\thestar\), see e.g., \cite{Vaa98}.
However, although there is no maximization in \cref{eq:z-theta}, the estimator $\hat\theta_n$ is also called an M-estimator of $\thestar$. 
Assume that $\Psi(\theta)$ is differentiable at  $\thestar$ and there exists a $d \times d$ matrix $\.\Psi_0$ satisfying
\begin{align*}
	\Psi (\theta) - \Psi(\thestar) - \.\Psi_{0} (\theta - \thestar) = \lito( \lVert \theta -
	\thestar\rVert) \quad \text{as $\theta \to \thestar$.}
\end{align*}
Under some regularity conditions and the so called ``asymptotic equi-continuity'' condition,
\cite[]{Hub67}  proved that $\sqrt{n} (\hat\theta_n - \thestar)$ converges in distribution to
$\.\Psi_0^{-1} Z$, where  $Z \sim N(0, \IE \{\myphi_{\thestar}(X_i) \myphi_{\thestar} (X_i)\transpose\})$.
\citet*{Ben97f} proved a Berry--Esseen bound of order $\bigo (n^{-1/2})$ for the $1$-dimensional case under some convexity conditions, and \cite{Pau96a} proved a convergence
rate result for the $d$-dimensional case under some \emph{smooth stochastic
differentiability} conditions, which are different from the conditions
\cref{con:b1,con:b2,con:b3,con:b4,con:b5} below.

Let $p \geq 3$ be a fixed number,
and we make the following assumptions.
\begin{conditionB}
\item \label{con:b1}
	There exist  positive constants  $\mu, c_1$ and $\lambda_1$ and a positive definite matrix $\.\Psi_0$
	such that
	\begin{equ}
		\label{eq:z-B11}
		\bigl\langle \Psi(\theta_1) - \Psi(\theta_2), \theta_1 - \theta_2\bigr\rangle \geq \mu
		\lVert \theta_1 - \theta_2\rVert^2,
	\end{equ}
	and
	\begin{equ}
		\label{eq:z-B12}
		\lVert \Psi(\theta) - \Psi(\thestar) - \.\Psi_0 (\theta - \thestar)\rVert \leq c_1 \lVert
		\theta - \thestar\rVert^2, \quad  \lambda_{\min}(\.\Psi_0) \geq \lambda_1.
	\end{equ}
\item \label{con:b2} Let $h_{\theta,j}$ be the  $j$-th element of  $h_{\theta}.$ There exists a
	function $h_0 \colon
	\mathcal{X} \mapsto \IR_+$ and a constant  $ c_2 > 0$ such that for any $\theta, \theta' \in
	\Theta$,
	\begin{equ}
		\label{eq:z-B21}
		\bigl\lvert h_{\theta, j}(X) - h_{\theta', j} (X)\bigr\rvert \leq h_0 (X) \lVert \theta -
		\theta'\rVert.
	\end{equ}
	and
	\begin{equ}
		\label{eq:z-B22}
		\lVert h_0(X)\rVert_p \leq c_2.
	\end{equ}
\item\label{con:b3} Let $\xi_i = \myphi_{\thestar}(X_i) $ and $\Sigma = \IE \{\xi_i \xi_i\transpose\}$. Assume that there
	exist positive constants  $\ccc$ and $\lambda_2$ such that
	\begin{gather}
		\lambda_{\min}(\Sigma) \geq \lambda_2, 	\label{eq:z-B31}
		\intertext{and}
		\lVert \xi_1\rVert_p \leq \ccc d^{1/2}. \label{eq:z-B32}
	\end{gather}
\end{conditionB}
\begin{remark}
	Following notations in Theorem 3.1, we can choose $h_{\theta}(x) = \dot{m}_{\theta}(x)$. Note
	that the assumption \cref{con:b1} is weaker than \cref{con:a1} in the sense of the differentiability of $\myphi_{\theta}$, because we assume that the differentiability only holds for
	$\Psi(\theta)$ rather than $h_{\theta}(x)$.

\end{remark}

\begin{theorem}
	\label{thm:z-1}
	Let \(\hat\theta_n\) and \(\thestar\) be defined as in \cref{eq:z-theta}.
	Let $p \geq 3$ and  $D_{\Theta} \coloneqq \sup_{\theta_1, \theta_2 \in \Theta}
	\lVert \theta_1 - \theta_2\rVert$, the diameter of the parameter space $\Theta$.
	Assume that  conditions \cref{con:b1,con:b2,con:b3} are satisfied. Then,
	\begin{equ}
		\label{eq:thmz-1}
		\sup_{A \in \mathcal{A}} \bigl\lvert \IP \bigl(\sqrt{n} \Sigma^{-1/2} \.\Psi_0 (\hat\theta_n - \thestar) \in A\bigr) - \IP \bigl(Z \in A\bigr)\bigr\rvert \leq C  (\D +1 )^2
				d^{{7/2}}  n^{- {1/2} + \epsilon_p}.
	\end{equ}
	where $\epsilon_p = 1/(2p-2)$ and $C > 0$ is a constant depending on  $p,
	c_1,c_2,c_3, \lambda_1, \lambda_2$ and $\mu$.
\end{theorem}

\begin{remark}
	Under some different conditions and assuming that $\IE \lVert \xi_i\rVert^3$ is
	bounded,
	\cite[Theorem 9]{Pau96a} proved a bound of order $n^{-{1/4}} (\log
	n)^{{3/4}}$.
	In \cref{thm:z-1} with $p = 3$, the result \cref{eq:thmz-1} reduces to
	$D_{\Theta}^2 d^{{7/2}} n^{- {1/4}}$, which is of a sharper order than \cite{Pau96a}.
	Moreover, \cref{thm:z-1} provides a result with the dependence on the dimension $d$.
\end{remark}

The order $n^{-1/2 + \epsilon_p}$ can be improved to $n^{-1/2}\log n$ under some stronger
conditions. Let us introduce the so-called {\it Orlicz norm}, one may refer to \cite[Section 2.2]{van96} for more details.
Let $\psi: [0, \infty) \mapsto [0, \infty)$ be a nondecreasing, convex function with  $\psi(0) = 0$.
Let $Y$ be a $\IR^d$-valued
random variable, 
and define the Orlicz norm  of \(Y\) with respect to \(\psi\) to be
\begin{align}
	\label{eq:psi-norm-1}
	\lVert Y \rVert_{\psi} = \inf \Bigl\{ C > 0 \, : \, \IE \Bigl\{\psi \Bigl(\frac{\|Y\|}{C}\Bigr) \Bigr\}
	\leq 1 \Bigr\}.
\end{align}
Specially, if we choose $\psi(x) = x^p$ for $p \geq 1$, then the corresponding Orlicz norm is simply the $L_p$-norm.
Let  $\psi_1 (x) \coloneqq e^{x} - 1$.
Now we propose the following assumptions.


\begin{conditionB}
\setcounter{conditionBi}{3}
	\item \label{con:b4}
		The condition \cref{eq:z-B22} in \cref{con:b2} is replaced by
		\begin{equ}
			\label{eq:z-B42}
			\lVert h_{0}(X) \rVert_{\psi_1} \leq \ccd.
		\end{equ}
		where $\ccd > 0$ is a constant.
	\item \label{con:b5}  The condition \cref{eq:z-B32} in \cref{con:b3} is replaced by
		\begin{equ}
			\label{eq:z-B51}
			\lVert \xi_{1 }\rVert_{\psi_1} \leq \cce,
		\end{equ}
		where $\cce > 0$ is a constant.
\end{conditionB}
\begin{remark}
Let \(Y\) be a random variable. It can be shown (see \cite[(5.14)--(5.16)]{Ver10} for example) that, there exist positive constants \(K_1, K_2, K_3\) that differ from each other by at most an absolute constant factor such that the following are equivalent:
\begin{enumerate}[(a)]
	\item \( \lVert Y \rVert_{\psi_1} \leq K_1\);
	\item \( \IP (|Y| \geq t) \leq \exp \{1 -t/K_2\}\) for all \(t \geq 0\);
	\item \( \lVert Y \rVert_p \leq K_3 p \) for all \(p \geq 1\).
\end{enumerate}

\end{remark}
We have the following theorem.
\begin{theorem}
	\label{thm:z-2}
	Let \(\hat\theta_n, \thestar\) and \(D_{\Theta}\) be defined as in \cref{thm:z-1}.
	Under the assumptions \cref{con:b1,con:b4,con:b5},
	\begin{align*}
		& \sup_{A \in \mathcal{A}} \bigl\lvert \IP \bigl(\sqrt{n} \Sigma^{-1/2} \.\Psi_0
		(\hat\theta_n - \thestar) \in A\bigr) - \IP \bigl(Z \in A\bigr)\bigr\rvert	\leq C  (\D +1)^2 d^4
		n^{- {1/2}} \log n,
	\end{align*}
	where $C > 0$ is a constant depending on  $c_1,c_4,c_5, \lambda_1, \lambda_2$ and $\mu$.

\end{theorem}


\subsection{Averaged stochastic gradient descent algorithms}
Consider the problem of searching for the minimum point $\thestar$ of a smooth
function $f(\theta), \theta\in \Theta \subset \IR^d$.  The stochastic gradient
descent method provides a direct way to solve the minimization problem. In this subsection, we consider the averaged stochastic
gradient descent algorithm, which is proposed by \cite{Pol90} and \cite{Rup88}. The algorithm is
given as follows: Let $\theta_0\in \IR^d$ be the initial value (might be
random), and for  $n
\geq 1$, we update  $\theta_n$ by
\begin{equ}
	\label{eq:saa}
	\theta_n & = \theta_{n - 1} - \ell_n \bigl( \nabla f (\theta_{n - 1}) + \zeta_n \bigr), \\
	\thebar_n & = \frac{1}{n} \sum_{i = 0}^{n-1} \theta_i.
\end{equ}
where 
$\ell_n>0$ is the so called \emph{learning rate} and
$(\zeta_1, \zeta_2, \dots)$ is a sequence of \(\IR^d\)-valued martingale differences. The convergence rate of $\IE \lVert \theta_n - \thestar\rVert^2$ and $\IE \lVert \bar\theta_n -
\thestar\rVert^2$ was thoroughly studied in the literature, see \cite{Pol90} and \cite{Bac11}. The normality of
$\sqrt{n}(\bar\theta_n - \thestar)$ is also well-known, see \cite{Pol92a}.
Suppose that the learning rate $\ell_n = \ell_0 n^{-\alpha}$ where $\alpha\in
(1/2, 1)$, under some regularity conditions,
\citet*{Pol92a} proved that $\sqrt{n} (\thebar_n - \theta^*)$ converges weakly to a multivariate normal distribution. Recently, \cite{Ana19a} used Stein's method and the techniques of martingales to
prove a convergence rate for a class of smooth test functions, see \cite[Theorem 4]{Ana19a} for more details.

In this subsection, we provide a Berry--Esseen bound for the normal approximation for $\sqrt{n} (\thebar_n - \thestar )$.

We make the following assumptions:
\begin{conditionC}
	\setcounter{conditionCi}{-1}
\item \label{con:c0}
	There exists a constant $\tau_0 >0 $ such that  $\lVert \theta_0 -
	\thestar\rVert_4 \leq
	\tau_0.$
\item \label{con:c1} The sequence $(\zeta_1, \zeta_2, \dots)$ is independent of \(\theta_0\), and  for each
	$n \geq 1$, $\zeta_n$ admits the decomposition
	\begin{align*}
		\zeta_n = \xi_n + \eta_n,
	\end{align*}
	where
	\begin{enumerate}[(i).]
		\item $(\xi_1, \xi_2, \dots)$ is a sequence of independent random variables and  $\IE \left\{
			\xi_i \right\} = 0$ and  $\IE \bigl\{\xi_i \xi_i\transpose\bigr\} = \Sigma_i$; there exist
			positive numbers $\lambda_{1}$ and $\lambda_{2}$ such that for any $i \geq 1$,
			$\lambda_{1} \leq \lambda_{\min}(\Sigma_i) \leq \lambda_{\max}(\Sigma_i) \leq
			\lambda_{2}$;
			moreover, there exists a positive number $\tau $ such that
			\begin{align*}
				\max_{1 \leq i\leq n} \lVert \xi_i\rVert_4\leq \tau;
			\end{align*}
	\item
		let $\mathcal{F}_0 = \sigma \{\theta_0\}$, and for each $n \geq 0$,  $\mathcal{F}_n = \sigma\{\theta_0,
		\xi_1, \ldots, \xi_n\}$;
		let \(g(\cdot, \cdot) : \IR^d \times \IR^d \mapsto \IR^d\), and
	the random variable $\eta_n:= g (\theta_{n - 1}, \xi_n)$ satisfies $\IE \{\eta_n \vert \mathcal{F}_{n - 1}\} = 0$ and
	for any $\theta$ and $\theta'$, there exists a nonnegative number $c_1 \geq 0$ such that
		\begin{equ}
			\label{eq:etan}
			\lVert g(\theta, \xi) - g(\theta', \xi) \rVert \leq c_1 \lVert \theta -
			\theta'\rVert \quad  \text{and} \quad  g(\thestar, \xi) = 0 \quad \text{for } \xi \in \mathbb{R}^d.
		\end{equ}
	\end{enumerate}

\item \label{con:c2} The function $f$ is {\it $L$-smooth} and {\it strongly convex}  with
	convexity constant $\mu > 0$, i.e., $f$ is twice differentiable and there exist two constants \(\mu >0\) and \(L>0\) such that
	\begin{align}
		\label{eq:str_con}
		\mu I_d \preccurlyeq \nabla^2 f(\theta) \preccurlyeq L I_d, \text{  for all  $\theta\in \Theta $}.
	\end{align}
\item\label{con:c3}  There exist positive constants  $c_2$ and $\beta$ such that for all $\theta$ with  $\lVert
	\theta - \thestar\rVert \leq \beta$,
	\begin{align}
		\label{eq:hes}
		 \bigl\lVert \nabla^2 f(\theta ) - \nabla^2
			f(\thestar)\bigr\rVert \leq c_2 \lVert
		\theta - \thestar\rVert.
\end{align}
\end{conditionC}


Let \(G := \nabla^2 f(\thestar)\).
Recall that $(\ell_n)_{n \geq 1}$ is the learning rate sequence in  \cref{eq:saa}, and let
\begin{align*}
	Q_i = \ell_i \prod_{j = i}^{n - 1}\prod_{k = i + 1}^{j} (I_d - \ell_k G).
\end{align*}
Here, for any $n \geq 0$, set
\begin{math}
	\prod_{i = n + 1}^n A_i = I_d, \prod_{i = n + 1}^n a_i = 1,
\end{math}
where $(A_i)_{i \geq 1}$ is a $\IR^{d \times d}$-valued sequence and  $(a_i)_{i \geq 1}$ is a
$\IR$-valued sequence.
Let
\begin{align*}
	\Sigma_n = \frac{1}{n} \sum^{n - 1}_{i=1} Q_i \Sigma_{i} Q_i\transpose.
\end{align*}

We have the following theorem.

\begin{theorem}
	\label{thm:sgd}
	Let $\ell_n = \ell_0 n^{-\alpha}$ where \(\ell_0 >0\) and  $1/2 < \alpha \leq 1$. Under the assumptions
	\cref{con:c0,con:c1,con:c2,con:c3}, 	
	we have
	\begin{enumerate}
		\item [\textup{(1)}]
		if $\alpha \in (1/2,1)$,
		\begin{equ}
			\label{eq:sgd-a}
				 \MoveEqLeft\sup_{A \in \mathcal{A}} \Bigl\lvert \IP \bigl(\sqrt{n} \Sigma_n^{-1/2} (\thebar_n
				 - \thestar) \in A\bigr) - \IP (Z \in A)\Bigr\rvert\\
				& \leq C  \bigl(d^{3/2} + \tau^3 +
			\tau_0^3\bigr) (d^{1/2} n^{-1/2} + n^{-\alpha + 1/2});
		\end{equ}
		\item [\textup{(2)}]
	 if $\ell_n = \ell_0 n^{-1}$ with $\ell_0 \mu \geq 1$, we have
		\begin{equ}
			\MoveEqLeft\sup_{A \in \mathcal{A}} \Bigl\lvert \IP \bigl(\sqrt{n} \Sigma_n^{-1/2} (\thebar_n -\thestar) \in A\bigr) - \IP (Z \in A)\Bigr\rvert \\
			& \leq C n^{-1/2} (d^{3/2} + \tau^3 + \tau_0^3) \times
			\begin{cases}
				\label{eq:sgd-b}
				d^{1/2} + \log n, & \ell_0 \mu > 1; \\
				d^{1/2} (\log n)^3 
				, & \ell_0 \mu = 1.
			\end{cases}
		\end{equ}
\end{enumerate}
	Here, $C>0$ is a constant depending only on  $\ell_0, \lambda_1, \lambda_2, c_1,
c_2, \alpha, \beta, L $ and $\mu$ and independent of $d, \tau$ and  $\tau_0$.
\end{theorem}
\begin{remark}
	Typically, $\tau \sim \tau_0 \sim d^{1/2}$.
	Specially, if $\alpha = 1 - \epsilon$ with an arbitrary $0 < \epsilon < 1/2$,
	then the RHS of \cref{eq:sgd-a} reduces to $C (d^2 n^{-1/2} +  d^{3/2} n^{-1/2 +
	\epsilon} ) $.
	If $\alpha = 1$ with $\ell_0 \mu \geq 1$, the Berry--Esseen bound \cref{eq:sgd-b} is of an optimal order up to a
	polynomial of a $(\log n)^3$ factor.


\end{remark}

\begin{remark}
	For $\alpha = 1$,  it has been proved (see \cite[Theorem 2]{Bac11} and also \cref{lem:b1}) that
	\begin{align*}
		\IE \lVert \theta_n - \thestar\rVert^2 \leq
		\begin{cases}
			n^{-1}, & \ell_0 \mu > 1; \\
			n^{-1} (\log n) , & \ell_0 \mu = 1; \\
			n^{-\ell_0 \mu /2}, & 0 < \ell_0 \mu < 1.
		\end{cases}
	\end{align*}
	Therefore, for $\alpha = 1$, the choice of  $\ell_0$ is critical, but the problem is:
	a small $\ell_0$ leads to a very slow convergence rate of order $n^{-\ell_0 \mu /2}$
	while a large $\ell_0$ might lead to explosion due to the initial condition (see, e.g.,
	\citet*{Bac11} and \citet*{Nem09a} for more details). In practice, one prefers to use a learning rate of  order $n^{-\alpha}$
	with $0 < \alpha < 1$.

\end{remark}

\begin{theorem}
Consider the model \cref{eq:map}. Let
\begin{equation*}
	\thestar = \min_{\theta \in \IR^{d}} M(\theta),
\end{equation*}
 and the algorithm
 \begin{equation*}
	 \theta_{n } = \theta_{n - 1} - \ell_{n } \dot{m}_{\theta_{n - 1}}(X_{n}),
 \end{equation*}
 where $\.m_{\theta}$ is as in \cref{eq:m-mdot}, $\ell_{n} = \ell_0 n^{-\alpha}$ is the learning rate, \(\ell_0 > 0\), \(1/2 < \alpha \leq 1\) and $\theta_0$ is the initial value that is independent
 of $(X_1, \dots, X_n)$.
	Let
	\begin{align*}
		\xi_n & =  \dot{m}_{\thestar}(X_n) - \nabla
			M(\thestar)  , \\
		\eta_n &=\dot{m}_{\theta_{n - 1}}(X_n) -  \dot{m}_{\thestar} (X_n) - \nabla
		M(\theta_{n - 1}) +  \nabla M(\thestar) .
	\end{align*}
	Assume that \textup{(C1(i))} is satisfied for $(\xi_1, \dots, \xi_n)$ and  for any $\theta_1, \theta_2 \in \IR^d$,
	\begin{align}
		\sup_{z \in \mathcal{X}} \| \dot{m}_{\theta_1}(z) - \dot{m}_{\theta_2}(z) \| \leq
		L_{F} \|\theta_1 - \theta_2\|.
		\label{eq:t3.5-con}
	\end{align}
	Assume further that \cref{con:c0}, \cref{con:c2} and \cref{con:c3} are satisfied with $f(\theta) = M(\theta)$, and let
	$\bar\theta_n$ be as defined in \cref{eq:saa}. Then, we have
	\cref{eq:sgd-a,eq:sgd-b} hold with \(c_1 = 2 L_F\).

\end{theorem}
\begin{proof}
We only need to check the condition (C1(ii)) is satisfied.
	 Note that for each $n \geq 1$,
	\begin{align*}
		\dot{m}_{\theta_{n - 1}}(X_n)
		&= \nabla M (\theta_{n - 1}) +  \bigl(\dot{m}_{\theta_{n - 1}}(X_n) - \nabla M(\theta_{n -1})\bigr)\\
		&= \nabla M (\theta_{n - 1}) + \bigl( \dot{m}_{\thestar}(X_n) - \nabla
			M(\thestar) \bigr) \\
		& \quad  + \bigl( \dot{m}_{\theta_{n - 1}}(X_n) -  \dot{m}_{\thestar} (X_n) - \nabla
		M(\theta_{n - 1}) +  \nabla M(\thestar)\bigr) \\
		& = \nabla M (\theta_{n - 1}) + \xi_n + \eta_n .
	\end{align*}
	For $n \geq 0$, let $\mathcal{F}_n = \sigma (\theta_0, X_1, \ldots, X_n)$. Then we have
	\begin{math}
		\IE \{\xi_n\} = 0 \text{ and }\IE \bigl\{ \eta_n \bigm\vert \mathcal{F}_{n - 1}\bigr\} = 0.
	\end{math}
	By \cref{eq:t3.5-con}, it follows that the condition (C1(ii)) in \cref{eq:m-A11} holds with \(c_1 = 2 L_F\).
	Hence, \cref{thm:sgd} implies the desired result.
\end{proof}

%

\section{Proofs of main results}\label{sec:proof1}
\subsection{A randomized concentration inequality}
 To prove \cref{n4-t3-1}, we need to develop a  randomized concentration inequality  for sums of
 multivariate independent random vectors. We use the following notation. For a subset $A$ of $\R^d$, let $d(x,A)=\inf \{ \lVert x - y\rVert: y\in A\}$.
For a given number $\varepsilon >0$,
define
\(
A^{\varepsilon}= \left\{ x \in \IR^d: d(x,A)\leq \varepsilon \right\},
\)
and
\(
    A^{-\varepsilon} = \left\{  x \in A : B(x,\varepsilon) \subset A  \right\},
\)
where $B(x,\varepsilon)$ is the  $d$-dimensional ball centered in $x$ with radius $\varepsilon$.
Specially, for $\epsilon = 0$, let $A^\epsilon = A$. 
Let $\bar{A}$ be the closure of $A$ and let $r(\bar{A}) = \max \{ y\, : \, B(x, y) \subset \bar{A} \text{ for some $x \in \IR^n$} \}$ be the \emph{inradius} of $\bar{A}$.
For $a, b \in \IR$, write $a \wedge b = \min(a, b)$ and $a \vee b = \max(a,b)$. 
Let $\gamma =\sum_{i=1   } ^n \E\{\|\xi_i \|^3\}$ be as in \cref{eq:gamma}.
We have the following proposition. 

\begin{proposition}
Let $W = \sum_{i = 1 }^n \xi_i$, where \( (\xi_i)_{i = 1}^n \) is a sequence of \(\mathbb{R}^d\)-valued independent random vectors satisfying that \( \IE \{ \xi_i \} = 0\) for \(1 \leq i \leq n\) and \(\sum_{i = 1 }^n \IE \{ \xi_i \xi_i \transpose \} = I_d\).
Let $\Delta_1$ and  $\Delta_2$ be nonnegative random variables. Then we have  for all $A\in \A$
  such that $r(\bar{A}) > \gamma$,
  \begin{align}
	  \IP\bigl(W\in A^{4\gamma +  \Delta_1  } \setminus A^{4 \gamma - \bar\Delta_2 } \bigr)
	  & \leq 19 d^{1/2} \gamma +
	  2 \E \bigl\{\| W\| (\Delta_1 + \Delta_2) \bigr\} + 2\sum_{i=1}^n \sum_{j = 1}^2 \E \bigl\{\| \xi_i\| |\Delta_j - \Delta_j^{(i)} |\bigr\}, \label{n4-t1.1-2}
  \end{align}
  where  $\bar\Delta_2 = \Delta_2 \wedge (r(\bar{A}) - \gamma)$ and $\Delta^{(i)}$ is a random variable independent of $\xi_i$.
  \label{n4-thm1.1}
\end{proposition}
The proof of this proposition is postponed in Subsection 4.3. 
\begin{remark}
    Specially, if $\Delta_1 = \varepsilon$ and \(\Delta_2 = 0\) where \(\epsilon >0\) is a constant, then \cref{n4-t1.1-2} reduces to
  \[
	  P(W\in A^{4 \gamma + \varepsilon } \setminus A^{4\gamma}) \leq 2d^{1/2} \varepsilon + 19 d^{1/2} \gamma,
  \]
  which is equivalent to the result in \cite{Che15b} up to a constant factor.
\end{remark}

\begin{remark}
	When $d = 1$, the right hand side of \cref{n4-t1.1-2} reduces to
	\begin{align*}
		19 \gamma + 2 \IE  |W (\Delta_1 + \Delta_2)|  + 2 \sum_{i = 1}^n \sum_{j = 1}^2 \IE  \lvert \xi_i (\Delta_j -
		\Delta_j^{(i)})\rvert,
	\end{align*}
	which is equivalent to \cite{Ch07N}'s concentration inequality result.
	Recently, \cite{Sh16C} proved that the term $\IE  \lvert W \Delta \rvert$ can
	be improved to be $\IE  | \Delta| $ in \cref{n4-i0}. However, due to some technical difficulty, we
	are not able to remove the $W$ term in our result.  Nevertheless, 
	the order
	in $n$ is optimal in many applications.

\end{remark}

\ignore{
\begin{remark}
	The bound of \(d\) is not optimal.
	It is known that for \(Z \sim N(0, I_d)\), and for \(\epsilon_1 , \epsilon_2 \geq 0\),
	\begin{align*}
		\IP ( Z \in A^{\epsilon_1} \setminus A^{-\epsilon_2}) \leq 4 d^{{1/4}} (\epsilon_1 +
		\epsilon_2),
	\end{align*}
which is the best known order in terms of $d$. We refer to \cite{Ba90S} and \cite{Ben03b} for reference.
\end{remark}
}

\subsection{Proofs of \texorpdfstring{\cref{thm2.1} and \cref{cor:01,cor:2.3}}{Theorem 2.1 and Corollary
2.2}}
We first give the proof of \cref{thm2.1}.
\begin{proof}[Proof of \cref{n4-thm3}]
Without loss of generality, let $A$ be an arbitrary nonempty convex subset of $\IR^d$. 
Let $Z \sim N(0, I_d)$ be independent of all others.
It has been shown in  \cite[Proposition 2.5 and Theorem 3.5]{Che15b} that for \(\epsilon_1 , \epsilon_2 \geq 0\),  
\begin{gather}
	\sup_{A\in \A} \bigl\lvert \IP(W\in A)-\IP(Z\in A) \bigr\rvert\leq 115d^{1/2}\gamma,
	\label{eq:linear.bound}
	\\
	\IP ( Z \in A^{\epsilon_1}\setminus A^{-\epsilon_2} ) \leq d^{1/2} (\epsilon_1 + \epsilon_2).
    \label{n4-l31-1}
\end{gather}
	For each $1 \leq i \leq n$, let $\Delta^{(i)}$ be any random variable that is independent of $\xi_i$.
   Note that $\|T-W\|\leq \Delta $ and that $r(\bar{A}^{2\gamma}) > \gamma$. 
   Applying
   \cref{n4-thm1.1} to $A^{2\gamma}$ with $\Delta_1 = \Delta$ and $\Delta_2 = 0$, and by \cref{eq:linear.bound,n4-l31-1}, we have
  \begin{align}\label{n4-t3-00}
	\MoveEqLeft {\IP(T\in A)-\IP(Z\in A)} \nonumber\\
	  \leq {} & \IP( T \in A^{6\gamma}) - \IP(W \in A^{6\gamma}) + \IP(W \in A^{6\gamma} )  - \IP(Z \in A^{6\gamma}) + \IP(Z \in A^{6\gamma} \setminus A)\nonumber \\
	  \leq {} & \IP(W \in (A^{2\gamma})^{\Delta + 4\gamma} \setminus (A^{2\gamma})^{4\gamma}) + 121 d^{1/2} \gamma \nonumber\\
      \leq {} & 140 d^{1/2} \gamma + 2 \E \bigl\{\| W \| \Delta \bigr\}+ 2 \sum_{i=1}^n \E
	  \bigl\{\|X_i\| |\Delta - \Delta^{(i)} |\bigr\}.
  \end{align}

  This proves the upper bound of \(\IP (T \in A) - \IP (Z\in A)\). For the upper bound of $\IP (Z \in A) - \IP (T \in A)$, we introduce the following notation. 
  Recall that $\bar{A}$ is the closure of $A$ and  $r:=r(\bar{A})$ is the inradius of $\bar{A}$. We consider the following two cases. 

  If $r < 9 \gamma$, then $A^{-9\gamma} = \emptyset$. 
By \cref{n4-l31-1},
  \begin{align}
      \label{n4-t3-01}
      \IP(Z \in A) - \IP(T \in A ) \leq \IP(Z \in A\setminus A^{-9\gamma}) \leq 9 \,d^{1/2}
	  \gamma .
  \end{align}

  Now we consider the case where $r \geq 9\gamma$. Let $A_0 = A^{-4\gamma}$ and it follows that $A_0 \neq \emptyset$ and $r(\bar A_0) = r - 4\gamma$. Let $\Delta_0 = \Delta \wedge (r - 5\gamma) = \Delta \wedge (r(\bar A_0) - \gamma)$.   
  Since $A_0^{4\gamma} = (A^{-4\gamma})^{4\gamma} \subset A$, we have 
\begin{align*}
	\IP \bigl( Z \in A \bigr) - \IP \bigl( T \in A \bigr) 
	& \leq \IP \bigl( Z \in A \bigr) - \IP \bigl( T \in A_0^{4\gamma} \bigr) = Q_1 + Q_2 + Q_3 
	\intertext{where}
	Q_1 &=  \IP \bigl( Z \in A \bigr) - \IP \bigl( Z \in A_0^{4\gamma} \bigr), \\
	Q_2 & =  \IP \bigl( Z \in A_0^{4\gamma} \bigr) - \IP \bigl( W \in A_0^{4\gamma} \bigr), \\
	Q_3 & =  \IP \bigl( W \in A_0^{4\gamma} \bigr) - \IP \bigl( T \in A_0^{4\gamma} \bigr) .
\end{align*}
For $Q_1$, by \cref{n4-l31-1}, we have 
\begin{align*}
	|Q_1| \leq \IP \bigl( Z \in A \setminus A_0 \bigr) \leq 4 d^{1/2} \gamma. 
\end{align*}
For $Q_2$, noting that $A_0^{4\gamma}$ is also convex, by \cref{eq:linear.bound}, we have 
\begin{align*}
	\lvert Q_2 \rvert \leq 115 d^{1/2} \gamma. 
\end{align*}
We now move to give an upper bound of $Q_3$.    
If $0 \leq \Delta \leq  r - 5\gamma$, 
\begin{align}
	\mathds{1} \{ w \in A_0^{4\gamma} \} - \mathds{1} \{ w + D \in A_0^{4\gamma} \} 
	& \leq \mathds{1} \{ w \in A_0^{4\gamma} \setminus A_0^{4\gamma - \Delta} \}. 
    \label{eq:9}
\end{align}
If $\Delta > r - 5\gamma$, then
\begin{equ}
	& \mathds{1} \{ w \in A_0^{4\gamma} \} - \mathds{1} \{ w + D \in A_0^{4\gamma} \} \\
	& \leq \mathds{1} \{ w \in A_0^{4\gamma} \} \\
	& \leq \mathds{1} \{ w \in A_0^{4\gamma} \setminus A_0^{9\gamma - r} \} + \mathds{1} \{ w \in A_0^{9\gamma - r} \}\\
	& \leq \mathds{1} \{ w \in A_0^{4\gamma} \setminus A_0^{4\gamma - (r - 5\gamma)} \} + \mathds{1} \{ w\in A^{5\gamma - r} \}, 
    \label{eq:10}
\end{equ}
where the last line follows from the fact that $(A^{-4\gamma})^{9\gamma - r} \subset A^{5\gamma - r}$. 
Equations \cref{eq:9,eq:10} yield 
\begin{equ}
	Q_3 & =  \IP \bigl( W \in A_0^{4\gamma}  \bigr) - \IP \bigl( W + D \in A_0^{4\gamma} \bigr) \\
	& \leq \IP \bigl( W \in A_0^{4\gamma} \setminus A_0^{4\gamma - \Delta_0} \bigr) + \IP \bigl( W \in A^{5\gamma - r} \bigr).
    \label{eq:11}
\end{equ}
For each $1 \leq i \leq n$, let $\Delta_0^{(i)} = \Delta^{(i)} \wedge (r(\bar A_0) - \gamma)$. 
For the first term of \cref{eq:11}, by \cref{n4-thm1.1}, we have 
\begin{align*}
	\IP \bigl( W \in A_0^{4\gamma} \setminus A_0^{4\gamma - \Delta_0}  \bigr)
	& \leq 19 d^{1/2} \gamma + 2 \IE \{ \lVert W\rVert \Delta_0 \} + 2 \sum_{i=1}^n \IE \{ \lVert \xi_i \rVert \lvert \Delta_0 - \Delta_0^{(i)}\rvert \}\\
	& \leq 19 d^{1/2} \gamma + 2 \IE \{ \lVert W\rVert \Delta \} + 2 \sum_{i=1}^n \IE \{ \lVert \xi_i \rVert \lvert \Delta - \Delta^{(i)}\rvert \}.
\end{align*}

For the second term of \cref{eq:11}, since $A^{-r - \gamma} = \emptyset$ and $A^{5\gamma - r}$ is convex and nonempty, by \cref{eq:linear.bound,n4-l31-1}, we have 
\begin{align*}
	\IP ( W \in A^{5\gamma - r} ) 
	& \leq \lvert \IP (W \in A^{5\gamma -r }) - \IP (Z \in A^{5\gamma - r})	\rvert + \IP (Z \in A^{5\gamma - r} \setminus A^{-\gamma - r}) \\
	& \leq 115 d^{1/2} \gamma + 6 d^{1/2} \gamma \leq 121 d^{1/2} \gamma. 
\end{align*}
Then it follows that 
\begin{align*}
	Q_3 \leq 140 d^{1/2} \gamma +  2 \IE \{ \lVert W\rVert \Delta \} + 2 \sum_{i=1}^n \IE \{ \lVert \xi_i \rVert \lvert \Delta - \Delta^{(i)}\rvert \}.
\end{align*}
  Combining the upper bounds of $Q_1, Q_2$ and $Q_3$, we have  
  \begin{equ}
      \label{n4-t3-02}
       \MoveEqLeft 
	   \IP(Z \in A) - \IP(T \in A)	   \leq  259 d^{1/2} \gamma + 2 \E \bigl\{\| W \| \Delta \bigr\} + 2 \sum_{i=1}^n \E
	   \bigl\{\|X_i\| \lvert \Delta - \Delta^{(i)}\rvert\bigr\}.
  \end{equ}
  By \cref{n4-t3-00,n4-t3-01,n4-t3-02}, we have
  \begin{align*}
	  \sup_{A \in \A} | \IP(T\in A) - \IP(W \in A)| \leq {} &  259 d^{1/2} \gamma + 2 \E \{\|W\| 
	  \Delta\}    + 2 \sum_{i=1}^n \E \{\|X_i\|
															\lvert \Delta - \Delta^{(i)}\rvert\},
  \end{align*}
  as desired.
\end{proof}

\begin{proof}
[Proof of \cref{cor:01}]
Let $\widetilde{T} = W + D \I(O)$. For any $A \in \mathcal{A}$,
\begin{align*}
	\lvert \IP (T \in A) - \IP (\tilde{T} \in A )\rvert \leq \IP (O^c).
\end{align*}
Applying \cref{n4-thm3} to $\tilde{T}$ yields
\begin{align*}
	\sup_{A \in \mathcal{A}} \bigl\lvert \IP (\tilde{T} \in A) - \IP (Z \in A) \bigr\rvert
		& \leq 259 d^{1/2} \gamma + 2 \E \bigl\{\|W\| \Delta \bigr\} + 2 \sum_{i=1}^n
		\E \bigl\{\|\xi_i\| \lvert \Delta - \Delta^{(i)}\rvert\bigr\}. 	
\end{align*}
Combining the foregoing inequalities we obtain the desired result.
\end{proof}

\begin{proof}
[Proof of \cref{cor:2.3}]
For any convex set \(A \subset \IR^d\), we have \(\Sigma^{-1/2 } A := \{  y \in \IR^d: y = \Sigma^{-1/2 } x, x \in A \}\) is also a convex subset of \(\IR^d\).
To see this, it suffices to show that for any \(y_1 , y_2 \in \Sigma^{-1/2 } A\) and for any \(0 \leq t \leq 1\),
\begin{equ}
	\label{eq:t-convex}
	t y_1 +(1 - t) y_2 \in \Sigma^{-1/2 } A.
\end{equ}
Since \(y_1, y_2 \in \Sigma^{-1/2 } A\), it follows that there exist \(x_1, x_2 \in A\) such that
\begin{align*}
	y_1 = \Sigma^{-1/2 } x_1, \quad y_2 = \Sigma^{-1/2 } x_2.
\end{align*}
Moreover, as \(A\) is convex, we have
for any \(0 \leq t \leq 1\),
\begin{align*}
	t x_1 + (1 - t) x_2 \in A,
\end{align*}
and thus
\begin{align*}
	t y_1 + (1 - t) y_2
	& = t \Sigma^{-1/2 } x_1 + (1 - t ) \Sigma^{-1/2 } x_2 \\
	& = \Sigma^{-1/2 } ( t x_1 + (1 - t) x_2 ) \in \Sigma^{-1/2 } A.
\end{align*}
This proves \cref{eq:t-convex} and hence \(\Sigma^{-1/2 } A\) is convex.
Note that
\begin{align*}
	\IP ( T\in A) - P(\Sigma^{1/2 } Z \in A)
	& = \IP (\Sigma^{-1/2 } T \in \Sigma^{-1/2 }A) - \IP( Z \in \Sigma^{-1/2 } A),
\end{align*}
and we have
\begin{align*}
	\sup_{A \in \mathcal{A}} \bigl\lvert \IP(T \in A) - \IP (\Sigma^{-1/2 } A \in A)\bigr\rvert
	& = \sup_{A \in \mathcal{A}} \bigl\lvert \IP (\Sigma^{-1/2 } T \in A ) - \IP(Z\in A)\bigr\rvert.
\end{align*}
Applying \cref{n4-thm3} yields the desired result.
\end{proof}

\subsection{Proof of \texorpdfstring{\cref{n4-thm1.1}}{Proposition 4.1}}
We apply the ideas in \cite{Ch07N} and \cite{Che15b} to prove \cref{n4-thm1.1} in this subsection.
Before the proof, we first introduce some definitions and lemmas.

Given $A \in \mathcal{A}$ and  $\varepsilon \geq 0$, we construct
$  f_{A,\varepsilon} \colon \R^d\rightarrow \R^d$ as follows.
Let $\mathcal{P}_A$ be the projection operator on $A$, that is, for any $x \in \IR^d$, let
\begin{align*}
	\mathcal{P}_A (x) := \argmin_{y \in A} \lVert x - y\rVert.
\end{align*}
Therefore, $\mathcal{P}_A(x)$ is the nearest point of $x$ in the set $A$.

Let $\bar A$ be the closure of $A$, and
\begin{equ}
	f_{A, \epsilon} (x) =
	\begin{cases}
		\label{eq:function.f}
		0, & x \in \bar A, \\
		x - \mathcal{P}_{\bar A}(x), & x \in A^{\epsilon } \setminus \bar A, \\
		\mathcal{P}_{(\bar A)^{\epsilon}} (x) - \mathcal{P}_{\bar A} (x), & x \in \IR^d \setminus
		A^{\epsilon}.
	\end{cases}
\end{equ}
Let $r(\bar{A}) = \max \{ y: B(x,y) \subset \bar{A} \text{ for some $x \in \IR^d$}\}$ be the inradius of $\bar{A}$.  
We introduce the following lemma, whose proof can be found in \citet*[Lemmas 2.1, 2.2 and Proposition 2.7]{Che15b}.
\begin{lemma}
	Let $\epsilon > 0$ and $\gamma > 0$ and  $f:=f_{A, \epsilon + 8 \gamma}$ be as in \cref{eq:function.f}.
  We have
  \begin{enumerate}
	  \item [\textup{(i)}] $\|f\|\leq  \epsilon + 8 \gamma $;
	  \item [\textup{(ii)}] for all $\xi,\eta \in \R^d$, $\langle{\xi , f(\eta+\xi)-f(\eta)}\rangle\geq 0$;
	  \item [\textup{(iii)}] for
	  $w\in A^{4\gamma +\varepsilon}\setminus A^{4\gamma}$ and $\|x\|\leq 4 \gamma,$ we have
	  \[
		  \inprod{x, f(w)-f(w-x)}\geq \frac{3}{4} (x\cdot h_1)^2,
	  \]
	  where $h_1=(w_0-w)/\|w_0-w\|$ and $w_0 = \mathcal{P}_{\bar A}(w)$.
  \end{enumerate}
  \label{n4-lem1}
\end{lemma}

\def\f{g_{\Delta_1, \bar\Delta_2}}
\def\fa{g_{\Delta_1^{(i)}, \bar\Delta_2}}
\def\fii{g_{\Delta_1^{(i)}, \bar\Delta_2^{(i)}}}
Now we are ready to give the proof of \cref{n4-thm1.1}.
\begin{proof}[Proof of  \cref{n4-thm1.1}]
Let $A \in \mathcal{A}$ be nonempty such that $r:=r(\bar{A}) > \gamma$. 
Set $\bar\Delta_{2} = \Delta_2 \wedge (r - \gamma)$. 
Let $\Delta_1^{(i)}$ and $\Delta_2^{(i)}$ be any random variables that are independent of $\xi_i$ and let $\bar\Delta_{2}^{(i)} = \Delta_2^{(i)} \wedge (r - \gamma)$.
For any $a\geq 0$ and $0 \leq b \leq r - \gamma$,   
define $g_{a,b} = f_{A^{-b}, 8\gamma + a + b}.$
Noting that $\IE \xi_i= 0$ and observing that $\Delta_1^{(i)}$ and $\bar\Delta_2^{(i)}$ are independent of $\xi_i$, we have
\begin{align*}
	\IE \bigl\{ \inprod{ \xi_i, \fii(W - \xi_i) }\bigr\} = 0,
\end{align*}
and thus,
\begin{equ}
  \label{eq:wgw.bound}
  \E \{\inprod{W,  \f (W)}\}
  = {} & \sum_{i=1}^n \Bigl( \E \{\inprod{ \xi_i,  \f(W)}\} - \IE \bigl\{ \inprod{ \xi_i, \fii(W - \xi_i) }\bigr\} \Bigr)\\
  = {} & H_1+ H_2,
\end{equ}
where
\begin{align*}
  H_1 &= \sum_{i=1}^n \E \bigl\{\inprodb{ \xi_i,  \f (W)-\f(W-
  \xi_i)}\bigr\} ,\\
  H_2 & =\sum_{i=1}^n \E \bigl\{\inprod{ \xi_i,  \f(W-\xi_i)-\fii(W-\xi_i)}\bigr\}.
\end{align*}

For the upper bound of $H_2$, by the definition of $f$, we have
\begin{equ}
	\label{eq:h20}
	\| \f(w) - \fii(w)\| & \leq
	\bigl\lVert \f(w) - \fa (w) \bigr\rVert   + \bigl\lVert \fa (w) - \fii (w) \bigr\rVert .
\end{equ}
Without loss of generality, assume that $\Delta_1^{(i)} \leq \Delta_1$. Let $A_2 = (\bar A)^{-\bar\Delta_2}, A_3 = A^{8 \gamma + \Delta_1^{(i)}}, A_4 = A^{8\gamma +
\Delta_1}$ and  $w_j = \mathcal{P}_{A_j}(w)$ for $j = 2, 3, 4$.

If $w \in A_3\subset A_4$, then
$$\f(w) = \fa(w); $$
if
$w \in A_4 \setminus A_3$,
then
$$\f(w) = w - w_2
, \quad
\fa (w) = w_3 - w_2, $$
and
\begin{align*}
	\lVert w - w_3\rVert = \lVert w_4 - w_3\rVert , \quad \text{for } w \in A_4 \setminus A_3;
\end{align*}
if $w \in A_4^c$,
then
\begin{align*}
	\f(w) = w_4 - w_2
\text{ and }
\fa (w) = w_3 - w_2.
\end{align*}
By the definition of \(w_3\) and \(w_4\), it follows that $\lVert w_4 - w_3\rVert  \leq \lvert \Delta_1 - \Delta_1^{(i)}\rvert$. Hence,
\begin{equ}
	\label{eq:h21}
	\bigl\lVert \f(w) - \fa (w) \bigr\rVert \leq \lvert \Delta_1 - \Delta_1^{(i)}\rvert.
\end{equ}
Similarly,
\begin{equ}
	\label{eq:h22}
	\bigl\lVert \fa(w) - \fii (w) \bigr\rVert \leq \lvert \bar\Delta_2 - \bar\Delta_2^{(i)}\rvert \leq \lvert \Delta_2 - \Delta_2^{(i)}\rvert.
\end{equ}
By \cref{eq:h20,eq:h21,eq:h22},
  \begin{equation}
    \label{n4-h2}
	H_2 \leq  \sum_{i=1}^n \IE \{\|\xi_i\| (| \Delta_1 - \Delta_1^{(i)}| +  \lvert \Delta_2 -
	\Delta_2^{(i)}\rvert )\}.
  \end{equation}

  We next estimate the lower bound of $H_1$.
  By \cref{n4-lem1}, we have
  \begin{align}
	  \label{eq:h1.lower.bound}
    H_1  = {} & \sum_{i=1}^n \E \bigl\{\inprodb{ \xi_i,  \f(W)-\f(W-\xi_i)}\bigr\} \nonumber\\
	\geq {} & \sum_{i=1}^n \E \Bigl\{\inprodb{\xi_i,  \f(W)-\f(W-\xi_i)}  \I(|\xi_i|\leq 4\gamma)
	\I(\seta)\Bigr\}\nonumber \\
     \geq {} & \frac{3}{4} \sum_{i=1}^n \E \Bigl\{\inprod{\xi_i,  U}^2 \I(\setb)\I(\seta)\Bigr\}
	 \coloneqq
	 {}  \frac{3}{4} R
  \end{align}
  where $U \coloneqq (W_0-W)/\|W_0-W \| = (U_1, \dots, U_d)$ and $W_0 = \mathcal{P}_{\bar A} (W)$.  Observe that by \cref{eq:h1.lower.bound},
  \begin{align*}
    R = {} & \sum_{i=1}^n \sum_{j=1}^d\E \bigl\{\xi_{ij}^2 U_j^2 \I(\setb)\I(\seta)\bigr\}\\
     &+\sum_{i=1}^n \sum_{j\neq j'}\E \bigl\{\xi_{ij}\xi_{ij'} U_j U_{j'} \I(\setb)\I(\seta)\bigr\}\\
     \coloneqq {} & R_1+R_2.
  \end{align*}
  For $R_1$, rearranging the summations yields
  \begin{align*}
    R_1 = {} &
    \sum_{j=1}^d \E \biggl\{\I(\seta)U_j^2 \sum_{i=1}^n \xi_{ij}^2\I(\setb)\biggr\}\\
     = {} &
	 \sum_{j=1}^d \E \biggl\{\I(\seta)U_j^2\biggl( \sum_{i=1}^n \Bigl(\xi_{ij}^2\I(\setb)-\IE  \xi_{ij}^2
	 \I(\setb)\Bigr)\biggr)\biggr\}\\
     & +\sum_{j=1}^d \Bigl(\E \{\I(\seta)U_j^2\}\Bigr) \biggl(\sum_{i=1}^n \IE\{ \xi_{ij}^2 \I(\setb)\}\biggr)\\
     := {} & R_{11}+R_{12}.
  \end{align*}
  By the basic inequality that \( ab \leq \gamma a^2 + (1/ 4\gamma) b^2\) for \(a , b \geq 0\), it follows that with
  \begin{align*}
  	a = U_j^2 \quad \text{and} \quad b = \biggl\lvert \sum_{i=1}^n \Bigl(\xi_{ij}^2\I(\setb)-\IE  \xi_{ij}^2
	 \I(\setb)\Bigr)\biggr\rvert,
  \end{align*}
  we have
  \begin{equ}
	  \label{eq:r11.bound}
	  |R_{11}| \leq {} & \sum_{j=1}^d \E \biggl\{ U_j^2 \biggl\lvert \sum_{i=1}^n \Bigl(\xi_{ij}^2\I(\setb)-\IE  \xi_{ij}^2
	 \I(\setb)\Bigr)\biggr\rvert \biggr\} \\
	\leq {} & \gamma\sum_{j=1}^d \E\{U_j^4\} +\frac{1}{ 4\gamma} \sum_{j=1}^d \Var\biggl(\sum_{i=1}^n
    \xi_{ij}^2\I(\setb)\biggr)\\
     \leq {} & \gamma\sum_{j=1}^d \E\{U_j^4\} +\frac{1}{4 \gamma} \sum_{j=1}^d \sum_{i=1}^n \E \bigl\{\xi_{ij}^4 \I(\setb)\bigr\}.
  \end{equ}
  As for $R_{12}$,
   recalling that \(\sum_{j = 1 }^d U_i^2 = 1\) and $\sum_{i = 1}^n\IE \{\xi_i \xi_i\transpose\} = I_d$, we have  $\sum_{i = 1}^n \IE
   \{\xi_{ij}^2\} = 1$ for each $1 \leq j \leq d$,  and
  \begin{align}
	  \label{eq:r12.bound}
    R_{12} = {} & \sum_{j=1}^d \E \Bigl\{\I(\seta) U_j^2\Bigr\} \biggl( \sum_{i=1}^n \IE
		\{\xi_{ij}^2 \}  - \sum_{i=1}^n
	\IE \{ \xi_{ij}^2\I(\|\xi_i\|>4\gamma) \} \biggr) \nonumber \\
		= {} & \IP(\seta) \\
			 & \qquad - \IE \biggl\{\I (\seta)\sum_{j = 1}^d  U_j^2 \biggl( \sum_{i = 1}^n \IE
		\{\xi_{ij}^2 \I( \lVert \xi_i\rVert > 4 \gamma)\}\biggr)\biggr\}. \nonumber
  \end{align}
  By \cref{eq:r11.bound,eq:r12.bound}, it follows that
  \begin{align}
	  \label{eq:r1.bound}
	  R_1\geq {} &
	  \begin{multlined}[t]
		  \IP(\seta)-\gamma\sum_{j=1}^d \E\{U_j^4\}  - \frac{1}{4\gamma}\sum_{i = 1}^n \sum_{j=1}^d
		  \E \bigl\{\xi_{ij}^4 \I(\setb)\bigr\}  \\
		  -\IE \biggl\{\I (\seta) \sum_{j = 1}^d  U_j^2 \biggl( \sum_{i = 1}^n \IE \{\xi_{ij}^2 \I( \lVert \xi_i\rVert > 4 \gamma)\}\biggr)\biggr\}.
	  \end{multlined}
  \end{align}
  Similarly, noting that \(\sum_{i = 1}^n \IE \{ \xi_{ij} \xi_{ij'}\} = 0 \) for \(j \neq j'\), we have
  \begin{equ}
	  \label{eq:r2.bound}
	  R_2
    & \geq -\gamma \sum_{j\neq j'} \E\bigl\{U_j^2 U_{j'}^2\bigr\}-\frac{1}{4\gamma} \sum_{i = 1}^n
	\smashoperator[r]{\sum_{1 \leq j\neq j' \leq d}} \E\bigl\{(\xi_{ij}\xi_{ij'})^2 \I(\setb)\bigr\} \\
	& \quad  - \sum_{i = 1}^n \sum_{1 \leq j \neq j' \leq d} \IE \Bigl\{ \I(\seta) U_{j} U_{j'} \IE \{ \xi_{ij} \xi_{ij'}
	\I( \lVert \xi_i\rVert \leq  4\gamma) \} \Bigr\}\\
    & = -\gamma \sum_{j\neq j'} \E\bigl\{U_j^2 U_{j'}^2\bigr\}-\frac{1}{4\gamma} \sum_{i = 1}^n
	\smashoperator[r]{\sum_{1 \leq j\neq j' \leq d}} \E\bigl\{(\xi_{ij}\xi_{ij'})^2 \I(\setb)\bigr\} \\
	& \quad  - \sum_{i = 1}^n \sum_{1 \leq j \neq j' \leq d} \IE \Bigl\{ \I(\seta) U_{j} U_{j'} \IE \{ \xi_{ij} \xi_{ij'} \} \Bigr\}\\
	& \quad  - \sum_{i = 1}^n \sum_{1 \leq j \neq j' \leq d} \IE \Bigl\{ \I(\seta) U_{j} U_{j'} \IE \{ \xi_{ij} \xi_{ij'}
	\I( \lVert \xi_i\rVert >  4\gamma) \} \Bigr\}\\
    & = -\gamma \sum_{j\neq j'} \E\bigl\{U_j^2 U_{j'}^2\bigr\}-\frac{1}{4\gamma} \sum_{i = 1}^n
	\smashoperator[r]{\sum_{1 \leq j\neq j' \leq d}} \E\bigl\{(\xi_{ij}\xi_{ij'})^2 \I(\setb)\bigr\} \\
	& \quad  - \sum_{i = 1}^n \sum_{1 \leq j \neq j' \leq d} \IE \Bigl\{ \I(\seta) U_{j} U_{j'} \IE \{ \xi_{ij} \xi_{ij'}
	\I( \lVert \xi_i\rVert >  4\gamma) \} \Bigr\}.
  \end{equ}
	Observe that
	\begin{equ}
		\label{eq:r3-1}
		\sum_{i = 1}^{n} \IE \{ \lVert \xi_i \rVert^4 \I (\setb) \} \leq 4 \gamma \sum_{i = 1}^n \IE \lVert \xi_i \rVert^3 \leq 4 \gamma^2
	\end{equ}
	and by the Markov inequality,
	\begin{align}
		\label{eq:r3-2}
		\sum_{i = 1}^n \IE \{ \lVert \xi_i\rVert^2 \I ( \lVert \xi_i\rVert > 4\gamma ) \}
		& \leq \frac{1}{4 \gamma} \sum_{i = 1}^n \IE \lVert \xi_i \rVert^3 = \frac{1}{4}.
	\end{align}
  Recall that $\lVert U \rVert = 1$ and $\gamma = \sum_{i = 1}^n \IE \{ \lVert \xi_i\rVert^3\}$, and
  thus \cref{eq:r1.bound,eq:r2.bound,eq:r3-1,eq:r3-2} yield
  \begin{equ}\label{n4-eq:R}
	  R & \geq \IP (\seta) - \gamma \IE \{ \lVert U \rVert^4 \} -
	  \frac{1}{4\gamma} \sum_{i = 1}^n \IE \{\lVert \xi_i\rVert^4 \I(\setb)\}  \\
		& \qquad   - \sum_{i = 1}^n \IE \Bigl\{ \I(\seta) \IE \Bigl\{ \Bigl( \sum_{j = 1 }^d U_j \xi_{ij}\Bigr)^2 \I(\lVert
	  \xi_i\rVert > 4 \gamma)\Bigr\} \Bigr\} \\
		& \geq \IP (\seta) - \gamma \IE \{ \lVert U \rVert^4 \} -
	  \frac{1}{4\gamma} \sum_{i = 1}^n \IE \{\lVert \xi_i\rVert^4 \I(\setb)\}  \\
& \qquad   - \sum_{i = 1}^n \IE \Bigl\{ \I(\seta) \IE \{ \lVert U \rVert^2 \lVert \xi_i \rVert^2 \I(\lVert
	  \xi_i\rVert > 4 \gamma)\} \Bigr\} \\
		& = \IP (\seta) - \gamma \IE \{ \lVert U \rVert^4 \} -
	  \frac{1}{4\gamma} \sum_{i = 1}^n \IE \{\lVert \xi_i\rVert^4 \I(\setb)\}  \\
& \qquad   - \sum_{i = 1}^n \IE \Bigl\{ \I(\seta) \IE \{ \lVert \xi_i \rVert^2 \I(\lVert
	  \xi_i\rVert > 4 \gamma)\} \Bigr\} \\
		& \geq
		\IP(\seta)-{2\gamma}  - \frac{1}{4} \IP (\seta) \\
	  & = \frac{3}{4} \IP (\seta) - 2 \gamma.
  \end{equ}
  By \cref{eq:h1.lower.bound,n4-eq:R}, we have
  \begin{equ}
	  \label{eq:h1.final.bound}
	  H_1 \geq \frac{9}{16} \IP (\seta) - \frac{3}{2} \gamma.
  \end{equ}
  On the other hand, note that \(\IE \lVert W \rVert^2 = d\) and by \cref{n4-lem1}, $\|\f(W)\|\leq (\Delta_1 + \Delta_2 +8 \gamma)$. Thus,
  \begin{align}
	  \label{eq:h.wgbound}
  	|\E \inprodb{W, \f (W)}|\leq \E\|W\| (\Delta_1 + \Delta_2)   + 8d^{1/2}\gamma.
  \end{align}
  Combining  inequalities
  \cref{n4-h2,eq:h1.final.bound,eq:wgw.bound,eq:h.wgbound} yields
  \begin{align*}
	  \MoveEqLeft
	  \IP(\seta) \\
	  \leq {} & 2 \Bigl\lvert \E \bigl\langle W, g_{\Delta_1, \Delta_2 } (W) \bigr\rangle\Bigr\rvert + 2 H_2 + 3 \gamma  \\
	  \leq {} & 2 \E \bigl\{\|W\| (\Delta_1 + \Delta_2) \bigr\}+16 d^{1/2} \gamma+
	  3\gamma + 2\sum_{i=1}^n \sum_{j = 1}^2 \E\|\xi_i\||\Delta_j-\Delta_j^{(i)} |\\
	  \leq {}           & 2 \E\|W\| (\Delta_1 + \Delta_2) + 19 d^{1/2}\gamma + 2\sum_{i=1}^n
	  \sum_{j = 1}^2
	  \E\|\xi_i\|\big| \Delta_j - \Delta_j^{(i)}\big|.
  \end{align*}
  This proves \eqref{n4-t1.1-2}.
\end{proof}

\section{Proofs of other results}
\label{sec:proof2}
In this section, we give the proofs of the theorems in Section 3.
\subsection{Proof of \tpstring{\cref{thm:mt1}}{Theorem 3.1}}
Note that $\hat\theta_n$ minimizes  $\mathbb{M}_n(\theta)$, and  $m_\theta$ is smooth for  $\theta$. By the Taylor expansion, it follows that
\begin{align*}
	0 = \frac{1}{n} \sum_{i = 1}^n \dot{m}_{\hat\theta_n} (X_i)
	= \frac{1}{n} \sum^{n}_{i=1} \dot{m}_{\theta^*} (X_i) +  \frac{1}{n} \sum_{i = 1}^n \int_{ 0 }^{ 1 }  \bigl(\ddot{m}_{\theta_t} (X_i)\bigr) \bigl(\hat\theta_n - \thestar\bigr) d t   ,
\end{align*}
where
\begin{math}
	\theta_t = \thestar + t (\hat \theta_n - \thestar).
\end{math}
Therefore, recalling that \(V = \IE \{ \ddot{m}_{\thestar} (X) \}\) and \(\xi_i = \dot{m}_{\thestar} (X_i)\),
\begin{align*}
	V \bigl(\hat\theta_n - \thestar\bigr) &=
	- \frac{1}{n}\sum_{i = 1}^{n} \xi_i - \Bigl(\frac{1}{n} \sum_{i = 1}^{n} \ddot{m}_{\thestar}
	(X_i) - V \Bigr) \bigl(\hat \theta_n - \thestar\bigr) \\
	& \quad  -  \frac{1}{n} \sum_{i = 1}^{n} \int_{ 0 }^{ 1 } \bigl( \ddot{m}_{\theta_t}
	(X_i) - \ddot{m}_{\thestar} (X_i) \bigr) \bigl(\hat\theta_n - \thestar\bigr)d t   ,
\end{align*}
Let
\begin{align*}
	W = - \frac{1}{\sqrt{n}} \sum_{i = 1}^n \Sigma^{-1/2} \xi_i,
\end{align*}
and
\begin{align*}
	D & = - \sqrt{n} \Sigma^{-1/2}  \Bigl(\frac{1}{n} \sum_{i = 1}^{n} \ddot{m}_{\thestar}
	(X_i) - V \Bigr) \bigl(\hat \theta_n - \thestar\bigr) \\
	& \quad  - \sqrt{n} \Sigma^{-1/2}  \Bigl( \frac{1}{n} \sum_{i = 1}^{n} \int_{ 0 }^{ 1 } \bigl( \ddot{m}_{\theta(t)}
	(X_i) - \ddot{m}_{\thestar} (X_i) \bigr) d t  \Bigr) \bigl(\hat\theta_n - \thestar\bigr).
\end{align*}
Then, we have
\begin{equ}
	\label{eq:m-T}
	T\coloneqq \sqrt{n} \Sigma^{-1/2} V (\hat\theta_n - \thestar) = W + D.
\end{equ}
By \cref{con:a1,con:a2}, we have
\begin{equ}
	\label{eq:m-D}
	\lVert D\rVert \leq n^{1/2} \cla^{-1/2} \bigl( H_1 \lVert \hat \theta_n -
	\thestar \rVert + H_2 \lVert \hat\theta_n - \thestar \rVert^2\bigr),
\end{equ}
where
\begin{align*}
	H_1 &=  \Bigl\lVert \frac{1}{n} \sum_{i = 1}^{n} \bigl( \ddot{m}_{\thestar} (X_i) - \IE \{
	\ddot{m}_{\thestar} (X_i) \} \bigr)\Bigr\rVert,
		& H_2 = \frac{1}{n} \sum_{i = 1}^n \lVert m_2(X_i)\rVert.
\end{align*}
Let
\begin{math}
	\Delta = n^{1/2} \cla^{-1/2} \bigl( H_1 \lVert \hat\theta_n - \thestar \rVert + H_2 \lVert
	\hat\theta_n - \thestar\rVert^2\bigr),
\end{math}
and it follows that $\lVert T - W\rVert \leq \Delta$.
Let $ (X_1', \dots, X_n' )$ be an
independent copy of $ (X_1, \dots, X_n )$, and define
\begin{align*}
	X_j^{(i)} =
	\begin{cases}
		X_j, & j \neq i; \\
		X_i', & j = i.
	\end{cases}
\end{align*}
Moreover, let
\begin{align*}
	H_1^{(i)} & =\Bigl\lVert  \frac{1}{n} \sum_{j = 1}^n \bigl( \ddot{m}_{\thestar} (X_j^{(i)}) - \IE \left\{ \ddot{m}_{\thestar} (X_j)
	\right\}\bigr) \Bigr\rVert, \quad
	H_2^{(i)}  = \frac{1}{n} \sum_{j =1}^n \lVert m_2(X_j^{(i)})\rVert \\
	\hat{\theta}_n^{(i)} & = \argmin_{\theta \in \Theta} \frac{1}{n} \sum_{j = 1}^n m_{\theta}
	(X_j^{(i)}) , \quad
						  \Delta^{(i)}  = n^{1/2} \cla^{-1/2} \bigl(H_1^{(i)} \lVert \hat\theta_n^{(i)} - \thestar\rVert +
	H_2^{(i)} \lVert \hat\theta_n^{(i)} - \thestar\rVert^2\bigr).
\end{align*}
Then, $\Delta^{(i)}$ is independent of  $X_i$ and  $\xi_i$.
\cref{thm:mt1} follows directly from \cref{n4-thm3} and  the following proposition.
\begin{proposition}
	\label{pro:m-1}
	Assume that conditions \cref{con:a1,con:a2} are satisfied.
	Then, we have
	\begin{gather}
		\sum_{i= 1}^n \IE \lVert \Sigma^{-1/2} \xi_i\rVert^3 \leq C d^{{3/2}}n, \label{eq:pm-0} \\
		\IE \{ \lVert W\rVert \Delta\} \leq Cd^{{13}/{8}} n^{-{1/2}}, \label{eq:pm-1}\\
		\sum_{i = 1}^n \IE \lVert \Sigma^{-1/2}\xi_i\rVert \lvert \Delta - \Delta^{(i)}\rvert \leq C
	d^{9/4}  ,
		\label{eq:pm-2}
	\end{gather}
	where $C > 0$ is a constant
	depending only on  $\cla, \clb, c_1, c_2,
	c_3,c_4$ and $\mu$.

\end{proposition}



In order to prove \cref{pro:m-1}, we first need to prove three useful lemmas, whose proofs are postponed to \cref{sec:appendix}. 

The following lemma provides an upper bound for the $p$-th moment of $\lVert \hat\theta_n - \thestar \rVert$, whose proof (see \cref{sub:proof_of_lemma_5_2}) is based on the ideas in \cite[Theorem 3.2.5]{van96}. 
\begin{lemma}
	\label{lem:m-t4}
	Assume that there exist $p \geq 2$ and  $a_6 > 0$ such that  \cref{eq:m-A11,eq:m-A12} are satisfied with
	\begin{math}
		\lVert m_1(X)\rVert_{p + 1} \leq a_6.
	\end{math}
	Then,
	we have
	\begin{align*}
		\IE \lVert \hat\theta_n - \thestar\rVert^{p} \leq C \mu^{-p-1}a_6^{p + 1} d^{\frac{p + 1}{2}} n^{-\frac{p}{2}},
	\end{align*}
	where $C > 0$ is a constant depending only on  $p$.

\end{lemma}
The next lemma gives upper bounds for the fourth moments of $H_1$ and $H_2$. 
The proof can be found in \cref{sub:proof_of_lemma_5_3}, where we use the Rosenthal-type inequality for random matrices (see, e.g, \cite[Theorem A.1]{Che12a}).
\begin{lemma}
	\label{lem:m-h12}
	Under the assumption \cref{con:a1},
	we have
	\begin{align}
		\IE \left\{ H_1^4 \right\} & \leq C c_3^4n^{-2}, \label{eq:m-h18} \\
		\IE \left\{ H_2^4 \right\} & \leq C c_2^4, \label{eq:m-h28}
	\end{align}
	where $C > 0$ is an absolute constant.

\end{lemma}

The following lemma gives an upper bound of $\mathbb{E}\lVert \hat\theta_n -
\hat\theta_n^{(i)}\rVert^2$, whose proof is given in \cref{sub:proof_of_lemma_5_4}.
\begin{lemma}
	\label{lem:m-ti}
	Under the assumptions \cref{con:a1,con:a2}, we have
	\begin{equ}
		\label{eq:m-thi}
		\IE \lVert \hat\theta_n - \hat\theta_n^{(i)}\rVert^2 \leq
		C  \lambda_2^{-2} n^{-2} \Bigl( c_4^2 d + \mu^{-9/4}c_1^{9/4} c_3^2 d^{{9/8}} + \mu^{-9/2}c_1^{9/2} c_2^2 d^{{9/4}} \Bigr)
	\end{equ}
	where $C > 0$ is an absolute constant. 
\end{lemma}

With \cref{lem:m-ti,lem:m-t4,lem:m-h12}, we are ready to give the proof of \cref{pro:m-1}.
\begin{proof}
[Proof of \cref{pro:m-1}]
In this proof, we denote by $C$ a general positive constant depending on $\cla,\,\clb,\,c_1,\\ c_2, c_3,c_4$ and $\mu$, and the value of $C$ might be different in different places.
The bound \cref{eq:pm-0} follows from \cref{con:a2}.
For \cref{eq:pm-1},
since $$\Delta = n^{1/2} \cla^{-1/2} (H_1 \lVert \hat\theta_n - \thestar\rVert + H_2 \lVert
\hat\theta_n - \thestar\rVert^2), $$ it follows from \cref{lem:m-t4,lem:m-h12} and the H\"older inequality that
\begin{equ}
	\label{eq:pm11}
	\IE \bigl\{\Delta \lVert W \rVert \bigr\}
	& \leq C n^{1/2} \cla^{-1/2} \Bigl( \IE \bigl\{ H_1 \lVert W \rVert \lVert \hat\theta_n - \thestar\rVert + H_2 \lVert W \rVert \lVert \hat\theta_n - \thestar\rVert^2 \bigr\} \Bigr) \\
	& \leq C n^{1/2} \cla^{-1/2} \Bigl( \lVert H_1 \rVert_2 \lVert W \rVert_4 \lVert
	\hat\theta_n - \thestar\rVert_2 + \lVert H_2\rVert_2 \lVert W \rVert_4 \lVert \hat\theta_n -
\thestar\rVert_4 \Bigr) \\
	& \leq C d^{\frac{13}{8}} n^{-1/2},
\end{equ}
and
\begin{equ}
	\label{eq:pm12}
	\IE \Delta \leq C d^{{9/8}} n^{-1/2}.
\end{equ}
Therefore \cref{eq:pm11,eq:pm12} yield \cref{eq:pm-1}.

For \cref{eq:pm-2}, we have
\begin{align*}
	| \Delta - \Delta^{(i)} |
	& \leq n^{1/2} \cla^{-1/2} \Bigl( H_1 \lVert \hat\theta_n - \hat\theta_n^{(i)} \rVert +
	| H_1 - H_1^{(i)} | \lVert \hat\theta_n - \theta_n^{(i)}\rVert+ | H_2 - H_2^{(i)}| \lVert
\hat\theta_n^{(i)} - \thestar \rVert^2 \\
	& \hspace{2.4cm}  + H_2 (\lVert \hat\theta_n  - \thestar\rVert + \lVert \hat\theta_n^{(i)} -
	\thestar\rVert)  \lVert \hat\theta_n - \hat\theta_n^{(i)} \rVert \Bigr).
\end{align*}
By the assumption \cref{con:a1},
\begin{align*}
	\lVert H_1 - H_1^{(i)}\rVert_4 \leq \frac{1}{n} \lVert \ddot{m}_{\thestar} (X_i) - \ddot{m}_{\thestar} (X_i^{(i) })\rVert_4 \leq  C n^{-1}
\end{align*}
and
\begin{align*}
	\lVert H_2 - H_2^{(i)}\rVert_4 \leq \frac{1}{n} \Bigl( \lVert m_2 (X_i)\rVert_4 + \lVert m_2 (X_i^{(i) }) \rVert_4 \Bigr) \leq  C n^{-1}.
\end{align*}
By \cref{lem:m-ti,lem:m-t4,lem:m-h12} and the H\"older inequality, we have
\begin{align*}
	\IE \{\lVert \Sigma^{-1/2} \xi_i \rVert \lvert \Delta - \Delta^{(i)}\rvert\} \leq C d^{{13/8}} n^{-1},
\end{align*}
which yields \cref{eq:pm-2}.
\end{proof}

\subsection{Proof of \tpstring{\cref{thm:z-1}}{Theorem 3.2}}
\def\willignore#1{}
Let $\delta_n = \willignore{\color{red}D_1} (D_{\Theta} + 1) d n^{-{(p -2) }/{ (2p - 2)}}$, where \(D_{\Theta}\) is the diameter of the parameter space \(\Theta\). As \(p \geq 3\), it follows that $ \delta_n \geq n^{-1/2}$.
In this subsection, we denote by \(C >0\)  a constant depending only on \(p, c_1, c_2, c_3, \lambda_1, \lambda_2 \) and \(\mu\), which might be different in different places.
The main idea is to rewrite $\sqrt{n} \Sigma^{-1/2} \.\Psi_0 (\hat\theta_n - \thestar) $ as a summation of a linear statistic plus an error term,
and then apply \cref{cor:01} to prove \cref{eq:z-000}.
To this end, by \cref{eq:z-theta},
\begin{equ}
	\label{eq:z-phi.0}
	& \sqrt{n} \bigl( \Psi(\hat\theta_n) - \Psi(\thestar)\bigr)\\
	&= \sqrt{n} \Bigl( \Psi(\hat\theta_n) - \Psi_n(\hat\theta_n)\Bigr) \\
	&= -\sqrt{n} \bigl(\Psi_n(\thestar) - \Psi(\thestar)\bigr)
	- \Bigl( \sqrt{n} (\Psi_n(\hat\theta_n) - \Psi(\hat\theta_n)) - \sqrt{n} (\Psi_n(\thestar) -
	\Psi(\thestar))\Bigr).
\end{equ}
By \cref{eq:z-phi.0}, we obtain
\begin{align*}
	T \coloneqq \sqrt{n} \Sigma^{-1/2} \.\Psi_0 (\hat\theta_n - \thestar)
	= W + D,
\end{align*}
where
\begin{align*}
	W &=  - \frac{1}{\sqrt{n}} \sum_{i = 1}^n
	\Sigma^{-1/2} \xi_i, \\
	D &= - \Sigma^{-1/2} \Bigl( \sqrt{n} (\Psi_n(\hat\theta_n) - \Psi(\hat\theta_n)) - \sqrt{n} (\Psi_n(\thestar) -
	\Psi(\thestar))\Bigr)\\
	  & \quad - \sqrt{n} \Sigma^{-1/2}\Bigl( \Psi(\hat\theta_n ) - \Psi(\thestar) - \.\Psi_0 (\hat\theta_n - \thestar
	  ) \Bigr).
\end{align*}
By \cref{eq:z-B12} and \cref{eq:z-B31},
\begin{align*}
	\|D\| \I ( \lVert \hat\theta_n - \thestar\rVert \leq \delta_n ) &\leq \Delta_1 + \Delta_2,
\end{align*}
where
\begin{align*}
	\Delta_1 & =  \lambda_2^{-1/2}\sqrt{n}\sup_{\theta: \lVert \theta - \thestar\rVert \leq \delta_n}  \Bigl\lVert
	\bigl(\Psi_n(\theta) - \Psi (\theta)\bigr) - \bigl(\Psi_n(\thestar) - \Psi (\thestar)\bigr)
\Bigr\rVert \\
	\Delta_2 & =  c_1 \lambda_2^{-1/2}\sqrt{n} \lVert \hat\theta_n - \thestar\rVert^2\mathds{1} \bigl( \lVert
		  \hat\theta_n - \thestar\rVert \leq \delta_n\bigr).
\end{align*}

Now we construct random variables $\Delta_1 ^{(i)}$ and  $\Delta_2^{(i)}$ that are independent of
$\xi_i$.
Let $(X_1', \\ \dots, X_n')$ be an independent copy of $(X_1, \dots, X_n)$ and let
\begin{align*}
	\Psi_n^{(i)} (\theta) = \Psi_n (\theta) - \frac{1}{n} \bigl(\myphi_{\theta} (X_i) - \myphi_{\theta}
	(X_i')\bigr), \quad 1 \leq i \leq n.
\end{align*}
Let $\hat\theta_n^{(i)}$ be the minimizer of  $\Psi_n^{(i)}$, and let
\begin{align*}
	\Delta_1^{(i)} &= \lambda_2^{-1/2}\sqrt{n}\sup_{\theta: \lVert \theta - \thestar\rVert \leq \delta_n}  \Bigl\lVert
		\bigl(\Psi_n^{(i)}(\theta) - \Psi (\theta)\bigr) - \bigl(\Psi_n^{(i)}(\thestar) - \Psi (\thestar)\bigr)
\Bigr\rVert \\
		\Delta_2^{(i)} &= c_1 \lambda_2^{-1/2}\sqrt{n} \lVert \hat\theta_n^{(i)} - \thestar\rVert \mathds{1} \bigl(
		\lVert \hat\theta_n^{(i)}  - \thestar\rVert \leq \delta_n\bigr).
\end{align*}
To apply \cref{cor:01}, we need to develop the following proposition.
\begin{proposition}
	\label{pro:z-1}
	Let  $B_{\delta} = \{\theta \in \Theta: \lVert \theta -
	\thestar\rVert \leq \delta\}$.
	Under the conditions \cref{con:b1,con:b2,con:b3},
	\begin{align}
		\IP \bigl( \hat\theta_n \in B_{\delta_n}^c\bigr)  & \leq C (D_{\Theta} +1 )^p d^p\delta_n^{-p} n^{-p/2}, \label{eq:z-t2-01}\\
		\sum_{i = 1}^n \IE \lVert \Sigma^{-1/2} \xi_i\rVert^3  & \leq C d^{{3/2}} n ,
		\label{eq:z1-w.bound}\\
		\IE \bigl\{ \lVert W \rVert  (\Delta_1 + \Delta_2)\bigr\}  & \leq C d^{{3/2}} \delta_n + C (D_{\Theta} +1 )^2d^{{5/2}}  n^{-1/2}, \label{eq:z1-wd.bound}\\
		\label{eq:z1-xid.bound}
		\sum_{i = 1}^n \sum_{j = 1}^2 \IE \bigl\{ \lVert \xi_i\rVert \bigl\lvert \Delta_j  -
\Delta_j^{(i)} \bigr\rvert\bigr\}  & \leq
C \Bigl( (D_{\Theta} +1 )^p d^{p + {1/2}} \delta_{n}^{-p + 2} n^{- \frac{p - 3}{2}}  \\
								   & \qquad \quad     + (D_{\Theta} +1 ) d^2 + (D_{\Theta} +1 ) d^{{5/2}}
							   n^{{1/2}}\delta_n\Bigr).  \nonumber
	\end{align}
\end{proposition}

By \cref{cor:01} with $O = \{ \lVert \hat\theta_n - \thestar\rVert \leq \delta_n\}$ and
\cref{pro:z-1},
\begin{equ}
	\label{eq:z-000}
	\MoveEqLeft \sup_{A \in \mathcal{A}} \bigl\lvert  \IP \bigl( \sqrt{n} \Sigma^{-1/2} \.\Psi_0 (\hat\theta_n - \thestar) \in A\bigr) - \IP \bigl(Z \in A \bigr)\bigr\rvert
	\\
	& \leq C d^{{1/2}}n^{-3/2} \sum_{i = 1}^n \IE \lVert \Sigma^{-1/2} \xi_i\rVert^3 + C \IE \{ \lVert W \rVert
	(\Delta_1 + \Delta_2)\} \\
	& \quad  + C n^{-1/2} \sum_{i = 1}^n \IE \bigl\{ \lVert \xi_i\rVert (\Delta_1 +
	\Delta_2 - \Delta_1^{(i)} - \Delta_2^{(i)})\bigr\} + \IP \bigl( \lVert \hat\theta_n -
\thestar\rVert > \delta_n\bigr) \\
	& \leq C n^{-1/2} \bigl(d^{2} + (D_{\Theta} +1 )^2 d^{{5/2}} \willignore{\color{red}D_1^{2}} \bigr) + C
	(D_{\Theta } + 1)^2d^{{5/2}} \willignore{\color{red}D_1}
	\delta_n \\
	& \quad + C (D_{\Theta} +1 )^p d^{p + {1/2}} \willignore{\color{red}D_1^p} \delta_n^{-p + 2}
	n^{-(p-2)/2} + C (D_{\Theta} +1 )^p d^p \willignore{\color{red}D_1^p} \delta_n^{-p} n^{-p/2}.
\end{equ}
Recall that $p \geq 3$, and then $\delta_n^2 n \geq 1$. Therefore,
\begin{align*}
	\text{RHS of \cref{eq:z-000}}
	& \leq C n^{-1/2} (D_{\Theta} +1 )^2 d^{{5/2}} \willignore{\color{red}D_1^{2}}
	 + C (D_{\Theta} +1 )^2 d^{{5/2}} \willignore{\color{red}D_1}
	\delta_n  \\
	& \quad + C (D_{\Theta} +1 )^p d^{p + {1/2}} \willignore{\color{red}D_1^p} \delta_n^{-p + 2} n^{-(p-2)/2}\\
	& \leq C (D_{\Theta} +1 )^2  \Bigl(d^{{5/2}} n^{-1/2}   + d^{{7/2}} 	n^{- \frac{p-2}{2(p -1 )}}\Bigr)\\
	& \leq C (D_{\Theta} +1 )^2 d^{{7/2}} n^{- {1/2} + \epsilon_p},
\end{align*}
where $\epsilon_p = 1/(2p - 2)$. This proves \cref{thm:z-1}.

It suffices to prove \cref{pro:z-1}, and we need to apply some preliminary lemmas, whose proofs are put in \cref{sec:appendix}.
\begin{lemma}
	\label{lem:z-thp}
	Let $B_\delta $ be as in \cref{pro:z-1}.
	Under the assumptions \cref{con:b1,con:b2,con:b3}, we have
\begin{gather}
	\label{eq:m-w23bound}
	\lVert W\rVert_2 \leq \lambda_2^{-1/2} c_3 d^{1/2 }, \quad \lVert W \rVert_3 \leq C \lambda_2^{-1/2} c_3 d^{1/2 }, \\
	\label{eq:Ed12}
	\mathbb{E} \Delta_1^2 \leq  C \lambda_2^{-1} c_2^2 d^2 \delta_n^2, 
\end{gather}
where 
$C > 0$ is an absolute constant. 
Moreover, 
	\begin{align}
		\label{eq:z-that}
		\IE \bigl\{\lVert \hat\theta_n - \thestar\rVert^p\bigr\} & \leq C (D_{\Theta} +1)^p d^p	n^{-p/2} ,\\
		\label{eq:z-thati}
		\IE \Bigl\{\lVert \hat\theta_n - \hat\theta_n^{(i)}\rVert^p  \mathds{1} \bigl(\hat\theta_n \in
		B_{\delta}, \hat\theta_n^{(i)} \in B_{\delta}\bigr)\Bigr\}
																 & \leq C \bigl( d^{p/2} n^{-p} + d^p n^{-p/2} \delta^{p} \bigr),
	\end{align}
	where $C > 0$ is a constant depending only on $c_2, c_3, \mu$ and $p$.

\end{lemma}

\begin{lemma}
	\label{lem:z-th2}
	We have
	\begin{align}
		 \IE \bigl\{\lVert \xi_i \rVert \lVert \hat\theta_n - \thestar\rVert^2 \mathds{1}
		 \bigl(\hat\theta_n \in B_{\delta}, \hat\theta_n^{(i)} \in B_\delta^c\bigr)\bigr\} & \leq C
		 (D_{\Theta} +1)^p
		  d^{p + {1/2}}
		\delta^{-p + 2} n^{-p/2},\label{eq:z-t2-02}\\
		\IE \bigl\{\lVert \xi_i \rVert \lVert \hat\theta_n^{(i)} - \thestar\rVert^2 \mathds{1}
	 \bigl(\hat\theta_n \in B_{\delta}^c, \hat\theta_n^{(i)} \in B_\delta\bigr)\bigr\} & \leq C
	 (D_{\Theta} + 1)^p d^{p + {1/2}} \delta^{-p + 2} n^{-p/2},
		\label{eq:z-t2-03}
	\end{align}
	and
	\begin{equ}
		\label{eq:z-t2-04}
		\MoveEqLeft
		\IE \Bigl\{\lVert \xi_i\rVert (\lVert \hat\theta_n - \thestar\rVert + \lVert
			\hat\theta_n^{(i)} - \thestar\rVert ) \lVert \hat\theta_n - \hat\theta_n^{(i)}\rVert
		\mathds{1} \bigl( \hat\theta_n \in B_{\delta}, \hat\theta_n^{(i)} \in B_\delta\bigr)\Bigr\} \\
		& \leq C (\D + 1)  \willignore{\color{red}D_1} \Bigl(d^{2}n^{-3/2} + d^{5/2 } n^{-1} \delta\Bigr) .
	\end{equ}
\end{lemma}

Now we are ready to give the proof of \cref{pro:z-1}.
\begin{proof}[Proof of \cref{pro:z-1}]

The inequality \cref{eq:z-t2-01} follows directly from \cref{eq:z-that} and the Chebyshev
inequality.

For \cref{eq:z1-w.bound}, by \cref{con:b3}, we have
\begin{align*}
	\IE  \lVert \Sigma^{-1/2} \xi_i \rVert^3  \leq \lambda_2^{-1/2 } c_3 d^{{3/2}},
\end{align*}
and thus \cref{eq:z1-w.bound} holds.

For \cref{eq:z1-wd.bound}, 
it suffices to give the bounds for the moments of $\lVert W \rVert,
\Delta_1$ and $\Delta_2$.
By \cref{eq:m-w23bound,eq:Ed12} and the Cauchy inequality, it follows that
\begin{equ}
	\label{eq:z1-wd.bound1}
	\IE \{ \lVert W \rVert   \Delta_1\} \leq C d^{{3/2}} \delta_n.
\end{equ}
Recall that $p \geq 3$, $\Delta_2 \leq c_1 \lambda_2^{-1/2}n^{1/2} \lVert \hat\theta_n - \theta_0\rVert^2 $ and by \cref{eq:z-that},
\begin{equ}
	\label{eq:m-hat3bound}
	\IE \lVert \hat\theta_n - \thestar\rVert^3 \leq C  (D_{\Theta} +1)^{3}d^{3} n^{-3/2},
\end{equ}
and then by \cref{eq:m-w23bound}, \cref{eq:m-hat3bound} and the H\"older inequality,
\begin{equ}
	\label{eq:z1-wd.bound2}
	\IE \{ \lVert W \rVert  \Delta_2\}
	\leq C n^{1/2}\ \bigl\{ \lVert W \rVert_3  \lVert \hat\theta_n - \thestar\rVert_3^2 \bigr\}
	\leq C (D_{\Theta} +1)^2d^{{5/2}}  n^{-1/2}.
\end{equ}
Combining \cref{eq:z1-wd.bound1,eq:z1-wd.bound2} yields \cref{eq:z1-wd.bound}.

It suffices to prove \cref{eq:z1-xid.bound}.
By the definition of $\Delta_1$ and $\Delta_1^{(i)}$,
\begin{align*}
	\bigl\lvert \Delta_1 - \Delta_1^{(i)} \bigr\rvert
	& \leq 2 \lambda_2^{-1/2}n^{-1/2} \sup_{\theta: \lVert \theta - \thestar\rVert \leq \delta_n} \bigl\lVert
	\myphi_{\theta} (X_i) - \myphi_{\theta} (X_i') \bigr\rVert \leq 2 d^{{1/2}} \lambda_2^{-1/2} n^{-1/2} \delta_n \lvert h_0(X)\rvert,
\end{align*}
and by \cref{eq:z-B22},
\begin{align*}
	\IE \bigl\lvert \Delta_1 - \Delta_1^{(i)}\bigr\rvert^2 \leq C  d n^{-1}\delta_n^2.
\end{align*}
Thus, by \cref{eq:z-B32} and the Cauchy inequality,
\begin{equ}
	\label{eq:z1-xid.bound1}
	\sum_{i = 1}^n \IE \Bigl\{\lVert \xi_i \rVert \bigl( \lvert \Delta_1 - \Delta_1^{(i)}\rvert\bigr)\Bigr\} \leq
	C d^{} n^{1/2} \delta_n .
\end{equ}

For $\Delta_2 - \Delta_2^{(i)}$, we have
\begin{align*}
	& \bigl\lvert \Delta_2 - \Delta_2^{(i)}\bigr\rvert \\
	& \leq c_2 \lambda_2^{-1/2}\sqrt{n} \Bigl( \lVert \hat\theta_n - \thestar \rVert^2 \mathds{1} \bigl( \lVert
	\hat\theta_n - \thestar\rVert \leq \delta_n , \lVert
\hat\theta_n^{(i)} - \thestar\rVert > \delta_n\bigr) \\
	& \qquad \qquad + \lVert \hat\theta_n^{(i)} - \thestar \rVert^2 \mathds{1} \bigl( \lVert
	\hat\theta_n - \thestar\rVert > \delta_n , \lVert
\hat\theta_n^{(i)} - \thestar\rVert \leq \delta_n\bigr) \\
	& \qquad \qquad  + (\lVert \hat\theta_n - \thestar\rVert + \lVert \hat\theta_n^{(i)} -
	\thestar\rVert) \lVert \hat\theta_n - \hat\theta_n^{(i)}\rVert \mathds{1} \bigl( \lVert
	\hat\theta_n - \thestar\rVert \leq \delta_n , \lVert
\hat\theta_n^{(i)} - \thestar\rVert \leq \delta_n\bigr)\Bigr).
\end{align*}
By \cref{lem:z-thp,lem:z-th2}, it follows that
\begin{equ}
	\label{eq:z1-xid.bound2}
	\sum_{i = 1}^n \IE \Bigl\{\lVert \xi_i \rVert \bigl( \bigl\lvert \Delta_2 -
	\Delta_2^{(i)}\bigr\rvert\bigr)\Bigr\}
	& \leq
	C (D_{\Theta} + 1)^p d^{ p + {1/2}}  \delta^{-p + 2} n^{-(p-3)/2} \\
	& \quad  + C (D_{\Theta} + 1) d^2   + C
	(D_{\Theta} + 1)d^{{5/2}} n^{1/2} \delta_n
	.
\end{equ}
Together with \cref{eq:z1-xid.bound1,eq:z1-xid.bound2}, we obtain \cref{eq:z1-xid.bound}.
\end{proof}

\subsection{Proof of \tpstring{\cref{thm:z-2}}{Theorem 3.3}}


\cref{thm:z-2} follows from the proof of \cref{thm:z-1} and the following proposition.

\begin{proposition}
	\label{pro:z2}
	Under the conditions \cref{con:b1}, \cref{con:b4} and \cref{con:b5}, we have
	\begin{align*}	
		\IP \bigl( \hat\theta_n \in B_\delta^c\bigr) & \leq C \exp \Bigl( - \frac{C' \sqrt{n}
		\delta }{ (D_{\Theta} +1 ) d^{3/2}}\Bigr), \label{eq:z-t2-01}\\
		\sum_{i = 1}^n \IE \lVert \Sigma^{-1/2} \xi_i\rVert^3 & \leq C d^{{3/2}} n \\
		\IE \bigl\{ \lVert W \rVert  (\Delta_1 + \Delta_2)\bigr\} & \leq C d^{{3/2}} \delta_n +
		C (D_{\Theta} +1 )^2
		d^{{5/2}}  n^{-1/2}, \\
		\sum_{i = 1}^n \IE \bigl\{ \lVert \xi_i\rVert \bigl\lvert \Delta_1 + \Delta_2 - \Delta_1^{(i)} -
		\Delta_2^{(i)}\bigr\rvert\bigr\} & \leq C  (D_{\Theta} +1 )^2
		d^{{5/2}}  \exp \Bigl( - \frac{C'
		\sqrt{n} \delta_n }{4 (D_{\Theta} +1 ) d^{{3/2}}}\Bigr) \\
										 & \quad  + C  (D_{\Theta} +1 )^2
			d^{2}   + C (D_{\Theta} +1 ) d^{{5/2}} n^{1/2} \delta_n,
	\end{align*}
	where $C' > 0$ is a constant depending only on $c_4,c_5$ and $\mu$ and $C > 0$ is a constant
	depending only on $c_1,c_4,c_5,\mu,\lambda_1$ and $\lambda_2$.
\end{proposition}

Similar to the proof of \cref{thm:z-1}, and by \cref{pro:z2}, we have
\begin{align*}
	\MoveEqLeft \sup_{A \in \mathcal{A}} \bigl\lvert  \IP \bigl( \sqrt{n} \Sigma^{-1/2} \.\Psi_0 (\hat\theta_n - \thestar) \in A\bigr) - \IP \bigl(Z \in A \bigr)\bigr\rvert \\
	& \leq C (D_{\Theta} +1 )^2 d^{{5/2}} n^{-1/2}
	+ C (D_{\Theta} +1 ) d^{{5/2}} \delta_n + C (D_{\Theta} +1 )^2 d^{{5/2}}\exp \Bigl( - \frac{C' \sqrt{n}
	\delta_n}{ 4 (D_{\Theta} +1 ) d^{{3/2}} }\Bigr).
\end{align*}
Choosing $\delta_n = {(C')^{-1} (D_{\Theta} + 1)
	d^{{3/2}} n^{-1/2} \log n}$, we completes the proof of \cref{thm:z-2}. It suffices to prove \cref{pro:z2}.

The following lemma is a modification of \cref{lem:z-thp,lem:z-th2}, whose proof is given in \cref{sec:appendix}.
\begin{lemma}
	\label{lem:z2-theta}
	Let $B_\delta$ be as in \cref{pro:z-1}.
	Under the assumptions \cref{con:b1}, \cref{con:b4} and \cref{con:b5}, we have
	\cref{eq:z-that,eq:z-thati} hold for each $p \geq 1$ with a positive constant $C$ depending on
	$c_1, c_4,c_5,\mu,\lambda_1, \lambda_2$ and $p$. Moreover, we have
	there exists a
	constant  $C' > 0$ depending only on $c_4,c_5$ and $\mu$ such that
	\begin{equ}
		\label{eq:z2-tail.prob}
		\IP \bigl( \lVert \hat\theta_n - \thestar\rVert > t\bigr) \leq 2 \exp \Bigl( - \frac{C'
		n^{1/2} t}{  (D_{\Theta} +1 ) d^{3/2} }\Bigr), \quad \text{ for  } t >0,
	\end{equ}
	and
		\begin{align}
		 \IE \Bigl\{\lVert \xi_i \rVert \lVert \hat\theta_n - \thestar\rVert^2 \mathds{1}
		\bigl(\hat\theta_n \in B_{\delta}, \hat\theta_n^{(i)} \in B_\delta^c\bigr)\Bigr\} &\leq C
	(D_{\Theta} +1 )^2 d^{{5/2}}
	n^{-1} \exp \Bigl( - \frac{C' n^{1/2} \delta}{4 (D_{\Theta} +1 ) d^{3/2}} \Bigr),
		\label{eq:z2-t2-02}\\	
		 \IE \Bigl\{\lVert \xi_i \rVert \lVert \hat\theta_n^{(i)} - \thestar\rVert^2 \mathds{1}
		\bigl(\hat\theta_n \in B_{\delta}^c, \hat\theta_n^{(i)} \in B_\delta\bigr)\Bigr\} &\leq C
	(D_{\Theta} +1 )^2 d^{{5/2}}
	n^{-1} \exp \Bigl( - \frac{C' n^{1/2} \delta}{4 (D_{\Theta} +1 ) d^{3/2}} \Bigr),
		\label{eq:z2-t2-03}
	\end{align}
	where $C > 0$ depending only on $c_1,c_4,c_5,\mu,\lambda_1$ and $\lambda_2$.

\end{lemma}

\begin{proof}
[Proof of \cref{pro:z2}]
The first inequality follows directly from \cref{eq:z2-tail.prob},
and the last three inequalities follow from \cref{eq:z2-t2-02,eq:z2-t2-03} and from the proof of \cref{pro:z-1}.
\end{proof}





\subsection{Proof of \texorpdfstring{\cref{thm:sgd}}{Theorem 3.4}}
Without loss of generality, we assume that $n \geq 4 \{(2 L \ell_0)^{\alpha} + 1\}$; otherwise, the
bound \cref{eq:sgd-a} is trivial.

In this subsection, we denote by $C, C_1, C_2, \dots$ a sequence of general positive constants depending only on  $\ell_0, \lambda_1, \lambda_2, c_1,
c_2, \alpha,\beta, L $ and $\mu$ and independent of $\tau$ and  $\tau_0$.
Let $L_1 := \max \{c_2, 2 L / \beta\}$ and $L_2 := c_1 + L. $
We introduce the following family of functions:
Let $\phi_{\beta} \colon \IR_+ \setminus \{0\} \mapsto \IR$ be given by
\begin{equ}
	\label{eq:phi.function}
	\phi_{\beta} (t) =
	\begin{cases}
		\frac{t^\beta - 1}{\beta}, & \text{ if } \beta \neq 0, \\
		\log t, & \text{ if } \beta = 0.
	\end{cases}
\end{equ}
\begin{proof}
[Proof of \cref{thm:sgd}]
		Note that $\thestar$ is the minimum point of  $f$ and by the differentiability and convexity of  $f$, we have
		$\nabla f(\thestar) = 0$.  By \cref{eq:hes},
		\begin{equ}
			\label{eq:sgd-00}
			\nabla f(\theta) = \nabla^2 f(\thestar) (\theta - \thestar ) + H (\theta),
		\end{equ}
		where
		\begin{align*}
			H(\theta) &= \nabla f(\theta) - \nabla^2 f(\thestar) (\theta - \thestar ) \\
					  &= \nabla f(\theta) - \nabla f(\thestar) - \nabla^2 f(\thestar) (\theta -
					  \thestar ) \\
					  &= \int_0^1 \{ \nabla^2 f(\thestar + t (\theta - \thestar) ) - \nabla^2 f(\thestar) \} (\theta -
					  \thestar) d t.
		\end{align*}
		By \cref{con:c2} and \cref{con:c3}, it follows that
		\begin{align*}
			\lVert H(\theta) \rVert \I ( \lVert \theta - \thestar\rVert \leq \beta) \leq c_2 \lVert
			\theta - \thestar\rVert^2,
		\end{align*}
		and
		\begin{align*}
			\lVert H(\theta) \rVert \I ( \lVert \theta - \thestar\rVert > \beta)
			& \leq 2 L \lVert \theta - \thestar\rVert \I( \lVert \theta - \thestar \rVert > \beta) \leq \frac{2 L }{\beta} \lVert \theta - \thestar\rVert^2.
		\end{align*}
		Hence, with $L_1 := \max \{ c_2 , 2 L / \beta\}$, we have
		\begin{equ}
			\label{eq:sgd-l1}
			\|H(\theta)\| \leq L_1 \lVert \theta - \thestar\rVert^2.
		\end{equ}

		Recall that $G:= \nabla^2 f(\thestar)$, and it follows from \cref{eq:saa,eq:sgd-00} that for any $n \geq 1$,
		\begin{equ}
			\label{eq:sgd.re.sys}
			\theta_n &=  \theta_{n - 1} - \ell_n \bigl(\nabla f(\theta_{n - 1}) + \zeta_n\bigr)\\
			&= \theta_{n - 1} - \ell_n \bigl(G (\theta_{n - 1} - \thestar) +
			\xi_n + \eta_{n} + H(\theta_{n - 1})\bigr).
		\end{equ}
		By definition,  $(\bar\theta_n - \thestar) = n^{-1} \sum_{i = 0}^{n - 1} (\theta_i -
		\thestar)$.
		Solving the recursive system \cref{eq:sgd.re.sys} yields
		\begin{align*}
			\sqrt{n} (\bar\theta_n - \thestar) = \frac{1}{\sqrt{n} \ell_0} Q_0 (\theta_0 -
			\thestar) +
			\frac{1}{\sqrt{n}} \sum_{i = 1}^{n-1} Q_i \bigl(\xi_i + \eta_i + H(\theta_{i - 1})\bigr),
		\end{align*}
		where
		\begin{math}
			Q_i = \ell_i \sum^{n - 1}_{j = i} \prod_{k = i + 1}^{j} (I - \ell_k G)  .
		\end{math}
		
		Recall that $\Sigma_n := n^{-1} \sum^{n -
			1}_{i=1} Q_i \Sigma_i Q_i\transpose. $
			Let
			$$T_n = n^{-{1/2}} \Sigma_n^{-{1/2}} \sum_{i = 0}^{n - 1} (\theta_n - \thestar),\quad  \zeta_i =
			\frac{1}{\sqrt{n}} \Sigma_n^{-1/2} Q_i \xi_i, \quad W_n = \sum_{i = 1}^{n - 1} \zeta_i $$ and
		\begin{align*}
			D_n & \phantom{:} = \frac{1}{\sqrt{n} \ell_0} \Sigma_{n}^{-1/2}Q_0 (\theta_0 -
			\thestar)+
			 \frac{1}{\sqrt{n}} \Sigma_n^{-1/2}\sum^{n -1}_{i=1} Q_i
			\eta_i+ \frac{1}{\sqrt{n}} \Sigma_n^{-1/2}\sum^{n -1}_{i=1}
		Q_i H(\theta_{i - 1}) \\
				& := D_{1, n} + D_{2, n} + D_{3,n}.
		\end{align*}
		It is easy to show that
		\begin{align*}
			\IE W_n = 0, \quad  \Var (W_n) = I_d.
		\end{align*}
		and
		\begin{align*}
			T_n = W_n + D_n.
		\end{align*}
		Also,
		\begin{align*}
			\lVert D_n\rVert &\leq    n^{-1/2} \ell_0^{-1} \lVert \Sigma_n^{-1/2} \rVert \cdot \|Q_0\| \cdot  \lVert \theta_0 - \thestar\rVert
			\\
				   & \quad  + n^{-1/2} \lVert \Sigma_{n}^{-1/2} \rVert
				   \biggl\lVert \sum^{n - 1}_{i=1} Q_i \eta_{i}\biggr\rVert + n^{-1/2}
				   \lVert \Sigma_n^{-1/2} \rVert
				   \sum_{i = 1}^{n - 1} \lVert Q_iH(\theta_{i - 1})\rVert^2 
		\end{align*}

		The following proposition  provides  the bounds of $Q_j$ and  $\Sigma_n^{-1}$.

		\begin{proposition}
			\label{pro:sdg-02}
			Suppose that $n \geq 4 \{(2 L \ell_0)^{\alpha} + 1\}$.
				 If $\ell_i = \ell_0 i^{-\alpha}$ with  $1/2 < \alpha \leq 1$, then there exists
				 a sequence  $(p_i)_{i \geq 1}$, and two positive
				 constants $C_1$ and $C_2$ depending on $\ell_0, \lambda_1, \lambda_2, c_1,
c_2, \alpha,\beta, L $ and $\mu$ such that for each $0 \leq i \leq n - 1$,
					\begin{equ}
						\label{eq:tibound}
							& \quad \quad \Sigma_n^{-1} \preccurlyeq C_1 I_d, \\
							&-p_i I_d  \preccurlyeq Q_i  \preccurlyeq  p_{i} I_d,
					\end{equ}
					where
					\begin{align*}
						p_{i} \leq
						\begin{cases}
							C_2, & \text{if } (\alpha = 1, \ell_0 \mu > 1) \text{ or } (\alpha \in (1/2,1)); \\
							C_2\log n , & \text{if } (\alpha = 1, \ell_0 \mu = 1).
							\end{cases}
					\end{align*}

		\end{proposition}

		Let \( (\xi_1', \dots, \xi_n' )\) be an independent copy of \( (\xi_1, \dots, \xi_n )\).
		For each $1 \leq i \leq n - 1$,  we now construct $D_{2, n}^{(i)}$ and  $D_{3,n}^{(i)}$ which are independent of
		$\xi_i$.
Firstly, for each $i$, we construct  $\thetaihat_1, \dots, \thetaihat_n$ as follows:
\begin{enumerate}[(a)]
	\item
	If $j < i$, let $\thetaihat_j = \theta_j$.
	\item
	If  $j = i$, let  $\thetaihat_j = \thetaihat_{j - 1} - \ell_j (\nabla
	f(\thetaihat_{j - 1}) + \xi_j' + \eta_j^{(i)})$,
	where $\eta_j^{(i)} = g(\thetaihat_{j - 1},
	\xi_i')$.
	\item
	If  $j > i$, let $\thetaihat_{j} = \thetaihat_{j - 1} - \ell_j (\nabla
	f(\thetaihat_{j - 1}) + \xi_j + \eta_j^{(i)})$,
	where
	$\eta_j^ {(i)} = g(\thetaihat_{j - 1}, \xi_j)$.
\end{enumerate}
Secondly, let
\begin{align*}
	D_{2,n}^{(i)} &=n^{-1/2}  \Sigma_n^{-1/2}\sum^{n - 1}_{j=1} Q_i \eta_j^{(i)} , \\
	D_{3,n}^{(i)} &=n^{-1/2} \Sigma_n^{-1/2}\sum^{n - 1}_{j=1}  Q_i H(\thetaihat_{j - 1}) .
\end{align*}
Then, we have for each \(1 \leq i \leq n - 1\), \(D_{2,n}^{(i)}\) and \(D_{3,n}^{(i)}\) is independent of \(\xi_i\).
Let
\begin{align*}
	\Delta = \Delta_1 + \Delta_2 + \Delta_3,
\end{align*}
where
\begin{math}
	 \Delta_1 =  \lVert D_{1,n}\rVert, \Delta_2 = \lVert D_{2,n}\rVert \text{ and }
	 \Delta_3 = C_1 L_1 n^{-1/2} \sum_{i = 1}^{n - 1} p_i \bigl\lVert \theta_{i - 1} - \thestar
	\bigr\rVert^2,
\end{math}
where \(C_1\) is given as in \cref{eq:sgd002}.
By \cref{eq:sgd-l1}, it follows that
\begin{math}
	\lVert D_{3,n}\rVert \leq \Delta_3.
\end{math}
Also, for each $1 \leq i \leq n - 1$, define
\begin{align*}
	\Delta_1^{(i)} & =  \lVert D_{1,n}\rVert, \\
	\Delta_2^{(i)} & =  \lVert D_{2,n}^{(i)} \rVert, \\
		\Delta_3^{(i)} & =  C_1 L_1 n^{-1/2} \sum_{j = 1}^{n - 1} p_j \bigl\lVert \thetaihat_{j - 1} - \thestar
		\bigr\rVert^2.
\end{align*}
Clearly, \(\Delta_1^{(i)}\), \(\Delta_2^{(i)}\) and \(\Delta_3^{(i)}\) are independent of \(\xi_i\) for each \(1 \leq i \leq n - 1\).
The following proposition provides the bounds of the moments for $\Delta_j$ and $\Delta_j -
\Delta_j^{(i)}$, $j = 1, 2, 3$.
		\begin{proposition}
			\label{lem:sgd01}
			We have $\Delta_1$ is independent of  $(\xi_1, \dots, \xi_n)$ and
			\begin{align*}
				\IE \clc{\Delta_1 \lVert W\rVert } & \leq C (\tau^2 +  \tau_0^2) n^{-1/2}.
			\end{align*}
			\begin{enumerate}
				\item
				For $\alpha \in (1/2,1)$,
				\begin{align*}
					\IE \clc{ \Delta_2 \lVert W \rVert } & \leq C d^{1/2}(\tau + \tau_0) n^{-\alpha/2}, \\
					\IE \clc{ \Delta_3  \lVert W\rVert } & \leq C d^{1/2} (\tau^2 + \tau_0^2)n^{-\alpha + 1/2}.
				\end{align*}
				and 	
				\begin{align*}
					\sum_{i = 1}^{n - 1} \IE \clc{ \lvert \Delta_2 - \Delta_2^{(i)}\rvert
					\lVert \zeta_i\rVert} & \leq C (\tau^2 + \tau_0^2)n^{-\alpha + 1/2},\\
					\sum_{i = 1}^{n - 1} \IE \clc{ \lvert \Delta_3 - \Delta_3^{(i)}\rvert
					\lVert \zeta_i\rVert} & \leq C (\tau^3 + \tau_0^3)n^{-\alpha/2} .
				\end{align*}
			\item For $\alpha = 1$,

\begin{align*}
	\IE \{\Delta_2 \lVert W\rVert \} & \leq
	\begin{cases}
		C d^{1/2}(\tau + \tau_0) n^{-1/2} (\log n)^{1/2}, & \ell_0 \mu > 1; \\
		C d^{1/2}(\tau + \tau_0) n^{-1/2} (\log n)^2, & \ell_0\mu = 1.
	\end{cases} \\
	\IE \{\Delta_3 \lVert W\rVert\} & \leq
	\begin{cases}
		C  d^{1/2}(\tau^2 + \tau_0^2) n^{-1/2} (\log n), & \ell_0 \mu > 1; \\
		C d^{1/2}(\tau^2 + \tau_0^2) n^{-1/2} (\log n)^{5/2}, & \ell_0 \mu = 1,
	\end{cases}\\
	\sum_{i = 1}^{n - 1} \IE \bigl\{\bigl\lvert \Delta_2 - \Delta_2^{(i)}\bigr\rvert \lVert \zeta_i\rVert\bigr\}
	& \leq C (\tau^2 + \tau_0^2) \times
	\begin{cases}
		n^{-1/2}, & \ell_0 \mu > 1; \\
		n^{-1/2} (\log n)^{5/2}, & \ell_0 \mu = 1.
	\end{cases}
\end{align*}
and
\begin{align*}
	\sum_{i = 1}^{n - 1}\IE \bigl\{\bigl\lvert \Delta_3 - \Delta_3^{(i)}\bigr\rvert \lVert \zeta_i\rVert\bigr\}
	& \leq C (\tau^3 + \tau_0^3) \times
	\begin{cases}
		n^{-1/2}, & \mu \ell_0 > 1; \\
		n^{-1/2} (\log n)^{5/2}, & \mu \ell_0 = 1.
	\end{cases}
\end{align*}

			\end{enumerate}

		\end{proposition}

		We apply \cref{n4-thm3} to prove the Berry--Esseen bound for $\sqrt{n}
		\Sigma_n^{-1/2} (\bar{\theta}_n - \thestar)$.
		

		\textbf{(1). For $1/2 < \alpha < 1$.}
		Firstly, by \cref{pro:sdg-02} and \cref{con:c1}, we have
		\begin{equ}
			\label{eq:sgd001}
			\sum_{i = 1}^{n - 1} \IE \lVert \zeta_i\rVert^3 \leq C n^{-3/2}\sum_{i = 1}^{n - 1} \IE
			\lVert \xi_i\rVert^3 \leq C n^{-1/2} \tau^3.
		\end{equ}
		By \cref{lem:sgd01}, we have
		\begin{equ}
			\label{eq:sgd002}
			&\IE \{\|W\|  \Delta \}\leq C (d^{3/2} + \tau^{3} + \tau_0^{3})n^{-\alpha + 1/2}, \\
			&\sum_{i = 1}^{n - 1} \IE \{\lVert \zeta_i\rVert \cdot \lvert \Delta - \Delta^{(i)}\rvert\} \leq
			C (d^{3/2} + \tau^{3} + \tau_0^{3})n^{-\alpha/2}.
		\end{equ}
		Substituting  \cref{eq:sgd001,eq:sgd002} to \cref{n4-thm3} yields \cref{eq:sgd-a}.

		\textbf{(2). For $\alpha = 1$.}
		By the definition of $\zeta_i$ and by \cref{eq:tibound},
		\begin{equ}
			\label{eq:sgd0011}
			\sum_{i = 1}^{n - 1} \IE \lVert \zeta_i\rVert^3 & \leq C n^{-3/2}\sum_{i = 1}^{n - 1} p_i^3 \IE
			\lVert \xi_i\rVert^3 \\
			& \leq
			\begin{cases}
				C \tau^{3} n^{-1/2} , & \text{if }  \ell_0 \mu >
				1, \\
				C \tau^3 n^{-1/2} (\log n)^3, & \text{if } \ell_0 \mu = 1.
			\end{cases}
		\end{equ}
		By \cref{lem:sgd01}, we have
		\begin{equ}
			\label{eq:sgd003}
			\IE \{\Delta \lVert W\rVert \} & \leq
			C (d + \tau^2 + \tau_0^2) n^{-1/2} \\& \quad   + C (d^{3/2} + \tau^3 + \tau_0^3) \times
			\begin{cases}
				n^{-1/2} (\log n) , & \ell_0 \mu > 1; \\
				n^{-1/2} (\log n)^{5/2}, & \ell_0 \mu = 1.
			\end{cases}
		\end{equ}
		Define $\Delta^{(i)} = \Delta_1 + \Delta_2^{(i)} + \Delta_3^{(i)}$, then we have
		$\Delta^{(i)}$ is independent of  $\zeta_i$. Also, $\Delta - \Delta^{(i)} = \Delta_2 -
		\Delta_2^{(i)} + \Delta_3 - \Delta_3^{(i)}$.  By  \cref{lem:sgd01}, we have
		\begin{equ}
			\label{eq:sgd004}
			& \sum_{i = 1}^{n - 1} \IE \{ \lvert \Delta - \Delta^{(i)}\rvert \lVert \zeta_i\rVert\}\\
			& \leq C (d^{3/2} + \tau^3 + \tau_0^3) \times
			\begin{cases}
				n^{-1/2}, & \ell_0 \mu > 1; \\
				n^{-1/2} (\log n)^{5/2}, & \ell_0 \mu = 1.
			\end{cases}
		\end{equ}
		Then the bound \cref{eq:sgd-b} follows from \cref{n4-thm3} and \cref{eq:sgd0011,eq:sgd003,eq:sgd004}.
\end{proof}


Now we are ready to give the proofs of \cref{lem:sgd01,pro:sdg-02}.
We first prove \cref{lem:sgd01}.
To prove \cref{lem:sgd01}, we need to apply some preliminary lemmas, which provide the bounds for
$\IE \lVert \theta_n - \thestar\rVert^2, \IE \lVert \theta_n - \theta_n^{(i)}\rVert^2$ and $\IE
\lVert \theta_n - \theta_0\rVert^4$. The proofs of the lemmas can be found in \cref{sec:appendix}.
\begin{lemma}
	\label{lem:b1}
	For $\alpha \in (1/2, 1)$, we have
	\begin{equ}
		\label{eq:lb1-1}
		\IE \lVert \theta_n - \thestar\rVert^2 \leq C n^{-\alpha} (\tau^2 + \tau_0^2), \quad \text{for  } n \geq 1.
	\end{equ}
	For $\alpha = 1$, we have
	\begin{equ}
		\label{eq:lb1-2}
		\IE \lVert \theta_n - \thestar\rVert^2
			\leq
			\begin{cases}
				C n^{-1} (\tau^2 + \tau_0^2), & \mu \ell_0 > 1, \\
				C n^{-1}(\log n )(\tau^2 + \tau_0^2), & \mu \ell_0 = 1, \\
				C n^{-\mu \ell_0}(\tau^2 + \tau_0^2), & 0 < \mu \ell_0 < 1.
			\end{cases}
	\end{equ}
\end{lemma}
\begin{lemma}
	\label{lem:lb2}
	For $\alpha \in (1/2, 1)$, we have
	\begin{align}
		\IE \lVert \theta_j - \thetaihat_j\rVert^2
	& \leq C (\tau^2 + \tau_0^2)i^{-2\alpha} \exp \Bigl\{  - \mu
	\bigl(\phi_{1 - \alpha}(j) - \phi_{1 - \alpha} (i) \bigr)\Bigr\} \label{eq:lb2-01}\\
   & \leq
   C(\tau^2 + \tau_0^2) j^{-2\alpha}. \label{eq:lb2-02}
	\end{align}
	For $\alpha = 1$, we have
	\begin{equ}
		\label{eq:lb2-03}
		\IE \lVert \theta_j - \thetaihat_j\rVert^2 \leq C (\tau^2 + \tau_0^2) i^{-2} \Bigl(\frac{i}{j}\Bigr)^{2\mu \ell_0} .
	\end{equ}
	Here, $\phi_{1 - \alpha}$ is as given in \cref{eq:phi.function}.
\end{lemma}

\begin{lemma}
	\label{lem:lb3}
	For $\alpha \in (0, 1)$,
\begin{equ}
	\label{eq:lb3-01}
	\IE \lVert \theta_{j } - \thestar \rVert^4 \leq C j^{-2 \alpha} (\tau^4 + \tau_0^4).
\end{equ}
For  $\alpha = 1$,
\begin{equ}
	\label{eq:lb3-02}
	\IE \lVert \theta_j - \thestar\rVert^4
	& \leq
	\begin{dcases}
		C j^{-2   } , & \ell_0 \mu >1 , \\
		C j^{-2   } \log j, & \ell_0 \mu = 1.
	\end{dcases}
\end{equ}
\end{lemma}

\begin{proof}[Proof of \cref{lem:sgd01}]
Recall that we assume that $n \geq 4 \{(2 L \ell_0)^{\alpha} + 1\}$.
Now we consider the following two cases.

\textbf{1. If $\alpha \in (1/2, 1)$.}
First, by \cref{pro:sdg-02}, $  \Sigma_n^{-1} \preccurlyeq C_1 I_d, Q_j \preccurlyeq C_2 I_d$ for each $0
\leq j \leq n-1$ and $n \geq  4 \{(2 L \ell_0)^{\alpha} + 1\}$.
For $\Delta_1$, by \cref{con:c0}, we have
\begin{align*}
	\IE \Delta_1^2 \leq C n^{-1} \IE \lVert \theta_0 - \thestar\rVert^2 \leq C \tau_0^2 n^{-1}.
\end{align*}
By the Cauchy inequality and noting that $\IE \{WW\transpose\} = I_d$, we have
\begin{align*}
	\IE \{\Delta_1 \lVert W\rVert \} \leq C d^{1/2} \tau_0 n^{-1/2}.
\end{align*}
Recall that by \cref{con:c1}, $(\eta_j)_{j \geq 1}$ is
a martingale difference sequence and  $\|\eta_j\| \leq c_1 \lVert \theta_{j - 1} - \thestar
\rVert$, and then by \cref{eq:tibound,eq:lb1-1}, if $\alpha \in (1/2,1)$,
\begin{align*}
	\begin{split}
		\IE \Delta_2^2 & \leq \lambda_2^{-1 }\IE \Bigl\lVert \frac{1}{\sqrt{n}} \sum^{n - 1}_{i=1} \Sigma_{n}^{-1/2}Q_i
		\eta_i\Bigr\rVert^2 \leq Cn^{-1} \sum^{n - 1}_{i=1}  \IE \lVert \eta_i\rVert^2 \\
		& \leq C n^{-1} \sum^{n - 1}_{i=1} \IE \lVert \theta_{i-1} - \thestar\rVert^2  \\
		& \leq C n^{-1} (\tau^2 + \tau_0^2) \sum^{n - 1}_{i=1} i^{-\alpha } \\
		& \leq C n^{-\alpha} (\tau^2 + \tau_0^2).
	\end{split}
\end{align*}
Recall that $\IE WW\transpose = I_d$ and thus \(\IE \lVert W \rVert^2 \leq C d\), then by the Cauchy inequality again,
\begin{align*}
	\IE \{\Delta_2 \lVert W\rVert\} \leq C d^{1/2} (\tau + \tau_0) n^{-\alpha/2}.
\end{align*}

For $\Delta_3$, by \cref{pro:sdg-02,lem:lb3},
\begin{equ}
	\label{eq:sgd-delta3}
	\IE \{\Delta_3 \lVert W\rVert \}
	& \leq C n^{-1/2} \sum_{i = 1}^{n - 1} \IE \{\lVert \theta_{i - 1} - \thestar\rVert^2 \lVert W\rVert 
	\}	\\
	& \leq
	C d^{1/2} n^{-1/2} \sum_{i = 1}^{n - 1} \bigl(\IE \lVert \theta_{i - 1} -
	\thestar\rVert^4\bigr)^{1/2} \\
	& \leq 	C d^{1/2} n^{-1/2} \bigg( \sum_{i = 1}^{n - 2} \bigl(\IE \lVert \theta_{i} -
	\thestar\rVert^4\bigr)^{1/2}  + (\IE \lVert \theta_0 - \thestar\rVert^4)^{1/2} \bigg) \\
	& \leq C d^{1/2} n^{-1/2} (\tau^2 + \tau_0^2) \sum_{i = 1}^{n - 1} i^{-\alpha}  \\
	& \leq  C d^{1/2} n^{-\alpha + 1/2} (\tau^2 + \tau_0^2).
\end{equ}
Now we move to give the bounds of $\IE \{\lvert \Delta_2- \Delta_2^{(i)}\rvert \lVert \zeta_i\rVert\}$ and
$\IE \{\lvert \Delta_2- \Delta_2^{(i)}\rvert \lVert \zeta_i\rVert\}$.
For $\lVert \zeta_i\rVert$, by \cref{con:c1} and  \cref{pro:sdg-02}, we have
\begin{equ}
	\label{eq:sgd-txi}
	\IE \lVert \zeta_i\rVert^4 &= n^{-2}\IE \lVert \Sigma_{n}^{-1/2}Q_i \xi_i\rVert^4 \leq C n^{-2}
	\tau^4.
\end{equ}
For $\Delta_2 - \Delta_2^{(i)}$,
\begin{equ}
	\label{eq:sgd-t2i}
	\lvert \Delta_2 - \Delta_2^{(i)}\rvert
	& \leq n^{-1/2} \biggl\lVert \sum_{j = 1}^{n - 1} \Sigma_n^{-1} Q_j (\eta_j -
	\eta_j^{(i)})\biggr\rVert,
\end{equ}
and
\begin{align*}
	\eta_j - \eta_j^{(i)}
	=
	\begin{cases}
		0, & j < i; \\
		g(\theta_{j - 1} , \xi_j) - g(\theta_{j - 1}, \xi_j') , & j = i; \\
		g(\theta_{j - 1}, \xi_j) - g(\theta_{j - 1}^{(i)}, \xi_j), & j > i.
	\end{cases}
\end{align*}
By the construction of $\eta_j^{(i)}$ and by \cref{eq:etan}, for each $j$, $\eta_j=_d \eta_j^{(i)}$
and $\lVert \eta_j \rVert \leq c_1 \lVert \theta_{j - 1} - \thestar\rVert$; and for $j > i$,  $\lVert \eta_j -
\eta_j^{(i)}\rVert \leq  c_1 \lVert \theta_{j - 1} - \theta_{j - 1}^{(i)}\rVert$. Set
\begin{align*}
	\mathcal{F}_j^{(i)}
	=
	\begin{cases}
		\mathcal{F}_j, & j < i; \\
		\mathcal{F}_j \bigvee \sigma(\xi_i'), & j \geq i.
	\end{cases}
\end{align*}
Then $(\eta_j - \eta_j^{(i)})_{j \geq 1}$ is a martingale difference sequence with respect to
$\mathcal{F}_j^{(i)}$.
Hence, by \cref{eq:sgd-t2i,eq:tibound},
\begin{equ}
	\label{eq:sgd-delta2i}
	\IE \bigl\lvert \Delta_2 - \Delta_2^{(i)}\bigr\rvert^2
	& \leq 2 n^{-1}  \IE \lVert \eta_i - \eta_i^{(i)}\rVert^2 + 2 \IE \Bigl\lVert
	\frac{1}{\sqrt{n}} \sum_{j = i + 1}^{n - 1} Q_j (\eta_j - \eta_j^{(i)})\Bigr\rVert^2 \\
	& \leq 4 n^{-1} \IE \lVert \eta_i\rVert^2 + 2 n^{-1} \sum_{j = i + 1}^{n - 1}
	p_i^2 \IE \lVert \eta_j - \eta_{j}^{(i)}\rVert^2 \\
	  & \leq 4 c_1^2 n^{-1} \IE \lVert \theta_{i - 1} - \thestar\rVert^2   + 2 c_1^2 n^{-1} \sum^{n
	  - 1}_{j = i + 1} p_j^2 \IE \lVert \theta_{j-1} -
	  \theta_{j-1}^{(i)}\rVert^2  \\
	& \leq C (\tau^2 + \tau_0^2)n^{-1} i^{-\alpha} + C (\tau^2 + \tau_0^2)n^{-1} \bigl(\phi_{1 - 2\alpha} (n - 1) - \phi_{1 - 2\alpha}
	(i)\bigr) \\
	&\leq C (\tau^2 + \tau_0^2) n^{-1}  i^{-2\alpha + 1},
\end{equ}
where we used \cref{eq:lb1-1,eq:lb2-02} in the last second line.
By \cref{eq:sgd-txi,eq:sgd-delta2i} and the Cauchy inequality, we have
\begin{align*}
	\sum_{i = 1}^{n - 1}\IE \bigl\{ \lvert \Delta_2 - \Delta_2^{(i)}\rvert \lVert \zeta_i\rVert\bigr\}
	\leq C  (\tau^2 + \tau_0^2) n^{-\alpha + 1/2}.
\end{align*}

For $\Delta_3 - \Delta_3^{(i)}$,
by the H\"older inequality, and noting that $\theta_j \stackrel{d}{=} \theta_j^{(i)}$,
\begin{align*}
	\MoveEqLeft  \sum_{i = 1}^{n - 1} \sum_{j = 1}^{n} \IE \Bigl\{\lVert \xi_i\rVert \Bigl(\bigl\lvert \lVert
				\theta_{j - 1} - \thestar \rVert^2 - \lVert \hat\theta^{(i)}_{j - 1} - \thestar
	\rVert^2\bigr\rvert\Bigr) \Bigr\}\\
	& \leq \sum_{i = 1}^{n - 1} \sum_{j = 0}^{n-1} \IE \bigl\{\lVert \xi_i\rVert \bigl( \lVert \theta_j -
	\thestar\rVert \cdot  \lVert \theta_j - \hat\theta^{(i)}_j \rVert\bigr)\bigr\} \\
	& \quad  + \sum_{i = 1}^{n - 1} \sum_{j = 0}^{n-1} \IE \bigl\{\lVert \xi_i\rVert \bigl( \lVert
			\theta_j^{(i)} -
	\thestar\rVert \cdot  \lVert \theta_j - \hat\theta^{(i)}_j \rVert\bigr)\bigr\} \\
	& \leq 2 \sum_{i = 1}^{n - 1} \sum_{j = 0}^{n-1} \bigl(\IE \|\xi_i\|^4 \bigr)^{1/4} \bigl(\IE
	\lVert \theta_j - \thestar\rVert^4\bigr)^{1/4} \bigl(\IE \lVert \theta_j -
\hat\theta_j^{(i)}\rVert^2\bigr)^{1/2}\\
	& \leq  C (\tau^3 + \tau_0^3) \sum_{i = 1}^{n - 1} \sum_{j = i}^{n} j^{-\alpha/2} i^{-\alpha} \exp \bigl\{ - C
	\bigl(j^{1 - \alpha} - i^{1 - \alpha}\bigr)\bigr\},
\end{align*}
where we used \cref{eq:lb2-01,eq:lb3-01,eq:sgd-txi} in the last line.
Observe that
\begin{align*}
	 \sum_{j = i}^{n} j^{-\alpha/2}  \exp \bigl\{ - C
	 \bigl(j^{1 - \alpha} - i^{1 - \alpha}\bigr)\bigr\} \leq C n^{\alpha/2},
\end{align*}
and it follows that
\begin{equ}
	\label{eq:sgd-delta3i}
	\sum_{i = 1}^{n - 1} \IE \{ \lvert \Delta_3 - \Delta_3^{(i)}\rvert \lVert \zeta_i\rVert \}
	& \leq C n^{-1}\sum_{i = 1}^{n - 1} \sum_{j = 0}^{n-1} \IE \bigl\{\lVert \xi_i\rVert \bigl\lvert \lVert
	\theta_j - \thestar \rVert^2 - \lVert \hat\theta^{(i)}_j - \thestar \rVert^2\bigr\rvert\bigr\}\\
	&
	\leq C (\tau^3 + \tau_0^3)n^{-\alpha/2 }.
\end{equ}

\textbf{2. If $\alpha = 1$.}
Since the bound of $\IE \{\Delta_1 \lVert W\rVert \}$ does not depend on $\alpha$, it suffices to give the bounds
of  $\IE \{\Delta_j \lVert W\rVert\}$ and $\sum_i \IE \{ \lvert \Delta_j - \Delta_j^{(i)}\rvert
\lVert \zeta_i\rVert\}$ for $j = 2, 3$.

By \cref{pro:sdg-02}, we have
\begin{math}
	\Sigma_n^{-1/2} \preccurlyeq C_2 I_d,
\end{math}
and  for $0 \leq i \leq n - 1$,
\begin{align*}
	p_i \leq
	\begin{cases}
		C , & \ell_0 \mu > 1; \\
		C (\log n) , & \ell_0 \mu = 1.
	\end{cases}
\end{align*}

For \(\Delta_2\), noting that \(\|\eta_i\| \leq c_1 \lVert \theta_i - \thestar\rVert \), by \cref{pro:sdg-02} with \(\alpha = 1\), and by \cref{eq:lb1-2}, we have
\begin{align*}
	\IE \Delta_2^2 &  \leq \lambda_2^{-1 }\IE \Bigl\lVert \frac{1}{\sqrt{n}} \sum^{n - 1}_{i=1} \Sigma_{n}^{-1/2}Q_i
	\eta_i\Bigr\rVert^2 \\
				   & \leq C n^{-1} \sum^{n - 1}_{i=1} p_i^2  \IE \lVert \eta_i\rVert^2 \\
				   & \leq
				   \begin{dcases}
					   C (\tau^2 +\tau_0^2) n^{-1} \sum_{i = 1}^{n - 1}  i^{-1} , & \text{if } \ell_0 \mu > 1, \\
					   C  (\tau^2 +\tau_0^2) n^{-1} \sum_{i = 1}^{n - 1} (\log n)^2 i^{-1} \log i , & \text{if } \ell_0 \mu = 1
				   \end{dcases}\\
				   & \leq
				   \begin{dcases}
					   C (\tau^2 + \tau_0^2 ) n^{-1} \log n , & \text{if } \ell_0 \mu >1, \\
					   C (\tau^2 + \tau_0^2) n^{-1} (\log n)^4, & \text{if } \ell_0 \mu = 1.
				   \end{dcases}
\end{align*}
Recalling that \(\IE W W \transpose = I_d\), we obtain
\begin{align*}
	\IE \{\Delta_2 \lVert W\rVert \} & \leq
	\begin{cases}
		C d^{1/2} (\tau + \tau_0) n^{-1/2} (\log n)^{1/2}, & \ell_0 \mu > 1; \\
		C d^{1/2} (\tau + \tau_0) n^{-1/2} (\log n)^2, & \ell_0\mu = 1,
	\end{cases}
\end{align*}
Similar to \cref{eq:sgd-delta3}, and by \cref{pro:sdg-02} and \cref{lem:lb3}, we have
\begin{align*}
	& \IE \{\Delta_3 \lVert W\rVert \} \\
	& \leq C n^{-1/2} \sum_{i = 1}^{n - 1} p_i \IE \{\lVert \theta_{i - 1} - \thestar\rVert^2 \lVert W\rVert 
	\}	\\
	& \leq C d^{1/2 }n^{-1/2} \sum_{i = 1}^{n - 1} p_i \Bigl(\IE \lVert \theta_{i - 1} - \thestar\rVert^4\Bigr)^{1/2}	\\
	& \leq
	\begin{cases}
		C d^{1/2} n^{-1/2} (\tau^2 + \tau_0^2)  \sum_{i = 1}^{n - 1} i^{-1} , & \ell_0 \mu > 1; \\
		C d^{1/2} n^{-1/2}\log n (\tau^2 + \tau_0^2) \sum_{i = 1}^{n - 1} i^{-1} (\log i)^{1/2}, & \ell_0 \mu = 1,
	\end{cases}\\
	& \leq
	\begin{cases}
		C d^{1/2} (\tau^2 + \tau_0^2) n^{-1/2} (\log n), & \ell_0 \mu > 1; \\
		C d^{1/2} (\tau^2 + \tau_0^2) n^{-1/2} (\log n)^{5/2}, & \ell_0 \mu = 1.
	\end{cases}
\end{align*}
Similar to \cref{eq:sgd-delta2i}, and note that \cref{eq:tibound},  we have
\begin{align*}
	\IE \bigl\lvert \Delta_2 - \Delta_2^{(i)}\bigr\rvert^2
	& \leq C (\tau^2 + \tau_0^2) \times
	\begin{cases}
		n^{-1} i^{-1}, & \ell_0 \mu > 1; \\
		n^{-1} (\log n)^2i^{-1} \log i, & \ell_0 \mu = 1.
	\end{cases}
\end{align*}
By \cref{con:c1} and \cref{pro:sdg-02},
\begin{align*}
	\IE \lVert \zeta_i\rVert^2 \leq C n^{-1}p_i^2 \IE \lVert \xi_i\rVert^2
	\leq
	\begin{dcases}
		C n^{-1} \tau^2 , & \ell_0 \mu > 1,\\
		C n^{-1} \tau^2 (\log n)^2, & \ell_0 \mu = 1.
	\end{dcases}
\end{align*}
By the Cauchy inequality,
\begin{align*}
	& \sum_{i = 1}^{n - 1} \IE \bigl\{\bigl\lvert \Delta_2 - \Delta_2^{(i)}\bigr\rvert \lVert \zeta_i\rVert\bigr\} \\
	& \leq C (\tau^2 + \tau_0^2 ) n^{-1}  \times
	\begin{dcases}
		\sum_{i = 1}^{n - 1}i^{-1/2 } , & \ell_0 \mu >1, \\
		\sum_{i = 1}^{n - 1}i^{-1/2 } (\log i)^{1/2} (\log n)^2, & \ell_0 \mu = 1,
	\end{dcases}
	\\
	& \leq C (\tau^2 + \tau_0^2) \times
	\begin{cases}
		n^{-1/2}, & \ell_0 \mu > 1; \\
		n^{-1/2} (\log n)^{5/2}, & \ell_0 \mu = 1.
	\end{cases}
\end{align*}
Similar to \cref{eq:sgd-delta3i}, and by \cref{eq:tibound}, we have
\begin{align*}
	\sum_{i = 1}^{n - 1} \IE \bigl\{\bigl\lvert \Delta_3 - \Delta_3^{(i)}\bigr\rvert \lVert
		\zeta_i\rVert\bigr\} & \leq C n^{-1/2} \sum_{i = 1}^{n - 1} \sum_{j = i}^{n} p_i^2\IE \bigl\{\lVert \xi_i\rVert \bigl\lvert \lVert
	\theta_j - \thestar \rVert^2 - \lVert \hat\theta^{(i)}_j - \thestar \rVert^2\bigr\rvert\bigr\} \\
	& \leq C (\tau^3 + \tau_0^3) \times
	\begin{cases}
		n^{-1/2}, & \mu \ell_0 > 1; \\
		n^{-1/2} (\log n)^{9/4}, & \mu \ell_0 = 1.
	\end{cases}
\end{align*}
This completes the proof.
\end{proof}

\begin{proof}
[Proof of \cref{pro:sdg-02}]
Note that $\ell_k = \ell_0 k^{-\alpha}$ is  decreasing in $k$,
and recall that $G = \nabla^2 f(\thestar)$.
By \cref{eq:str_con}, we have
\begin{math}
	\mu I_d \preccurlyeq G \preccurlyeq L I_d \text{ and } L^{-1} I_d \preccurlyeq G^{-1} \preccurlyeq
	\mu^{-1} I_d.
\end{math}
Let
$i_0 =  \lfloor (2 L \ell_0)^\alpha + 1 \rfloor.$
For $i > i_0$,
we have $\ell_0 i ^{-\alpha} \mu \leq \ell_0 i ^{-\alpha} L \leq 1/2$. Then,
\begin{equ}
	\label{eq:mat-01}
	Q_i & \succcurlyeq \ell_i \sum_{j = i}^{n - 1} \prod_{k = i + 1}^{j} (I - \ell_iG) \\
	& = \ell_i \bigl\{ I + (I - \ell_i G) + \dots + (I - \ell_i G)^{n - i - 1} \bigr\} \\
	&= G^{-1} - (I - \ell_iG)^{n - i} G^{-1} \\
	& \succcurlyeq L^{-1}\bigl\{ 1 - \bigl( 1 - \ell_i \mu\bigr)^{n - i} \bigr\}
	I_d
	\succcurlyeq 0,
\end{equ}
and for $i_0+1 \leq i \leq n/2,$
\begin{equ}
	\label{eq:mat-02}
	1 - \bigl( 1 - \ell_i \mu\bigr)^{n - i}
		& \geq  1 - \bigl( 1 - \ell_i \mu\bigr)^{n/2}
			 \geq 1 - \exp \Bigl\{- {1/2}\mu n \ell_i \Bigr\}\\
			& \geq 1 - \exp \Bigl\{- \frac{\ell_0}{2} \mu \Bigr\}\coloneqq c_G,
\end{equ}
where we used the fact that $n \ell_i \geq \ell_0$ in the last inequality.

Recall that for each $1 \leq i \leq n - 1$,  $\lambda_{\min}(\Sigma_i)\geq \lambda_1$, for any $i_0 + 1 \leq
i \leq n/2$, by \cref{con:c1}, \cref{con:c2}, \cref{eq:mat-01,eq:mat-02},
\begin{align*}
	Q_i \Sigma_i Q_i \succcurlyeq c_G^2 L^{-2} \lambda_1 I_d.
\end{align*}

By the assumption that $n \geq 4 \{(2 L \ell_0)^{\alpha} + 1\}$, it follows that $n$ is large enough such that  $i_0 \leq n/4$.
Therefore,
\begin{align*}
	\frac{1}{n} \sum_{i = 1}^{n - 1} Q_i \Sigma_i Q_i
	&= \frac{1}{n} \sum_{i = 1}^{i_0} Q_i \Sigma_i T_{i} + \frac{1}{n} \sum^{n}_{i = i_0 + 1} Q_i
	\Sigma_i Q_i \\
	& \succcurlyeq \frac{1}{n} \sum^{n}_{i = i_0 + 1} Q_i
	\Sigma_i Q_i \\
	  & \succcurlyeq
	\frac{c_G^2  \lambda_1}{4 L^2} I_d,
\end{align*}
because $Q_i \Sigma_i Q_i \succcurlyeq 0$ for each  $0 \leq i \leq i_0$.
Therefore,
\begin{align*}
	\Sigma_n^{-1} = \biggl( \frac{1}{n} \sum^{n - 1}_{i=1} Q_i \Sigma_i Q_i \biggr)^{-1} \preccurlyeq \frac{4
	 L^2}{c_G^2 \lambda_1} I_d.
\end{align*}
This proves the first inequality of \cref{eq:tibound}.

Now we move to prove the second inequality of \cref{eq:tibound}.
The following proof uses the idea of \cite[][pp. 845--846]{Pol92a}.
Write
\begin{align*}
	V_i^m = \prod_{k = i}^{m-1} \bigl(I - \ell_k G\bigr), \quad  U_i^m = (V_i^{m})\transpose
	G^{-1} V_i^{m},
\end{align*}
and it follows that
\begin{align*}
	U_i^{m + 1} &=  U_i^m - 2 \ell_m (V_i^m)\transpose (V_i^m) + \ell_m^2 U_i^m.
\end{align*}
Recall that by \cref{con:c2},
\begin{math}
	\mu I_d \preccurlyeq G \preccurlyeq L I_d,
\end{math}
and then
\begin{equ}
	\label{eq:sgd-t01}
	U_i^m \preccurlyeq \mu^{-1} (V_i^m )\transpose (V_i^m)
\\
U_i^m \succcurlyeq L^{-1} (V_i^m )\transpose (V_i^m)
\end{equ}
and
\begin{align*}
	U_i^{m + 1}  \preccurlyeq \bigl(1 - 2 \ell_m \mu + \ell_m^2\bigr) U_i^m.
\end{align*}
Therefore, for $j \geq i$, 
\begin{equ}
	\label{eq:sgd-t02}
	 U_i^j
	& \preccurlyeq \exp \biggl\{ - 2 \mu \sum_{k = i}^{j - 1} \ell_k\biggr\}
	\exp \biggl\{ \sum_{k = i}^{j - 1} \ell_k^2\biggr\}  U_i^i
	\preccurlyeq C L \exp \biggl\{ - 2 \mu \sum_{k = i}^{j - 1} \ell_k\biggr\}
	 I_d.
\end{equ}
If $\alpha \in (1/2, 1)$,
\begin{align*}
	 U_i^j
	& \preccurlyeq C \exp \biggl\{ - \mu \bclr{\phi_{1 - \alpha}(j - 1) - \phi_{1 -
	\alpha}(i - 1)}\biggr\} I_d \\
	&= C \exp \biggl\{ - \frac{\mu}{ (1 - \alpha)} \bclr{(j - 1)^{1 - \alpha} -
			(i - 1)^{1 -
	\alpha} }\biggr\} I_d.
\end{align*}
By \cref{eq:sgd-t01}, we have
\begin{equ}
	\label{eq:sgd-t3}
	V_i^j \preccurlyeq L^{1/2}   \bigl(  U_i^j\bigr)^{1/2} \preccurlyeq  C \exp \biggl\{ - \frac{\mu}{2 (1 - \alpha)} \bclr{(j - 1)^{1 - \alpha} -
			(i - 1)^{1 -
	\alpha} }\biggr\} I_d .
\end{equ}
For $\alpha \in (0,1)$, by a simple calculation,
\begin{align*}
	\sum_{j = i}^{n} \exp \biggl\{ - \frac{\mu}{2 (1 - \alpha)} \bclr{j ^{1 - \alpha} -
			i ^{1 -
	\alpha} }\biggr\} \leq C i^{\alpha},
\end{align*}
and we have
\begin{equ}
	\label{eq:sgd-t4}
	 Q_i  &\preccurlyeq  \ell_i  \sum_{j = i}^{n - 1} V_{i + 1}^{j + 1} \\
	& \preccurlyeq C \ell_i \sum_{j = i}^{n - 1} \exp \biggl\{ - \frac{\mu}{2 (1 - \alpha)} \bclr{j ^{1 - \alpha} -
			i ^{1 -
	\alpha} }\biggr\} I_d\\
	& \preccurlyeq C \ell_i i^{-\alpha } I_d \\
	& \preccurlyeq C I_d.
\end{equ}
Similarly,
\begin{math}
	Q_i \succcurlyeq -C I_{d}.
\end{math} This proves the second inequality of \cref{eq:tibound} for $\alpha \in (1/2, 1)$.

For $\alpha = 1$, by \cref{eq:sgd-t02}, and note that
\begin{math}
	\sum_{k = 1}^{\infty} \ell_k^2 \leq 2 \ell_0^2,
\end{math}
we have
\begin{align*}
	 U_i^j \preccurlyeq C \exp \left\{ - 2 \mu\bigl(\log (j - 1) - \log
	(i - 1)\bigr) \right\}I_d.
\end{align*}
Similar to \cref{eq:sgd-t3,eq:sgd-t4}, if $1 \leq i \leq n-1$,  we have
\begin{align*}
	Q_i  & \preccurlyeq C (i + 1)^{\ell_0 \mu - 1} \{ \phi_{1 - \ell_0
	\mu} (n - 1) - \phi_{1 - \ell_0
	\mu} (i + 1) \} I_d
	\preccurlyeq
	\begin{cases}
		C I_d, & \ell_0 \mu > 1,\\
		C (\log n) I_d, & \ell_0 \mu = 1.
	\end{cases}
\end{align*}
If $i = 0$, the result \cref{eq:tibound} follows from the observation that
\begin{math}
	Q_0 = \ell_0 I_d + \ell_0 (1 - \ell_1 G) Q_1.
\end{math}
This completes the proof of the upper bound. The lower bound  can be shown
similarly.
\end{proof}


\appendix 

\section{Proofs of some lemmas in Section 5}
\label{sec:appendix}
In the appendix, we give the proofs of some lemmas in Section 5. 
\subsection{Preliminary lemmas}%
To begin with, we introduce some preliminary lemmas. 
The first lemma provides a moment inequality for sums of independent random vectors.
\begin{lemma}
	\label{lem:m-W}
	Let $\zeta_1, \dots, \zeta_n \in \IR^d$ be mean-zero independent random vectors and $S_n =
	\sum_{i = 1}^n  \zeta_i / \sqrt{n}$.
	Assume that $\max_{1 \leq i \leq n}\lVert \zeta_i\rVert_p \leq a_1$ for some $p \geq 2$ and  $a_1 > 0$. Then,
	\begin{equ}
		\label{eq:w.p.norm}
		\lVert S_n \rVert_p \leq C  a_1, 	
	\end{equ}
	where $C > 0$ is a constant depending only on $p$.
	Let \(\lVert \cdot \rVert_{\psi }\) be the Orlicz norm defined as in \cref{eq:psi-norm-1} and \(\psi_1 = e^{x } - 1\).
	Assume that $\max_{1 \leq i \leq n}\lVert \zeta_i\rVert_{\psi_1} \leq a_2$ for some constant $a_2
	> 0$. Then,
	\begin{equ}
		\label{eq:w.phi.norm}
		\lVert S_n \rVert_{\psi_1} \leq C a_2.
	\end{equ}
\end{lemma}
\begin{proof}
	Noting that \(p \geq 2\), by the H\"older inequality,
	\begin{align*}
		\lVert \zeta_i\rVert_2 \leq \lVert \zeta_i\rVert_p \leq a_1,
	\end{align*}
	and then
	\begin{align*}
		\lVert S_n \rVert_1 \leq \lVert S_n \rVert_2 = \biggl( \frac{1}{n} \sum_{i = 1}^n
		\lVert\zeta_i\rVert_2^2 \biggr)^{1/2} \leq
		 a_1.
	\end{align*}
	By the Hoffmann-Jørgensen inequality (see \cite[Theorem 1]{Tal89}), we have there exists a constant
	$C_1 > 0$ depending only on  $p$ such that
	\begin{align}
		\lVert S_n\rVert_p & \leq C_1 \Bigl( \lVert S_n \rVert_1 + n^{-1/2} \Bigl\lVert \max_{1 \leq i \leq n}\lVert \zeta_i\rVert\Bigr\rVert_p\Bigr) .
	    \label{eq:l5.1-aa}
	\end{align}
	By the formula of integration by part,
	\begin{equ}
		\Bigl\lVert \max_{1 \leq i \leq n}\lVert \zeta_i\rVert\Bigr\rVert_p^p
		& = \int_0^{\infty} \mathbb{P} \Bigl( \max_{1 \leq i \leq n} \lVert \zeta_i \rVert^p \geq t \Bigr) dt \\
		& \leq \sum_{i = 1}^n \int_0^{\infty} \mathbb{P} \Bigl(  \lVert \zeta_i \rVert^p \geq t \Bigr) dt \\
		& = \sum_{i = 1}^n \lVert \zeta_i \rVert_p^p \\
		& \leq n \max_{1 \leq i \leq n} \lVert \zeta_i \rVert_p^p.
	    \label{eq:l5.1-bb}
	\end{equ}
	Substituting \cref{eq:l5.1-bb} to \cref{eq:l5.1-aa}, we have there exist two constants $C_2$ and $C_3$ depending only on $p$ such that
	\begin{align*}	
		 \lVert S_n\rVert_p
		 & \leq C_2 \Bigl(  a_1 +  n^{-1/2 + 1/p} \max_{1 \leq i \leq n} \lVert \zeta_i\rVert_p \Bigr) \leq C_3 a_1.
	\end{align*}
	This proves \cref{eq:w.p.norm}.
	Note that $\lVert \zeta_i\rVert_2 \leq \lVert \zeta_i\rVert_{\psi_1}$, and then
	by a similar argument and the Hoffmann-Jørgensen inequality for the $\lVert \cdot\rVert_{\psi_1}$
	norm (see \cite[Theorem 3]{Tal89}), the inequality \cref{eq:w.phi.norm} holds.	
\end{proof}

\indent 
We next introduce some notations of empirical process theory, following \cite{van96}.
For any function class $\mathcal{F}$, write
\begin{equ}
	\label{eq:m-GG}
	\lVert \mathbb{G}_n\rVert_{\mathcal{F}} &= \sup_{f \in \mathcal{F}} \bigl\lvert \mathbb{G}_n f
	\bigr\rvert, \quad  \lVert \mathbb{P}_n - P \rVert_{\mathcal{F}} = \sup_{f \in \mathcal{F}}
	\bigl\lvert \mathbb{P}_n f - P f\bigr\rvert.
\end{equ}
By \cref{eq:m-p}, it is easy to see $\lVert \mathbb{G}_n\rVert_{\mathcal{F}} = \sqrt{n} \lVert
\mathbb{P}_n - P\rVert_{\mathcal{F}}$.
Let
\begin{align*}
	\mathcal{M}_{\delta} \coloneqq \bigl\{ m_{\theta} - m_{\thestar} \, : \, \lVert \theta -
	\thestar\rVert \leq \delta\bigr\}.
\end{align*}
Then,
\begin{align*}
	\lVert \mathbb{P}_n - P\rVert_{\mathcal{M}_{\delta}} = \sup_{\theta : \lVert \theta -
		\thestar\rVert \leq \delta} \bigl\lvert (\mathbb{M}_n - M)(\theta) - (\mathbb{M}_n -
	M)(\thestar)\bigr\rvert, \\
	\lVert \mathbb{G}_n \rVert_{\mathcal{M}_{\delta}} = \sup_{\theta : \lVert \theta -
		\thestar\rVert \leq \delta} \sqrt{n}\bigl\lvert (\mathbb{M}_n - M)(\theta) - (\mathbb{M}_n -
	M)(\thestar)\bigr\rvert.
\end{align*}
Note that $ \lVert \mathbb{G}_n\rVert_{\mathcal{M}_{\delta}} $ may not be measurable, and we need to
consider its {\it outer expectation}, see \cite[Section 1.2]{van96} for a thorough reference. Let $\IE^*$ be
the outer expectation operator
and for any map $Y$,
\begin{align*}
	\lVert Y \rVert_p^* = \bigl(\IE^* \{|Y|^p\}\bigr)^{1/p}, \quad
	\lVert Y \rVert_{\psi}^* = \inf \Bigl\{C > 0: \IE^* \Bigl\{ \psi \Bigl( \frac{|Y|}{C}\Bigr)
	\Bigr\} \leq 1\Bigr\}.
\end{align*}
Let $\mathbb{P}^*$ be the outer probability operator. 

The next lemma provides a bound on the bounds of  $\lVert \lVert
\mathbb{G}_n\rVert_{\mathcal{M}_\delta}\rVert_q^*$ and $\lVert \lVert
\mathbb{G}_n\rVert_{\mathcal{M}_\delta}\rVert_{\psi_1}^*$, the
proof is based on the empirical process theory.
\begin{lemma}
	\label{lem:G}
	For $q \geq 2$, assume that {\cref{eq:m-A12}} is satisfied with  $\lVert m_1(X)\rVert_q \leq
	a_3$ for a positive constant $a_3$.
	Then, we have
	\begin{equ}
		\label{eq:G.p.norm}
		 \Bigl\lVert \lVert \mathbb{G}_n\rVert_{\mathcal{M}_\delta}\Bigr\rVert_{q}^*
		\leq C a_3\sqrt{d} \delta,
	\end{equ}
	where $C > 0$ is a constant depending only on  $q$.
	Assume further that there exists a constant $a_4 > 0$ such that
	\begin{math}
		\lVert m_1(X) \rVert_{\psi_1} \leq a_4,
	\end{math}
	then we have
	\begin{equ}
		\label{eq:G.psi.norm}
		 \Bigl\lVert \lVert \mathbb{G}_n\rVert_{\mathcal{M}_\delta}\Bigr\rVert_{\psi_1}^*
		\leq C a_4\sqrt{d} \delta,
	\end{equ}
	where $C > 0$ is an absolute constant.
\end{lemma}
\begin{proof}
	Recall that by \cref{eq:m-A12},
	and $\lVert m_1(X)\rVert_2 \leq \lVert m_1(X)\rVert_q \leq  a_3$,
	and we have $\mathcal{M}_{\theta}$ has an envelope  $F = \delta m_1$ such that $\lVert F(X)\rVert_q \leq a_3\delta $.
	It has been shown (see, e.g., \cite[Chapters 5 and 19]{Vaa98} and \cite[Corollary 3.1]{Wel05}) that
	\begin{align*}
		\IE^* \lVert \mathbb{G}_n\rVert_{\mathcal{M}_{\delta}} \leq C a_3d^{{1/2}} \delta,
	\end{align*}
	where $C > 0$ is an absolute constant. By the Hoffmann-Jørgensen’s inequality
	(see \cite[Theorem 2.14.5]{van96}), we have
	\begin{align*}
		\IE^* \lVert \mathbb{G}_n\rVert_{\mathcal{M}_{\delta}}^q
		& \leq C \Bigl( \bigl(\IE^*\lVert \mathbb{G}_n\rVert_{\mathcal{M}_{\delta}} \bigr)^q + \IE |F|^q
		\Bigr) \leq C a_3^q d^{\frac{q}{2}} \delta^q,
	\end{align*}
	where $C > 0$ is a constant depending only on  $q$.
	This proves \cref{eq:G.p.norm}.
	Inequality \cref{eq:G.psi.norm} follows from a similar argument.
\end{proof}

\begin{lemma}
	\label{lem:z2-Psi}
	Let $\mathbb{G}_n $ be as in \cref{eq:m-GG} and let
	\begin{math}
		\mathcal{F}_{\delta,j} = \bigl\{ h_{\theta,j} - h_{\thestar,j} \colon \lVert \theta -
		\thestar\rVert \leq \delta \bigr\}
	\end{math}
	for $1 \leq j \leq n$.
	\begin{enumerate}
		\item [\textup{(i)}]
		Assume that there exists a constant $a_5 > 0$ such that \cref{eq:z-B21} is satisfied with
		$$\lVert h_0(X) \rVert_p\leq a_5,$$ then
		\begin{equ}
			\label{eq:z2-psi01}
			\Bigl\lVert  \lVert \mathbb{G}\rVert_{\mathcal{F}_{\delta,j}}\Bigr\rVert_p^* \leq C
			a_5
			d^{{1/2}} \delta,
		\end{equ}
		and
		\begin{equ}
			\label{eq:z2-psi02}
			\biggl\lVert \sup_{ \lVert \theta - \thestar\rVert \leq \delta} \lVert \Psi_n(\theta) -
			\Psi(\theta) - \Psi_n(\thestar) + \Psi(\thestar) \rVert \biggr\rVert_p^* \leq C n^{-1/2}
			a_5 d \delta,
		\end{equ}	
		where $C > 0$ is a constant depending only on $p$.
	\item [\textup{(ii)}] Assume that \cref{eq:z-B21} is satisfied with $ \lVert h_0(X)\rVert_{\psi_1} \leq
		a_5$, then \cref{eq:z2-psi01,eq:z2-psi02} hold for
		all $p \geq 1$, and
		\begin{equ}
			\label{eq:z2-psi.psinorm}
			\biggl\lVert \sup_{ \lVert \theta - \thestar\rVert \leq \delta} \lVert \Psi_n(\theta) -
			\Psi(\theta) - \Psi_n(\thestar) + \Psi(\thestar) \rVert \biggr\rVert_{\psi_1}^*
			\leq C n^{-1/2}a_5 d^{{3/2}} \delta,
		\end{equ}
		where $C > 0$ is an absolute constant.
\end{enumerate}
\end{lemma}
\begin{proof}	
	For (i), note that \cref{eq:z2-psi01} follows directly from \cref{lem:G}, \cref{eq:z-B21,eq:z-B22}.
	For \cref{eq:z2-psi02},  by \cref{eq:z2-psi01} and the definitions of $h_{\theta}$ and  $h_{\theta,j}$,  $j = 1, 2, \dots, n$,
	\begin{align*}
		\MoveEqLeft
		\IE^*\biggl\lVert \sup_{ \lVert \theta - \thestar\rVert \leq \delta} \lVert \Psi_n(\theta) - \Psi(\theta) - \Psi_n(\thestar) + \Psi(\thestar) \rVert \biggr\rVert^p \\
		& \leq d^{p/2 - 1} n^{-p/2} \sum_{j = 1}^d \IE^*\Bigl\lVert  \lVert
		\mathbb{G}\rVert_{\mathcal{F}_{\delta,j}}\Bigr\rVert^p \\
		& \leq C n^{-p/2} a_5^p d^{p} \delta^p ,
	\end{align*}
	which proves \cref{eq:z2-psi02}.
	
	Observe that for any  $Y \in \IR^d$ and $p \geq 1$,
	\begin{align*}
		\lVert Y \rVert_p \leq C \lVert Y\rVert_{\psi_1}.
	\end{align*}
	Hence, under the conditions \cref{con:b1}, \cref{con:b4} and \cref{con:b5}, inequalities
	\cref{eq:z2-psi01,eq:z2-psi02} follow by a similar arguments.
	
	For \cref{eq:z2-psi.psinorm}, by \cite[Theorem 2.14.5]{van96}, it follows from
	\cref{eq:z2-psi01,eq:z-B42,eq:z-B21} that
		\begin{align*}
			\Bigl\lVert \lVert \mathbb{G}\rVert_{\mathcal{F}_{\delta, j}}\Bigr\rVert^*_{\psi_1}
			& \leq C \Bigl( \Bigl\lVert  \lVert \mathbb{G}\rVert_{\mathcal{F}_{\delta,j}}\Bigr\rVert^*_1
					+ \delta \lVert h_0(X)\rVert_{\psi_1}\Bigr) \\
			& \leq C a_5 d^{{1/2}} \delta.
		\end{align*}
		By the triangle inequality,
		\begin{align*}
			\MoveEqLeft
			\biggl\lVert \sup_{ \lVert \theta - \thestar\rVert \leq \delta} \lVert \Psi_n(\theta) -
				\Psi(\theta) - \Psi_n(\thestar) + \Psi(\thestar) \rVert \biggr\rVert_{\psi_1}^* \leq
				n^{-1/2} \sum_{j = 1}^d \Bigl\lVert \lVert \mathbb{G}\rVert_{\mathcal{F}_{\delta,
				j}}\Bigr\rVert^*_{\psi_1} \leq C n^{-1/2}a_5 d^{3/2} \delta,
		\end{align*}
		and hence \cref{eq:z2-psi.psinorm} holds.
\end{proof}
\subsection{Proof of  \tpstring{\cref{lem:m-t4}}{Lemma mt4}}%
\label{sub:proof_of_lemma_5_2}

For each $n$, let
\begin{align*}
	A_{j,n} = \{\theta: 2^{j-1} < \sqrt{n} \lVert \theta - \thestar\rVert \leq 2^{j}\}, \quad  j
	\geq 1.
\end{align*}
Recall that $\hat\theta_n$ truly minimizes  $\mathbb{M}_n(\theta)$, and thus,
\begin{equ}
	\label{eq:m-th-1}
	\IE^* \bigl\{\sqrt{n} \lVert \hat\theta_n - \thestar\rVert\bigr\}^p
	& \leq \sum_{j \geq 1} 2^{jp} \IP^* \bigl(\hat\theta_n \in A_{j, n}\bigr)\\
	& \leq \sum_{j \geq 1} 2^{jp} \IP^* \Bigl( \inf_{\theta \in A_{j,n}} \bigl(\mathbb{M}_n(\theta) - \mathbb{M}_{n}(\thestar) \bigr) \leq 0 \Bigr).
\end{equ}
By \cref{eq:m-A11},  we have
\begin{align*}
	\inf_{\theta \in A_{j,n}} \bigl(M(\theta) - M(\thestar)\bigr) \geq \mu \inf_{\theta \in
	A_{j,n}}\lVert \theta - \thestar\rVert^2 \geq \mu n^{-1} 2^{2j - 2}.
\end{align*}
Set  $\delta_j = 2^j n^{-1/2}$. By \cref{lem:G} and the
Chebyshev inequality, we have
\begin{align*}
	\text{RHS of }\cref{eq:m-th-1}
	& \leq \sum_{j \geq 1} 2^{jp} \IP^* \biggl( \lVert \mathbb{G}_n\rVert_{\mathcal{M}_{\delta_j}} \geq \frac{\mu 2^{2j - 2}}{\sqrt{n}} \biggr) \\
	  & \leq \Bigl(\frac{4}{\mu}\Bigr)^{p + 1} n^{\frac{p + 1}{2}} \sum_{j \geq 1} 2^{-j(p + 2)} \IE^* \Bigl\{ \lVert
	\mathbb{G}_n\rVert_{M_{\delta_j}}^{p + 1}\Bigr\} \\
	  & \leq C \Bigl(\frac{a_6 d^{1/2 }}{\mu}\Bigr)^{p + 1} n^{\frac{p + 1}{2}} \sum_{j \geq 1} 2^{-j(p + 2)} \delta_j^{p + 1 } \\
	  & \leq C \Bigl( \frac{a_6}{\mu} \Bigr)^{p + 1}d^{\frac{p + 1}{2}} ,
\end{align*}
where $C > 0$ depends only on   $p$.
This completes the proof.

\subsection{Proof of  \tpstring{\cref{lem:m-h12}}{Lemma mh12}}%
\label{sub:proof_of_lemma_5_3}

Write
\begin{math}
	Y_i = \ddot{m}_{\thestar} (X_i) - \IE \{\ddot{m}_{\thestar} (X_i) \}.
\end{math}
By the Rosenthal inequality for random
matrices (see, e.g., \cite[Theorem A.1]{Che12a}), and
noting that $Y_i$'s are symmetric  $(d\times d)$-matrices satisfying \cref{eq:m-A14},  
\begin{align*}
	\IE \biggl\lVert \sum_{i = 1}^n Y_i \biggr\rVert^4
	& \leq C \biggl\lVert \biggl(\sum_{i = 1}^n \IE
	Y_i^2 \biggr)^{1/2}\biggr\rVert^{4} + C \IE \bigl\{\max_{1 \leq i \leq n} \lVert Y_i\rVert^4\bigr\} \\
	& \leq C n^{-2} \Bigl\lVert \bigl(\IE \bigl\{\ddot{m}_{\thestar}(X)^2\bigr\} \bigr)^{1/2} \Bigr\rVert^{4} + C n^{-3} \IE \lVert
	\ddot{m}_{\thestar}(X)\rVert^4 \\
	& \leq C n^{-2 } \lVert m_3 (X)\rVert_2^4  \times \lVert I_d \rVert^4 + C n^{-3 } \lVert m_3(X)\rVert_4^4 \times \lVert I_d \rVert^4 \\
	& \leq C n^{-2}c_3^4,
\end{align*}
where $C > 0$ is an absolute constant and we use the fact that \( \lVert I_d\rVert = 1\) in the last inequality. This proves \cref{eq:m-h18}.
For $H_2$, by \cref{eq:m-A13}, we have
\begin{align*}
	\IE \bigl\{H_2^4\bigr\} \leq \max_{1 \leq i \leq n} \IE \{m_2(X_i)^4\} \leq c_2^4.
\end{align*}
This completes the proof of \cref{eq:m-h28} and hence the lemma.
\subsection{Proof of  \tpstring{\cref{lem:m-ti}}{Lemma 54}}%
\label{sub:proof_of_lemma_5_4}

By \cref{eq:m-T,eq:m-D} and the construction of $\hat\theta_n^{(i)}$, we have
\begin{align*}
	\lVert V (\hat\theta_n - \hat\theta_n^{(i)})\rVert
	& \leq  \frac{1}{n} \bigl\lVert \xi_i - \xi_i'\bigr\rVert + \bigl(H_1 \lVert \hat\theta_n -
	\thestar\rVert + H_1^{(i)}\lVert \hat\theta_n^{(i)} - \thestar\rVert \bigr) \\
	& \quad + \bigl(H_2 \lVert
\hat\theta_n - \thestar\rVert^2 + H_2^{(i)} \lVert \hat\theta_n^{(i)} - \thestar\rVert^2\bigr),
\end{align*}
where $\xi_i' = \dot{m}_{\thestar}(X_i')$ is an independent copy of $\xi_i$.

Note that $(\xi_i, \hat\theta_n, H_1, H_2)$ has the same distribution as $(\xi_i',
\hat\theta_n^{(i)}, H_1^{(i)}, H_2^{(i)})$.
By \cref{lem:m-t4} with $p = 8$ and \cref{lem:m-h12} and the H\"older
inequality, we have
\begin{align*}
	\IE \lVert V (\hat\theta_n - \hat\theta_n^{(i)})\rVert^2
	& \leq 4 \biggl( n^{-2} \IE \lVert
	\xi_i\rVert^2 + \IE \{H_1^2 \lVert \hat\theta_n - \thestar\rVert^2 \} + \IE \{H_2^2 \lVert
\hat\theta_n - \thestar\rVert^4\}\biggr)\\
	& \leq 4 \biggl( n^{-2} \lVert \xi_i \rVert_2^2 + \lVert H_1 \rVert_4^2 \lVert \hat\theta_n - \thestar \rVert_8^2 + \lVert H_2 \rVert_4^2 \lVert \hat\theta_n - \thestar \rVert_8^4  \biggr) \\
	  & \leq C  n^{-2} \Bigl( c_4^2d + \mu^{-9/4}c_1^{9/4} c_3^2 d^{{9/8}} + \mu^{-9/2}c_1^{9/2} c_2^2 d^{{9/4}} \Bigr),
\end{align*}
where $C > 0$ is an absolute constant. The result \cref{eq:m-thi} immediately follows from the
condition that $ \lambda_{\min} (V) \geq
\lambda_2$.

\subsection{Proof of  \tpstring{\cref{lem:z-thp}}{Lemma 56}}%

The inequality \cref{eq:m-w23bound} follows from \cref{lem:m-W} and \cref{eq:z-B32},
Note that by \cref{eq:z2-psi02}, we have
\begin{align*}
	\IE \Delta_1^2 \leq \lambda_2^{-1} n \IE^* \biggl\lVert \sup_{\theta: \lVert \theta - \thestar \rVert \leq
	\delta_n } \lVert \Psi_n(\theta) - \Psi(\theta) - \Psi_n(\thestar) + \Psi(\thestar)
\rVert\biggr\rVert^2  \leq C \lambda_2^{-1} c_2^2 d^2 \delta_n^2.
\end{align*}
This proves \cref{eq:Ed12}.

By \cref{eq:z-B11}, \cref{eq:z-theta} and the Cauchy inequality, we have
\begin{align*}
	\mu \lVert \hat\theta_n - \thestar\rVert^2
	& \leq \bigl\langle \hat\theta_n - \thestar, \Psi(\hat\theta_n) - \Psi(\thestar)\bigr\rangle \\
	&= - \bigl\langle \hat\theta_n - \thestar , \Psi_n(\thestar) - \Psi(\thestar) \bigr\rangle \\
	& \quad  - \bigl\langle \hat\theta_n - \thestar,  \bigl(\Psi_n (\hat\theta_n) - \Psi(\hat\theta_n)\bigr) - \bigl(\Psi_n(\thestar) -
	\Psi(\thestar)\bigr)\bigr\rangle\\
	& \leq \lVert \hat\theta_n - \thestar \rVert \lVert \Psi_n (\thestar) - \Psi (\thestar)\rVert \\
	& \quad  + \lVert \hat\theta_n - \thestar\rVert \bigl\lVert \bigl(\Psi_n (\hat\theta_n) - \Psi(\hat\theta_n)\bigr) - \bigl(\Psi_n(\thestar) -
	\Psi(\thestar)\bigr) \bigr\rVert,
\end{align*}
which implies
\begin{equ}
	\label{eq:z-th-01}
	\lVert \hat\theta_n - \thestar \rVert \leq \frac{1}{\mu} \lVert \Psi_n (\thestar ) - \Psi
	(\thestar)\rVert + \frac{1}{\mu} \bigl\lVert \bigl(\Psi_n (\hat\theta_n) - \Psi(\hat\theta_n)\bigr) - \bigl(\Psi_n(\thestar) -
	\Psi(\thestar)\bigr) \bigr\rVert.
\end{equ}
By \cref{eq:z-B32} and applying \cref{lem:m-W} to $\Psi_n(\thestar) - \Psi(\thestar)$, we have
\begin{equ}
	\label{eq:phi0.bound}
	 \lVert \Psi_n(\thestar) - \Psi (\thestar) \rVert_p \leq C c_3 d^{{1/2}} n
	^{-1/2}	.
\end{equ}
Taking expectations on both sides of \cref{eq:z-th-01}, by \cref{eq:phi0.bound} and \cref{lem:z2-Psi},
we obtain
\begin{align*}
	\lVert \hat\theta_n - \thestar \rVert_p
	& \leq \frac{1}{\mu} \lVert \Psi_n (\thestar) - \Psi(\thestar)\rVert_p \\
	& \quad  + \frac{1}{\mu}
	\Bigl\lVert \sup_{\theta \in \Theta}\bigl\lVert \bigl(\Psi_n (\theta) - \Psi(\theta)\bigr) - \bigl(\Psi_n(\thestar) -
	\Psi(\thestar)\bigr) \bigr\rVert \Bigr\rVert_p^* \\
	& \leq C \mu^{-1} ( \ccc d^{{1/2}} n^{-1/2} + c_2 d n^{-1/2} D_{\Theta} ) \\
	& \leq C  (D_{\Theta} +1) d n^{-1/2}.
\end{align*}
This proves \cref{eq:z-that}.

Now we move to prove  \cref{eq:z-thati}. By \cref{eq:z-theta}, we
have
\begin{align*}
	\Psi(\hat\theta_n) - \Psi(\thestar) & = - \bigl( \Psi_n (\thestar) - \Psi(\thestar) \bigr) -
	\bigl(\Psi_n (\hat\theta_n) - \Psi(\hat\theta_n)\bigr) + \bigl(\Psi_n(\thestar) -
	\Psi(\thestar)\bigr),
	\intertext{and}
	\Psi(\hat\theta_n^{(i)}) - \Psi(\thestar) & = - \bigl( \Psi_n^{(i)} (\thestar) - \Psi(\thestar) \bigr)
	- \bigl(\Psi_n^{(i)} (\hat\theta_n^{(i)}) - \Psi(\hat\theta_n^{(i)})\bigr) +
	\bigl(\Psi_n^{(i)}(\thestar) -
	\Psi(\thestar)\bigr).
\end{align*}
On the event that $ \lVert \hat\theta_n - \thestar\rVert \leq \delta$ and $\lVert
\hat\theta_n^{(i)} - \thestar\rVert \leq \delta$, taking difference of the foregoing terms, and by
\cref{eq:z-B11} again, we have
\begin{align*}
	\MoveEqLeft \mu \lVert \hat\theta_n - \hat\theta_n^{(i)}\rVert \mathds{1} \bigl( \lVert \hat\theta_n -
		\thestar\rVert \leq \delta , \lVert \hat\theta_n^{(i)} - \thestar \rVert \leq \delta \bigr) \\
	& \leq
	\bigl\lVert \Psi (\hat\theta_n) - \Psi (\hat\theta_n^{(i)})\bigr\rVert \mathds{1} \bigl( \lVert \hat\theta_n -
		\thestar\rVert \leq \delta , \lVert \hat\theta_n^{(i)} - \thestar \rVert \leq \delta \bigr) \\
	& \leq \frac{1}{n} \lVert
	\xi_i - \xi_i'\rVert + 2 \sup_{\theta: \lVert \theta - \thestar\rVert \leq \delta} \bigl\lVert \bigl(\Psi_n (\theta) - \Psi(\theta)\bigr) - \bigl(\Psi_n(\thestar) -
	\Psi(\thestar)\bigr) \bigr\rVert.
\end{align*}
By \cref{eq:z-B32} and \cref{lem:z2-Psi},
\begin{align*}
	\MoveEqLeft \mu \IE \{\lVert \hat\theta_n - \hat\theta_n^{(i)}\rVert^p \mathds{1} \bigl( \lVert \hat\theta_n -
		\thestar\rVert \leq \delta , \lVert \hat\theta_n^{(i)} - \thestar \rVert \leq \delta \bigr)  \} \\
		& \leq C d^{{p/2}} n^{-p  } + C d^{{p}} n^{-p/2  } \delta^p ,
\end{align*}
and then we complete the proof of \cref{eq:z-thati}.

\subsection{Proof of  \tpstring{\cref{lem:z-th2}}{Lemma 57}}%

For \cref{eq:z-t2-02}, note that $\xi_i$ and  $\hat\theta_n^{(i)}$ are independent, and $\hat\theta_n$ has the same
distribution as  $\hat\theta_n^{(i)}$,
\begin{align*}
	\MoveEqLeft \IE \bigl\{\lVert \xi_i \rVert \lVert \hat\theta_n - \thestar\rVert^2 \mathds{1}
	\bigl(\hat\theta_n \in B_{\delta}, \hat\theta_n^{(i)} \in B_\delta^c \bigr)\bigr\}  \\
	& \leq \delta^2 \IE \bigl\{ \lVert \xi_i \rVert \mathds{1} \bigl(\hat\theta_n^{(i)} \in
	B_{\delta}^c\bigr)\bigr\}\\
	& \leq \delta^2 \IE \lVert \xi_i\rVert \IP \bigl( \lVert \hat\theta_n^{(i)} - \thestar \rVert > \delta \bigr) \\
	& \leq C d^{1/2 }\delta^2  \IP \bigl( \lVert \hat\theta_n - \thestar \rVert > \delta \bigr) \\
	& \leq C  d^{{1/2}} \delta^{-p + 2} \IE \bigl\{ \lVert \hat\theta_n - \thestar\rVert^p \bigr\} \\
	& \leq C (\D +1)^p
		  d^{p + {1/2}}
		\delta^{-p + 2} n^{-p/2} ,
\end{align*}
where we used \cref{eq:z-that} in the last inequality.

For \cref{eq:z-t2-03}, by the H\"older inequality, we have
\begin{align*}
	\begin{split}
		& \IE \bigl\{\lVert \xi_i \rVert \lVert \hat\theta_n^{(i)} - \thestar\rVert^2 \mathds{1}
		\bigl(\hat\theta_n \in B_{\delta}^c, \hat\theta_n^{(i)} \in B_\delta\bigr)\bigr\} \\
		& \leq \Bigl( \IE \bigl\{ \lVert \xi_i\rVert^{p/2} \lVert \hat\theta_n^{(i)} - \thestar\rVert^p
				\bigr\}\Bigr)^{2/p} \Bigl( \IP \bigl( \lVert \hat\theta_n - \thestar \rVert > \delta\bigr)\Bigr)^{(p-2)/p}\\
		& \leq \delta^{-p + 2}\Bigl( \IE \bigl\{ \lVert \xi_i\rVert^{p/2}\bigr\} \IE \{\lVert \hat\theta_n^{(i)} - \thestar\rVert^p\}
				\Bigr)^{2/p} \Bigl( \IE \bigl\{ \lVert \hat\theta_n - \thestar\rVert^p\bigr\} \Bigr)^{(p-2)/p}\\
		& \leq C (\D +1)^p
		  d^{p + {1/2}}
		\delta^{-p + 2} n^{-p/2}.
	\end{split}
\end{align*}
As for \cref{eq:z-t2-04}, recalling that $p \geq 3$, and by  \cref{eq:z-B32,eq:z-that,eq:z-thati} and the H\"older
inequality, we have
\begin{align*}
	\MoveEqLeft
	\IE \Bigl\{\lVert \xi_i\rVert (\lVert \hat\theta_n - \thestar\rVert + \lVert
			\hat\theta_n^{(i)} - \thestar\rVert ) \lVert \hat\theta_n - \hat\theta_n^{(i)}\rVert
		\mathds{1} \bigl( \hat\theta_n \in B_{\delta}, \hat\theta_n^{(i)} \in B_\delta\bigr)\Bigr\}\\
	& \leq \lVert \xi_i \rVert_p \lVert \hat\theta_n - \thestar \rVert_p \Bigl\lVert \lVert \hat\theta_n - \hat\theta_n^{(i)} \rVert \mathds{1} \bigl( \hat\theta_n \in B_{\delta}, \hat\theta_n^{(i)} \in B_\delta\bigr)\Bigr\} \Bigr\rVert_p \\
	& \quad  + \lVert \xi_i \rVert_p \lVert \hat\theta_n^{(i)} - \thestar \rVert_p \Bigl\lVert \lVert \hat\theta_n - \hat\theta_n^{(i)} \rVert \mathds{1} \bigl( \hat\theta_n \in B_{\delta}, \hat\theta_n^{(i)} \in B_\delta\bigr)\Bigr\} \Bigr\rVert_p \\
	& =  2 \lVert \xi_i \rVert_p \lVert \hat\theta_n - \thestar \rVert_p \Bigl\lVert \lVert \hat\theta_n - \hat\theta_n^{(i)} \rVert \mathds{1} \bigl( \hat\theta_n \in B_{\delta}, \hat\theta_n^{(i)} \in B_\delta\bigr)\Bigr\} \Bigr\rVert_p \\
	& \leq C ( \D + 1 ) \Bigl( d^2 n^{-3/2} + d^{5/2 } n^{-1  } \delta \Bigr) .
\end{align*}
This completes the proof.

\subsection{Proof of  \tpstring{\cref{lem:z2-theta}}{Lemma 59}}%

In this proof, we denote by $C$ a positive constant that depends only on  $c_1$, $c_4$, $c_5$, $\mu$, $\lambda_1$
and $\lambda_2$ and $C_p$ a constant that also depends on $p$, which might take different values in different places.
By (ii) of \cref{lem:z2-Psi} and following the proof of \cref{lem:z-thp}, we have \cref{eq:z-that,eq:z-thati} also hold for a positive
constant $C_p$.
This proves the first argument of this lemma.
Note that
\begin{math}
	\Psi_n(\thestar) - \Psi(\thestar) = n^{-1} \sum_{i = 1}^n \xi_i,
\end{math}
and by \cref{lem:m-W} and \cref{eq:z-B51}, 
\begin{equ}
	\label{eq:z2-psi.0}
	\lVert \Psi_n(\thestar) -
	\Psi(\thestar)\rVert_{\psi_1} \leq C  \cce d^{{1/2}} n^{-1/2},
\end{equ}
and for any $p \geq 1$,
\begin{align*}
	\lVert \Psi_n(\thestar) - \Psi(\thestar)\rVert_p \leq C_p \cce d^{{1/2}} n^{-1/2}.
\end{align*}
For \cref{eq:z2-tail.prob}, by the fact that
\begin{math}
	\IP ( \lVert Y \rVert > t) \leq 2 e^{-t/\lVert Y\rVert_{\psi_1}},
\end{math}
it follows from \cref{eq:z2-psi.0,lem:z2-Psi} that
\begin{align*}
	\IP \bigl( \lVert \Psi_n(\thestar) - \Psi(\thestar)\rVert > \mu t / 2 \bigr)
	& \leq 2 \exp \Bigl( - C'' \frac{ \sqrt{n}\mu t}{ \cce}\Bigr),
\end{align*}
and
\begin{align*}
	\IP^* \biggl( \sup_{\theta \in \Theta} \lVert \Psi_n(\theta) - \Psi(\theta) - \Psi_n(\thestar) + \Psi(\thestar) \rVert > \mu t / 2 \biggr)
		& \leq 2 \exp \Bigl( -  \frac{C'' \sqrt{n} t}{ \ccd d^{{3/2}} (D_{\Theta} +1 )}\Bigr),
\end{align*}
where \(C'' >0\) is an absolute constant.
By \cref{eq:z-th-01},  for any $t > 0$,
\begin{align*}
	\begin{split}
		\IP \bigl( \lVert \hat\theta_n - \thestar\rVert > t \bigr)
		& \leq \IP \bigl( \lVert \Psi_n(\thestar) - \Psi(\thestar)\rVert > \mu t / 2 \bigr) \\
		& \quad  + \IP^* \biggl( \sup_{\theta \in \Theta} \lVert \Psi_n(\theta) - \Psi(\theta) - \Psi_n(\thestar) + \Psi(\thestar) \rVert > \mu t / 2 \biggr) \\
		& \leq 2 \exp \Bigl( - \frac{C'' \sqrt{n} \mu t}{ \ccd d^{{3/2}} (D_{\Theta} +1 ) + \cce d^{{1/2}}} \Bigr).
	\end{split}
\end{align*}
Taking $C' = { C'' \mu} / ( c_4 + c_5)$, we complete the proof
of \cref{eq:z2-tail.prob}.
By \cref{lem:z2-theta} and the H\"older inequality, we have
\begin{align*}
	\IE \Bigl\{\lVert \xi_i \rVert \lVert \hat\theta_n - \thestar\rVert^2 \mathds{1}
		\bigl(\hat\theta_n \in B_{\delta}, \hat\theta_n^{(i)} \in B_\delta^c\bigr)\Bigr\}
	& \leq \bigl( \IE \lVert \xi_i\rVert^4\bigr)^{1/4} \Bigl(\IE \lVert \hat\theta_n -
	\thestar\rVert^4\Bigr)^{1/2} \Bigl(\IP \bigl(\hat\theta_n^{(i)} \in B_{\delta}^c \bigr)\Bigr)^{1/4} \\
	& \leq C
	(D_{\Theta} +1 )^2 d^{{5/2}}
	n^{-1} \exp \Bigl( - \frac{C' n^{1/2} \delta}{4 (D_{\Theta} +1 ) d^{3/2}} \Bigr),
\end{align*}
	where we used \cref{eq:z-B51,eq:z-that,eq:z2-tail.prob} in the last line. This proves
\cref{eq:z2-t2-02}. The inequality \cref{eq:z2-t2-03} can be derived using a similar argument.

\subsection{Proof of  \tpstring{\cref{lem:b1}}{Lemma b1}}%

We use a recursion inequality to prove the bound. By \cref{eq:saa}, we
have
\begin{equ}
	\label{eq:sgd-5.15-a}
	\lVert \theta_n - \thestar \rVert^2 &= \lVert \theta_{n-1} - \thestar \rVert^2 - 2 \ell_n
	\bigl\langle \theta_{n - 1} - \thestar , \nabla f(\theta_{n - 1}) + \zeta_n \bigr\rangle
	\\
	& \quad  + \ell_n^2 \lVert \nabla f(\theta_{n - 1}) + \zeta_n \rVert^2 .
\end{equ}
Recall that by \cref{con:c1}, $\zeta_n = \xi_n + \eta_n$ where  $\xi_n$ is independent of
$\theta_{n - 1}$,
$\eta_n$ is  $\mathcal{F}_n$ measurable, $\lVert \eta_n \rVert \leq c_1 \lVert
\theta_{n - 1} - \thestar\rVert$, $\IE \{ \eta_n \vert
	\mathcal{F}_{n - 1} \} = 0$ and \(\theta_{n - 1} \in \mathcal{F}_{n - 1 }\). Therefore, $\IE \left\{ \bigl\langle \nabla
f(\theta_{n - 1}), \zeta_n\bigr\rangle \bigm\vert \mathcal{F}_{n - 1}\right\} = 0$. Moreover,
with $L_2 := c_1 + L,$
\begin{equ}
	\label{eq:sgd-5.15-b}
	& \IE \bigl\{ \bigl\lVert \nabla f(\theta_{n - 1}) + \zeta_n\bigr\rVert^2 \bigm\vert
	\mathcal{F}_{n - 1} \bigr\} \\
	&= \IE \bigl\{ \bigl\lVert \nabla f(\theta_{n - 1}) \bigr\rVert^2 \bigm\vert
	\mathcal{F}_{n - 1} \bigr\} + \IE \bigl\{ \bigl\lVert\zeta_n\bigr\rVert^2 \bigm\vert
	\mathcal{F}_{n - 1} \bigr\} \\
	&= \IE \bigl\{ \bigl\lVert \nabla f(\theta_{n - 1}) - \nabla f(\thestar)\bigr\rVert^2 \bigm\vert
	\mathcal{F}_{n - 1} \bigr\} + \IE \bigl\{ \bigl\lVert\xi_n + \eta_n\bigr\rVert^2 \bigm\vert
	\mathcal{F}_{n - 1} \bigr\}\\
	& \leq 2 L_2^2 \|\theta_{n - 1} - \thestar\|^2 + 2 \IE \|\xi_n\|^2.
\end{equ}
For the intersection term of \cref{eq:sgd-5.15-a}, under the strong convexity assumption
\cref{eq:str_con},
\begin{align*}
	\bigl\langle \nabla f (\theta_1 ) - \nabla f(\theta_2), \theta_1 - \theta_2\bigr\rangle
	\geq \mu \lVert \theta_1 - \theta_2\rVert^2,
\end{align*}
and noting that $\IE \clc{ \inprod{\theta_{n - 1} - \thestar, \zeta_n} \vert
\mathcal{F}_{n-1} } = 0$,
\begin{equ}
	\label{eq:sgd-5.15-c}
	\MoveEqLeft \IE \bigl\{ \bigl\langle \theta_{n - 1} - \thestar , \nabla f(\theta_{n -1}) + \zeta_n
	\bigr\rangle \bigm\vert \mathcal{F}_{n -1} \bigr\}\\
	&= \IE \bigl\langle \theta_{n - 1} - \thestar , \nabla f(\theta_{n -1}) - \nabla
	f(\thestar )\bigr\rangle  \\
	& \geq \mu \|\theta_{n - 1} - \thestar \|^2.
\end{equ}
Combining \cref{eq:sgd-5.15-a,eq:sgd-5.15-b,eq:sgd-5.15-c},
\begin{align*}
	\IE \bigl\{ \lVert \theta_n - \thestar \rVert^2 \bigm\vert \mathcal{F}_{n - 1}\bigr\}
	& \leq \bigl(1 - 2 \mu \ell_n + 2L_2^2\ell_n^2 \bigr)
	\lVert \theta_{n-1} - \thestar \rVert^2  + 2
	\ell_n^2 \IE \|\xi_n\|^2.
\end{align*}
Taking expectations on both sides yields
\begin{equ}
	\label{eq:sgd-5.15-d}
	\delta_{n} \leq \bigl(1 - 2 \mu \ell_n + 2L_2\ell_n^2 \bigr)
	\delta_{n-1}  + 2
	\ell_n^2 \IE \|\xi_n\|^2.
\end{equ}
Observe that   $\mu \leq L_2$ and thus  $2 \mu \ell_n - 2L_2^2
\ell_n^2 \leq 2 L_2 \ell_n - 2 L_2^2 \ell_n^2 \leq 1/2$. This ensures that the
coefficient in front of $\delta_{n-1}$ is always positive.
By \cref{con:c0}, we have
\begin{equation}
	\delta_0 = \IE \lVert \theta_0 - \thestar\rVert^2 \leq \tau_0^2.
	\label{eq:d0}
\end{equation}
By \cref{eq:sgd-5.15-d,eq:d0} and applying the recursion  $n$
times, recalling that $\max_{1 \leq i \leq n } \IE \|\xi_i\|^2 \leq \tau^2$,  we have
\begin{align*}
	\delta_{n}
	& \leq \prod_{k = 1}^{n} \bigl(1 - 2 \mu \ell_k + 2L_2^2
	\ell_k^2 \bigr) \delta_0  + 2 \tau^2 \sum_{k = 1}^{n} \prod_{i = k + 1}^n \bigl(1 - 2 \mu
	\ell_i + 2L_2^2
	\ell_i^2 \bigr) \ell_k^2\\
	& \leq \prod_{k = 1}^{n} \bigl(1 - 2 \mu \ell_k + 2L_2^2
	\ell_k^2 \bigr) \tau_0^2  + 2 \tau^2 \sum_{k = 1}^{n} \prod_{i = k + 1}^n \bigl(1 - 2 \mu
	\ell_i + 2L_2^2
	\ell_i^2 \bigr) \ell_k^2.
\end{align*}
Following the proof of Bach and Moulines (2011, Eqs. (18), (23) and (24)), we have
\begin{align}
	\delta_n & \leq \, \label{eq:sgd-d1}
		\Bigl( \tau_0^2 + \frac{\tau^2}{L_2^2} \Bigr) \exp \biggl(-  \mu \sum_{k = 1}^n
			\ell_k + 4L_2^2
		\sum_{k = 1}^n \ell_k^2 \biggr)
		+ 2 \tau^2 \sum^{n}_{k = 1} \prod_{i = k + 1}^n (1 - \mu \ell_i)
		\ell_k^2  \\
	& \leq \label{eq:sgd-d2}
	\begin{multlined}[t]
		\Bigl( \tau_0^2 + \frac{\tau^2}{L_2^2} \Bigr) \exp \biggl(-  \mu \sum_{k = 1}^n
			\ell_k + 4 L_2^2
		\sum_{k = 1}^n \ell_k^2 \biggr) \\
		+ 2 \tau^2 \biggl\{ \exp \biggl(-\mu \sum_{i = n/2 + 1}^n \ell_i \biggr)
		\sum_{k = 1}^n \ell_k^2 + \frac{\ell_{n/2}}{\mu}\biggr\}.
	\end{multlined}
\end{align}
We next consider the cases where $\alpha \in (1/2, 1)$ and $\alpha = 1$, separately.
If $\alpha \in (1/2,1)$, by \cref{eq:d0,eq:sgd-d2}, we have
\begin{align*}
	\sum_{k = 1}^{n} \ell_k^2 \leq \ell_0^2 \sum_{k = 1}^{\infty} k^{-2 \alpha} \leq C,
\end{align*}
and
\begin{align*}
	\delta_n & \leq C \exp \bigl( - \frac{\mu
			\ell_0}{4} n^{1 - \alpha}\bigr) \Bigl(\tau_0^2 + \frac{\tau^2}{L_2^2}\Bigr) + \frac{4 \ell_0 \tau^2}{\mu n^{\alpha}}
	 \leq C n^{-\alpha} (\tau^2 + \tau_0^2).
\end{align*}
This proves \cref{eq:lb1-1}.
Now we move to prove \cref{eq:lb1-2}.
For $\alpha = 1$, i.e.,  $\ell_i = \ell_0 i^{-1}$, we use the following basic
inequalities:
\begin{align*}
	\log n  \leq \sum_{k = 1}^{n} k^{-1} \leq \log n  + 1,
	\quad
	\sum_{k = 1}^{\infty} k^{-2} \leq 2.
\end{align*}
For the first term of \cref{eq:sgd-d1}, we have
\begin{align*}
	\exp \biggl(-  \mu \sum_{k = 1}^n \ell_k + 4 L_2^2
		\sum_{k = 1}^n \ell_k^2 \biggr)
	& \leq \exp \bigl(8 \ell_0^2 L_2^2 \bigr) \exp \bigl(-\mu \ell_0 \log n
	\bigr) \leq C n^{-\mu \ell_0},
\end{align*}
and for the second term of \cref{eq:sgd-d1}, we obtain
\begin{align*}
	 \sum^{n}_{k = 1} \prod_{i = k + 1}^n (1 - \mu \ell_i) \ell_k^2
	& \leq \ell_0^2 \sum_{k = 1}^n k^{-2} \exp \Bigl\{ - \mu \ell_0 \sum_{i = k + 1}^n
	i^{-1}\Bigr\} \\
	& \leq \ell_0^2 e^{\mu \ell_0} \sum_{k = 1}^n k^{-2} \exp \bigl\{ - \mu \ell_0 \log n + \mu
	\ell_0 \log k\bigr\} \\
	& \leq \ell_0^2 e^{\mu \ell_0} n^{-\mu \ell_0} \sum_{k = 1}^n k^{-2 + \mu \ell_0} \\
	& \leq
	\begin{cases}
		C n^{-1}, & \mu \ell_0 > 1, \\
		C n^{-1}\log n, & \mu \ell_0 = 1, \\
		C n^{-\mu \ell_0}, & 0 < \mu \ell_0 < 1.
	\end{cases}
\end{align*}
Therefore,
for $\alpha = 1$,
\begin{align*}
	\delta_n \leq
	\begin{cases}
		C n^{-1} (\tau^2 + \tau_0^2), & \mu \ell_0 > 1, \\
		C n^{-1}(\log n )(\tau^2 + \tau_0^2), & \mu \ell_0 = 1, \\
		C n^{-\mu \ell_0}(\tau^2 + \tau_0^2), & 0 < \mu \ell_0 < 1.
	\end{cases}
\end{align*}
This proves \cref{eq:lb1-2}.

\subsection{Proof of  \tpstring{\cref{lem:lb2}}{Lemma lb2}}%

By the construction of $(\thetaihat_j)_{1 \leq j \leq n}$,
\begin{align*}
	\theta_j - \thetaihat_j =
	\begin{cases}
		0, & j < i; \\
		-\ell_j (\xi_j - \xi_j' + \eta_j - \eta^{(i)}_j ), & j = i; \\
		\bigl(\theta_{j - 1} - \thetaihat_{j - 1}\bigr) - \ell_j \bigl(\nabla f(\theta_{j - 1}) - \nabla
		f(\thetaihat_{j - 1}) + \eta_j - \eta^{(i)}_j\bigr), & j > i.
	\end{cases}
\end{align*}
Since $\xi_i \stackrel{d}{=} \xi_i'$, $\eta_i \stackrel{d}{=} \eta_i^{(i)}$ and $\lVert \eta_i
\rVert \leq c_2 \lVert \theta_{i - 1} - \thestar\rVert$, it follows from
\cref{con:c1} and \cref{lem:b1} that
\begin{equ}
	\label{eq:thetaii}
	\IE \lVert \theta_i - \thetaihat_i \rVert^2 & \leq C \ell_i^2 (\IE \lVert \xi_i \rVert^2 + \IE
	\lVert \eta_i\rVert^2 ) \\
												& \leq C i^{-2\alpha} (\tau^2 + c_2 \IE \lVert \theta_{i - 1} -
	\thestar\rVert^2  ) \\
												& \leq C i^{-2 \alpha} (\tau^2 + \tau_0^2).
\end{equ}
For $j > i$, using a similar argument leading to \cref{eq:sgd-5.15-d},
\begin{align*}
	\begin{split}
		& \IE \lVert \theta_j - \thetaihat_j \rVert^2 \\
		& = \IE \lVert \theta_{j - 1} - \thetaihat_{j - 1} \rVert^2 - 2 \ell_j \IE \Bigl\{ \theta_{j - 1} - \thetaihat_{j - 1}, \nabla f(\theta_{j - 1}) - \nabla f(\thetaihat_{j - 1}) \Bigr\} \\
		& \quad  + \ell_j^2 \IE \lVert  \nabla f(\theta_{j - 1}) - \nabla f(\thetaihat_{j - 1}) +\eta_j - \eta_j^{(i) } \rVert^2 \\
		& \leq \bigl(1 - 2 \mu \ell_j + 2L_2^2 \ell_j^2  \bigr)\IE \lVert \theta_{j -
		1} - \thetaihat_{j - 1}\rVert^2 ,
	\end{split}
\end{align*}
Solving the recursive system, and by \cref{eq:thetaii}, we have
\begin{align*}
	\IE \lVert \theta_j - \thetaihat_j\rVert^2
	& \leq \exp \biggl\{ - 2 \mu \sum_{k = i + 1}^j \ell_k + 2 L_2^2 \sum_{k =
	1}^{j} \ell_k^2 \biggr\} \IE \lVert \theta_i - \thetaihat_{i}\rVert^2 \\
	& \leq C (\tau^2 + \tau_0^2) i^{-2\alpha} \exp \biggl\{ - 2 \mu \sum_{k = i + 1}^j \ell_k  \biggr\}.
\end{align*}

For $0 < \alpha < 1$, using a similar argument as in the proof of \cref{thm:sgd}, we have
for $j \geq i$,
\begin{align*}
	\IE \lVert \theta_j - \thetaihat_j\rVert^2
	& \leq C (\tau^2 + \tau_0^2)i^{-2\alpha} \exp \Bigl\{  - \mu
	\bigl(\phi_{1 - \alpha}(j) - \phi_{1 - \alpha} (i) \bigr)\Bigr\} \nonumber
	\\
	& \leq
	\begin{cases}
		C (\tau^2 + \tau_0^2) i^{-2 \alpha} ,
		& i \leq j \leq 2 i, \\
		C (\tau^2 + \tau_0^2) i^{-2\alpha} \exp \Bigl\{ - \mu \bigl(\phi_{1 - \alpha} (j) - \phi_{1 - \alpha}
		(\frac{j}{2})\bigr) \Bigr\} , & j > 2 i,
	\end{cases}\nonumber \\
	& \leq
	\begin{cases}
		\mathrlap{C(\tau^2 + \tau_0^2) (j/2)^{-2\alpha},}
		\hphantom{C (\tau^2 + \tau_0^2) i^{-2\alpha} \exp \Bigl\{ - \mu \bigl(\phi_{1 - \alpha} (j) - \phi_{1 - \alpha}
		(\frac{j}{2})\bigr) \Bigr\} ,}
		&  i \leq j \leq 2 i, \\
		C(\tau^2 + \tau_0^2) i^{-2\alpha} \exp \Bigl\{- \frac{\mu}{2} \phi_{1 - \alpha}(j)\Bigr\}, & j > 2 i
	\end{cases}\nonumber \\
   & \leq
   C(\tau^2 + \tau_0^2) j^{-2\alpha}.
\end{align*}
If $\alpha = 1$, \cref{eq:lb2-03} can be shown similarly.
\subsection{Proof of  \tpstring{\cref{lem:lb3}}{Lemma lb3}}%

Recall that for any $j \geq 1$,
\begin{align*}
	\theta_j - \thestar = (\theta_{j - 1} - \thestar) - \ell_j \bigl(\nabla f(\theta_{j - 1}) +
	\zeta_j\bigr),
\end{align*}
where $\zeta_j$ is a martingale difference such that
\begin{math}
	\IE \clc{ \zeta_j \vert \mathcal{F}_{j - 1}} = 0,
\end{math}
and  $\theta_{j - 1}$ is  $\mathcal{F}_{j - 1}$-measurable.
Hence,
\begin{align*}
	\lVert \theta_{j} - \thestar\rVert^4
	&= \lVert \theta_{j - 1} - \thestar\rVert^4 + 4 \ell_j^2 \bangle{ \theta_{j - 1} - \thestar ,
	\nabla f(\theta_{j - 1}) + \zeta_j }^2 + \ell_j^4 \lVert \nabla f(\theta_{j - 1}) + \zeta_{j} \rVert^4
	\\
	& \quad  - 4 \ell_{j} \lVert \theta_{j - 1} - \thestar \rVert^2 \bangle{ \theta_{j - 1} -
	\thestar , \nabla f(\theta_{j - 1}) + \zeta_j} \\
	& \quad  + 2 \ell_j^2 \lVert \theta_{j - 1} - \thestar \rVert^2
	\lVert \nabla f(\theta_{j - 1}) + \zeta_j  \rVert^2 \\
	& \quad  - 4 \ell_j^3 \bangle{ \theta_{j - 1} - \thestar , \nabla f(\theta_{j - 1}) + \zeta_j} \lVert
	\nabla f(\theta_{j - 1}) + \zeta_j  \rVert^2 \\
	& \leq \lVert \theta_{j - 1} - \thestar\rVert + 6 \ell_j^2 \lVert \theta_{j - 1} - \thestar \rVert^2
	\lVert \nabla f(\theta_{j - 1}) + \zeta_j  \rVert^2 \\
	& \quad  + \ell_j^4 \lVert \nabla f(\theta_{j - 1})
	+ \zeta_j\rVert^4 + 4 \ell_j^{3} \lVert \theta_{j - 1}
	 - \thestar \rVert \lVert \nabla f(\theta_{j -
	1}) + \zeta_j  \rVert^3 \\
	& \quad  - 4 \ell_j \lVert \theta_{j - 1} - \thestar \rVert^2 \bangle{
\theta_{j - 1} - \thestar , \nabla f(\theta_{j - 1}) + \zeta_j } ,
\end{align*}
and then
\begin{equ}
	\label{eq:4-01}
	& \IE \bclc{ \lVert \theta_{j} - \thestar \rVert^4 \vert \mathcal{F}_{j - 1} } \\
	& \leq \lVert \theta_{j - 1} - \thestar \rVert^4 + 6 \ell_j^2 \lVert \theta_{j - 1} - \thestar
	\rVert^2 \IE \bclc{ \lVert \nabla f(\theta_{j - 1}) + \zeta_j  \rVert^2 \bigm\vert
	\mathcal{F}_{j - 1}} \\
	& \quad  + \ell_j^4 \IE \bclc{ \lVert \nabla f(\theta_{j - 1}) + \zeta_j  \rVert^4 \bigm\vert
	\mathcal{F}_{j - 1}} + 4 \ell_j^3 \lVert \theta_{j - 1} - \thestar \rVert \IE\bclc{ \|\nabla
f(\theta_{j - 1}) + \zeta_j \|^3 \bigm\vert \mathcal{F}_{j - 1}  }\\
	& \quad  - 4 \ell_j \lVert \theta_{j - 1} - \thestar \rVert^2 \bangle{ \theta_{j - 1} -
	\thestar , \nabla f(\theta_{j - 1}) - \nabla f(\thestar) }  .
\end{equ}
Note that by \cref{eq:str_con},
\begin{equ}
	\label{eq:4-02}
	\bangle{ \theta_{j - 1} - \thestar , \nabla f(\theta_{j - 1}) - \nabla f(\thestar) } \geq \mu
	\lVert \theta_{j - 1} - \thestar \rVert^2.
\end{equ}
For $1 \leq p \leq 4$, recall that by \cref{con:c1}, $\IE \clc{ \lVert \xi_j\rVert^p } \leq \tau^p$ and \( \lVert \eta_j \rVert \leq c_1 \lVert \theta_{j - 1} - \thestar\rVert\), and it follows that
\begin{equ}
	\label{eq:4-03}
	\IE \bclc{ \lVert \nabla f(\theta_{j - 1}) + \zeta_j  \rVert^p \bigm\vert \mathcal{F}_{j - 1}}
	& \leq \IE \bclc{ \lVert \nabla f(\theta_{j - 1}) - \nabla f(\thestar)  + \xi_j + \eta_j  \rVert^p \bigm\vert \mathcal{F}_{j - 1}} \\
	& \leq 2^{p - 1} \Bigl(L_2^p \lVert \theta_{j - 1} - \thestar \rVert^p + \tau^p
	\Bigr).
\end{equ}
By \cref{eq:4-01,eq:4-02,eq:4-03}, we have
\begin{align*}
	\IE \bclc{ \lVert \theta_{j} - \thestar \rVert^4 \bigm\vert \mathcal{F}_{j - 1} }
	& \leq \lVert \theta_{j - 1} - \thestar \rVert^4 \bigl(1 - 4 \mu \ell_j + 12 \ell_j^2
	L_2^2 + 16 \ell_j^3 L_2^3 + 8 \ell_j^4 L_2^4 \bigr)
	\\
	& \quad  + 20 \lVert \theta_{j -1} - \thestar \rVert^2 \ell_j^2 \tau^2 + 16 \ell_j^4 \tau^4
	\\
	& \leq \lVert \theta_{j - 1} - \thestar \rVert^4 \bigl(1 - 4 \mu \ell_j + 16 \ell_j^2
	L_2^2 + 24 \ell_j^4 L_2^4 \bigr)
	\\
	& \quad  + 20 \lVert \theta_{j -1} - \thestar \rVert^2 \ell_j^2 \tau^2 + 16 \ell_j^4 \tau^4.
\end{align*}
Taking expectations on both sides, we have
\begin{align*}
	\IE \lVert \theta_{j} - \thestar \rVert^4 &  \leq  \IE \lVert \theta_{j - 1} - \thestar \rVert^4 \bigl(1 - 4 \mu \ell_j + 16 \ell_j^2
	L_2^2 + 24 \ell_j^4 L_2^4 \bigr) \\
	& \quad  + C (\tau^2 + \tau_0^2)\tau^2 j^{-3\alpha } + 16 \tau^4 j^{-4\alpha},
\end{align*}
where we used \cref{lem:b1} in the last inequality.
Using the similar arguments leading to \cref{lem:b1} (see also Bach and Moulines (2011, pp.
16--19)),  we have for $\alpha \in (0,1)$,
\begin{align*}
	\IE \lVert \theta_{j } - \thestar \rVert^4 \leq C j^{-2 \alpha} (\tau^4 + \tau_0^4).
\end{align*}
If $\alpha = 1$, we have
\begin{align*}
	\IE \lVert \theta_j - \thestar\rVert^4
	& \leq C j^{-2\mu \ell_0} \bigl( \phi_{2 \mu \ell_0 - 2}(j) + 1 \bigr) (\tau^4 + \tau_0^4),
\end{align*}
where $\phi$  is as defined in \cref{eq:phi.function}.  This proves \cref{eq:lb3-02}, and hence the lemma.

\section*{Acknowledgments}
Part of this work was based on the second 
author's Ph.D. thesis at the Chinese University of Hong Kong, and part was done while the second author was  visiting 
the Southern University of Science and Technology in 2019.

\setlength{\bibsep}{0.5ex}
\def\bibfont{\small}

\end{document}